\documentclass[brochure,english,12pt]{smfbourbaki}

\usepackage[T1]{fontenc}
\usepackage{lmodern,amssymb,bm, mathtools}

\usepackage{babel}

\usepackage[utf8]{inputenc}

\usepackage[colorlinks=true, linkcolor=blue, citecolor=red, urlcolor=blue]{hyperref}
\makeatletter
\def\SL@eqntext#1{\rlap{\quad\SL@margintext{#1}}}
\makeatother

\usepackage{enumerate,xspace}
\usepackage{rotating}

\usepackage[stretch=10,shrink=10,step=2,kerning=true,protrusion=true,expansion=true,final]{microtype} 

\usepackage[
backend=biber,
style=authoryear, 
citestyle=authoryear-comp,
maxnames=7,
sortcites=false 
]{biblatex}
\usepackage{csquotes}

\newtoggle{tnbcbx@howcited}
\DeclareEntryOption[boolean]{howcited}[true]{\settoggle{tnbcbx@howcited}{#1}}
\DeclareBibliographyOption[boolean]{howcited}[true]{\settoggle{tnbcbx@howcited}{#1}}
\DeclareTypeOption[boolean]{howcited}[true]{\settoggle{tnbcbx@howcited}{#1}}

\newbibmacro{howcited}{%
  \iftoggle{tnbcbx@howcited}
    {\iffieldundef{shorthand}
       {}
       {\setunit{\addspace}%
        \printtext[parens]{%
          \bibstring{citedas}%
          \setunit{\addcolon\space}%
          \printfield{shorthand}%
          \setunit{\addslash}%
          }}}
    {}}

\renewbibmacro{finentry}{\usebibmacro{howcited}\finentry}

\DefineBibliographyStrings{french}{citedas    = {cité ci-dessus avec l'acronyme}}
\DefineBibliographyStrings{english}{citedas    = {heretofore cited as}}

\DefineBibliographyExtras{french}{\restorecommand\mkbibnamefamily}

\DeclareDelimFormat{nameyeardelim}{\addcomma\space}

\DeclareNameAlias{sortname}{given-family}

\renewcommand{\bibnamedash}{\leavevmode\raise3pt\hbox to3em{\hrulefill}\space}

\AtEveryBibitem{%
  \clearfield{issn} 
  \clearfield{isbn} 
  \clearfield{doi} 
  \clearlist{language} 
  \ifentrytype{online}{}{
  \ifentrytype{unpublished}{}{
    \clearfield{url}
  }
  }
}

\renewbibmacro{in:}{%
    \ifentrytype{article}{}{\printtext{\bibstring{in}\intitlepunct}}}

\DeclareFieldFormat[article,periodical,inreference]{number}{\mkbibparens{#1}}
\DeclareFieldFormat[article,periodical,inreference]{volume}{\mkbibbold{#1}}
\renewbibmacro*{volume+number+eid}{%
    \printfield{volume}%
    \setunit*{\addthinspace}
    \printfield{number}%
    \setunit{\addcomma\space}%
    \printfield{eid}}

\DeclareFieldFormat[article,inbook,incollection]{title}{\enquote{#1}\addcomma} 

\date{Novembre 2025}
\bbkannee{78\textsuperscript{e} année, 2025--2026}  
\bbknumero{1245}                                      

\title{Model theory, differential algebra and functional transcendence}
\subtitle{after Freitag, Jaoui, and Moosa}

\author{Amador Martin-Pizarro}
\address{Dept. of Mathematical Logic\\ Mathematical Institute\\ University of Freiburg\\ Ernst-Zermelo-Str. 1\\ D-79104 Freiburg \\ GERMANY}
\email{pizarro@math.uni-freiburg.de}

\theoremstyle{definition}

\newtheorem*{notation}{Notation}

\theoremstyle{plain}
\newtheorem*{claimstar}{Claim}
\newenvironment{claimstarproof}{\noindent\textit{Proof of
		Claim.}}{\hfill\qedsymbol \tiny{Claim}
	\medskip}

\newcounter{claimCount}
\AtBeginEnvironment{proof}{\setcounter{claimCount}{0}}

\newtheorem*{TheoIntro}{Theorem}
\newtheorem*{Notation}{Notation}

\def\Ind#1#2{#1\setbox0=\hbox{$#1x$}\kern\wd0\hbox to
	0pt{\hss$#1\mid$\hss} \lower.9\ht0\hbox to
	0pt{\hss$#1\smile$\hss}\kern\wd0}
\def\Notind#1#2{#1\setbox0=\hbox{$#1x$}\kern\wd0\hbox to
	0pt{\mathchardef\nn="0236\hss$#1\nn$\kern1.4\wd0\hss}\hbox to
	0pt{\hss$#1\mid$\hss}\lower.9\ht0 \hbox to
	0pt{\hss$#1\smile$\hss}\kern\wd0}
\def\ind{\mathop{\mathpalette\Ind{}}}
\def\nind{\mathop{\mathpalette\Notind{}}}

\newcommand{\nc}{\newcommand}

\nc{\Z}{\mathbb{Z}}
\nc{\Q}{\mathbb{Q}}
\nc{\N}{\mathbb{N}}
\nc{\F}{\mathbb{F}}
\nc{\UU}{\mathbb{U}}
\nc{\C}{\mathbb{C}}
\nc{\CC}{\mathcal{C}}

\nc{\M}{\mathcal{M}}
\nc\LL{\mathcal L}

\newlength{\wpwidth}
\nc{\primep}{\raisebox{\depth}{\makebox[\wpwidth]{\rotatebox{10}{$\wp$}}}}

\nc\ord{\operatorname{ord}}
\nc{\dcl}{\operatorname{dcl}}
\nc{\acl}{\operatorname{acl}}
\nc{\tr}{\operatorname{trdeg}}
\nc\U{\operatorname{U}}
\nc\RD{\operatorname{RD}}
\nc\Bind{\operatorname{Bind}}
\nc\Aut{\operatorname{Aut}}
\nc{\Stab}{\operatorname{Stab}}
\nc{\ZZ}{\operatorname{Z}}

\nc{\nf}[1]{_{\mid {#1}}}
\nc{\restr}[1]{\xspace_{\upharpoonright {#1}}}

\nc{\sbgp}[1]{\langle\xspace {#1}\xspace\rangle}
\nc{\str}[2]{{#1}\xspace\langle\xspace {#2}\xspace\rangle}

\nc\inv{ ^{-1}}

\nc{\tp}{\operatorname{tp}}
\nc\cb{\operatorname{Cb}}

\nc{\cf}{\text{cf.\,}}
\nc{\eg}{\text{e.g. }}

\begin{document}

\maketitle
\tableofcontents

\section*{Introduction}

The moduli space of elliptic curves over the complex numbers can be identified with the affine line $\mathbb A^1$ using the \emph{j-invariant} function, an analytic map defined on the complex upper half plane $\mathbb H$. Indeed, an elliptic curve over the complex numbers is isomorphic to a complex torus $\C/\Lambda$, and the lattice $\Lambda\subset \C$ can be transformed to $\Lambda_\tau=\Z+\tau \Z$ for some  $\tau $ in $\mathbb H$. Using this transformation, the Weierstrass form of the elliptic curve is  \[y^2= 4x^3 + g_2(\Lambda) x+ g_3(\Lambda), \text{ with }g_2(\Lambda)=60\sum_{0\ne z\in \Lambda} z^{-4} \text{ and } g_3(\Lambda)=140\sum_{0\ne z\in \Lambda} z^{-6},\] so the $j$-invariant of $\C/\Lambda_\tau$, which we identify with the element $\tau$ of $\mathbb H$, is 
\[ j(\tau)=1728 \frac{g_2(\Lambda_\tau)^3}{g_2(\Lambda_\tau)^3-27g_3(\Lambda_\tau)^2}.\]
Using the \emph{counting theorem} of  \textcite[Theorem 1.10]{PilaWilkie}, a result in \emph{o-minimality}, an area which can be seen as a model theoretic generalisation of semi-algebraic geometry, \textcite[Theorem 1.1]{Pila} showed a generalization of the Ax--Lindemann--Weierstrass theorem: Given an irreducible variety $W$ defined over $\C$ and rational functions $a_1,\ldots, a_n$ on $W$ which locally at the point $P$ take values in $\mathbb H$, whenever  no $j(a_i)$ is a constant function and there is no \emph{modular relation}  for $i\ne k$ of the form  \[ a_k= \frac{a a_i + b}{c a_i +d} \text{ for some }\begin{pmatrix}a & b\\ c& d\end{pmatrix} \in \mathrm{GL}_2(\Q) \text{ with } \mathrm{det}\begin{pmatrix}a & b\\ c& d\end{pmatrix}>0,\]
 then  \[j(a_1), \ldots, j(a_n), j'(a_1), \ldots, j'(a_n), j''(a_1), \ldots, j''(a_n)\] are algebraically independent over the function field $\C(W)$, where each $j(a_i)$ is seen as a function on $W$ locally near P, and the derivative is taken with respect to the element $a_i(P)$ of $\mathbb H$.  The matrix of a modular relation as above can be rescaled  to an element of the double $\mathrm{SL}_2(\Z)$-coset of the matrix $(\begin{smallmatrix}1 & 0\\ 0& N\end{smallmatrix})$ for some $N\ge 1$ in $\N$. In doing so, the modular relation 
 translates into a \emph{Hecke correspondence} between the values $j(a_i)$ and $j(a_k)$ given by a \emph{modular polynomial} $F_N(X, Y)$ which has coefficients in $\Z$ and is monic both in $X$ and $Y$. 

Pila's modular Ax--Lindemann--Weierstrass theorem is optimal in the sense that no further derivative can be algebraically independent. Indeed, the $j$-invariant function satisfies the following order $3$ irreducible differential algebraic (rational) equation $\chi(T)=0$ with \[ \chi(T)=S(T)+\frac{T^2-1968T+2 654 208}{2T^2(T-1728)^2},\] where  $S(T)$ is the \emph{Schwarzian derivative} \[ S(T)=\bigl(\frac{T''}{T'}\bigr)'-\frac{1}{2} \bigl(\frac{T''}{T'}\bigr)^2.\] 

Hence, it makes  sense to study the solution set of the above differential algebraic equation within any differential field of characteristic $0$, and particularly within a differentially closed field. Differentially closed fields are the differential analogue of algebraically closed fields for systems of finitely many differential algebraic equations in finitely many differential variables. The solution set of such a system is called a Kolchin closed set, and Boolean combinations thereof are called Kolchin constructible. 

Model-theoretically,  differentially closed fields of characteristic $0$ are  existentially closed differential fields (see Corollary~\ref{C:DCF_existclosed}) and their common theory, denoted by DCF$_0$, is \emph{the least misleading example of a totally transcendental theory},  according to \textcite[pp. 4--5]{Sacks}. It is a complete theory and has quantifier elimination (see Theorem~\ref{T:DCF_QE}), which is equivalent to saying that the projection of a Kolchin constructible set is again Kolchin constructible. From the point of view of the \emph{classification program} introduced by \textcite{Shelah}, the theory DCF$_0$ is very tame: it is \emph{$\omega$-stable}. Thus DCF$_0$ is equipped with a canonical notion of independence, which will be introduced in Section~\ref{S:Stab}, arising from stability. The general notion of independence in arbitrary stable theories is of combinatorial nature and yet captures many of the well-known properties of algebraic independence for algebraically closed fields. 

Using Pila's Ax--Lindemann--Weierstrass as well as the \emph{Embedding Theorem} of \textcite[Theorem, p. 160]{Seidenberg_Embedding}, \textcite[Theorem 3.10]{FreitagScanlon} showed that the Kolchin closed set of solutions of the differential algebraic equation $\chi(T)=0$ satisfied by the $j$-invariant function is \emph{strongly minimal}, working inside a sufficiently large (or \emph{universal} as in Section~\ref{S:Univ}) differentially closed field $(\UU, \delta)$: Given any Kolchin constructible subset $X$ of $\UU$, either the set $X\cap (\chi(T)=0)$ is finite or its relative complement $(\chi(T)=0)\setminus X$ is. Moreover, its \emph{geometry is trivial}: Given a countable differential subfield $K$ of $\UU$ and solutions $a_1,\ldots, a_n$ of $\chi(T)=0$, none of which is algebraic over $K$, the family \[ a_1,\ldots, a_n, \delta(a_1),\ldots, \delta(a_n), \delta^2(a_1),\ldots, \delta^2(a_n)\] is algebraically independent over $K$, as long as no two solutions $a_i$ and $a_j$ with $i\ne j$ satisfy a modular polynomial relation $F_N(X, Y)$ as previously introduced. More generally (see Definition~\ref{D:trivial}),  given a strongly minimal differential algebraic equation defined over the differential subfield $k$ of $\UU$ and 
$n$ solutions $a_1,\ldots, a_n$, we say that the solutions $a_1,\ldots, a_n$ are \emph{independent} over $K$ if for every $1\le i\le n$, the differential field $\str{K}{a_i}$ generated by $K(a_i)$ is algebraically independent from $\str{K}{a_1,\ldots, a_{i-1}}$ over $K$. Thus, the strongly minimal differential algebraic equation has \emph{trivial geometry} if for every collection of $n$  solutions  $a_1,\ldots, a_n$ and every differential field extension $K$ of $k$, the solutions are independent over $K$,  as long as no solution $a_i$ is algebraic over $K$ and any two distinct solutions $a_i$ and $a_j$ with $i\ne j$ are pairwise independent over $K$.

\textcite[Lemme 9]{Poizat_Rang} considered  a similar  differential algebraic (rational) equation, namely  \[ \frac{T''}{T'}  - \frac{1}{T}=0,\] and showed that it is also strongly minimal with trivial geometry.  Poizat's result was later generalized by \textcite[Theorems 3.1 \& 6.1]{FJMN23} to a wider class of differential algebraic equations of Li\'enard type, that is, of the form \[ T'' + f(T) T' + g(T)=0, \] for   rational functions $f$ and $g$ defined over the field of constants $\CC_\UU=\{x\in \UU \ | \ \delta(x)=0\}$. A differential algebraic equation defined over the field of constants is called \emph{autonomous}. \textcite[Theorem A]{Jaoui} obtained a similar result for any autonomous differential algebraic equation of order two and (differential) degree at least $3$ as long as the defining coefficients build an algebraically independent tuple over the prime field $\Q$.  

Whilst the above transcendence results use different methods, particular to
the equations they considered, some of the tools are common. Indeed, Zilber's
\emph{trichotomy principle} (Fact~\ref{F:dichotomy}) and the binding group
(Fact~\ref{F:binding}) are always present in some way or other. In recent work, \textcite[Theorem 3.9]{FJM22} have extracted (many of) the essential steps of previous proofs in order to show a stronger version of triviality, often called \emph{total disintegration}, for  autonomous differential algebraic equations. In an extremely elegant yet concise proof, they show the following result (see Theorem~\ref{T:Main} and Corollaries~\ref{C:sm} \& \ref{C:D2_otherparam}):
\begin{TheoIntro}
	Consider an irreducible  differential algebraic (rational) equation $P(T)=0$ of order $n\ge 1$ defined over a countable algebraically closed subfield $K$ of the constant field $\CC_\UU$ of an ambient differentially closed field $(\UU, \delta)$.  Assume that the algebraic differential equation has Property D$_2$, that is, for any two distinct solutions $a_1\ne a_2$, none of them algebraic over $K$ (or equivalently, neither $a_1$ nor $a_2$ belongs to $K$), the elements \[ a_1, \delta(a_1), \ldots, \delta^{n-1}(a_1), a_2, \delta(a_2),\ldots, \delta^{n-1}(a_2)\] from an algebraically independent family over $K$. 
	
	We have that the Kolchin constructible set given by $P(T)=0$ is strongly minimal and totally disintegrated: for every integer $m\ge 2$,  every differential field extension~$L$ of~$K$ as well as $m$~pairwise distinct solutions $a_1,\ldots, a_m$ in $\UU$,  none of which is algebraic over~$L$, the elements \[ a_1,\delta(a_1), \ldots, \delta^{n-1}(a_1), \ldots, a_m,\delta(a_m), \ldots, \delta^{n-1}(a_m)\] forms an algebraically independent family over~$L$.
\end{TheoIntro}

Whilst Poizat's equation has Property D$_2$  \parencite[Example 6.3]{FJMN23}, the autonomous equation $\chi(T)=0$ of the $j$-invariant function does not. Indeed, two solutions $a_1$ and $a_2$ may be distinct and yet satisfy a \emph{modular relation} $F_N(a_1, a_2)=0$ given by a Hecke correspondence with $N\ge 2$. In this case, the solution $a_2$ is algebraic over $\Q(a_1)\subset \str{\Q}{a_1}$, witnessing the failure of Property D$_2$.

\subsection*{On the goal of this (long) text: To whom is it addressed?}

\textcite[Theorem 3.6 \& Section \S 3.2]{FJM22} go beyond   the autonomous case mentioned beforehand. They have a similarly impressive result for irreducible differential algebraic equations within a differentially closed field, even if the equation is not defined over the field of constants, as long as the order of the differential algebraic equation is at least $2$. Their result resonates with work of \textcite{NaglooPillay} on Painlev\'e's equation. Nevertheless, we have decided to restrict the exposition to the case of an irreducible differential algebraic equation defined over the field of constants. We believe that the autonomous case already contains relevant model-theoretic tools and shows their interactions with differential algebra and transcendence theory. However, the fact that the equation is defined over the field of constants renders the proof somewhat simpler. Therefore, the goal of this text is to exhibit some of the ideas and methods from model theory, translated to the particular context of differentially closed fields. Through the text, we often provide definitions in algebraic terms, identifying what the classical notions in model theory correspond to in the particular tame theory  DCF$_0$ of differentially closed fields of characteristic $0$. The reader should be aware that the definitions and notions in question come from model theory and can be stated in a more general context (without even assuming that the ambient universe is a field!). Similarly, most of the proofs  are model-theoretic, though some arguments  become simpler in the particular case of differential fields.

 The core of the notions of \emph{geometric model theory} needed for the proof are contained mostly in Sections~\ref{S:Univ},  \ref{S:Stab} and~\ref{S:Intern}. Whilst these notions are well-known among model theorists, we have decided to present them for a general audience in mathematics as self-contained as possible, for we could not find many suitable references.  The first two sections are by now more mainstream, so we have mostly used the existing literature without any substantial input from our side. Section~\ref{S:DiffAlg}
 presents a basic introduction to  differential algebra and Section~\ref{S:DCF} introduces  the basic model-theoretic properties of the theory DCF$_0$.  Regarding the last sections, Section~\ref{S:binding} presents the model theoretic avatar of the differential Galois group and Section~\ref{S:D2} contains the full proof of the result of  \textcite{FJM22} for autonomous irreducible differential algebraic equations.
 
\subsection*{Acknowledgements}
 
I am mostly indebted to my dear colleagues Piotr Kowalski, Martin Hils, Rémi Jaoui, Rahim Moosa and Martin Ziegler for their scientific input in order to render some of the proofs and notions more accessible to a general audience.  I would like to specially thank Charlotte Bartnick, Thomas Blossier, Charlotte Hardouin, Rahim Moosa and Daniel Palacin for the careful reading of this long text and the feedback they have provided. Whilst there is still plenty of improvement for this text, this is only due to my lack of communication skills, and not due to the immense help all of them have provided!

I would also like to thank the members of the \textit{association des collaborateurs de Nicolas Bourbaki} for suggesting this amazing opportunity, which has made me discover (and reflect on) some beautiful mathematics along the way, as well as for their patience and willingness to help during the course of this long year of typing and preparation. Last, but not least, I would like to thank the anonymous readers of the \textit{association}, whose comments and editorial suggestions have been extremely helpful. 

\section{Some (very) basics of differential algebra}\label{S:DiffAlg}

In this section, we will present in a concise way the basic notions of differential algebra which will be needed in the next sections. We mostly follow the presentations of \textcite{Marker, Tressl} without any significant input from our side. 

\begin{notation}
For the sake of the presentation, all rings and fields in this section are commutative with identity and contain the field $\Q$ of rational numbers, so they all have have characteristic $0$. 
\end{notation}

\begin{defi}\label{D:diff_field}
A \emph{differential ring} is a ring $R$ equipped with a \emph{derivation}, that is, an additive homomorphism  $\delta\colon R\to R$ satisfying the Leibniz rule \[ \delta(a\cdot b)=\delta(a)\cdot b+ a\cdot \delta(b) \ \text { for all $a$ and $b$ in $R$}.\]

Given a differential ring $R$, the subring 
$\CC_R=\{ a\in R \ | \ \delta(a)=0\}$ is  the ring of \emph{constants} of $R$.   A \emph{differential field} is a differential ring whose underlying set is a field. 
\end{defi}

Many of the basic manipulations we know from an undergraduate course in analysis also hold for a derivation. In particular, the additive map $\delta$ is linear with respect to the constant elements. If the ring is an integral domain, then the derivation $\delta$  extends uniquely to a derivation on the field of fractions 

\begin{rema}\label{R:extension_diff}
Working inside an ambient field $M$, consider a subfield $K$ of $M$ and a derivation $\delta$ on $K$. 
\begin{enumerate}[(a)]
\item If an element $a$ of $M$ is algebraic over $K$, then there is a unique extension $\widetilde{\delta}$ of $\delta$ to the subfield $K(a)\subset M$: Indeed, if $m_a(T)$ denotes the minimal polynomial of~$a$ over~$K$, then \[ 0= \widetilde{\delta}(m_a(a))=m_a^\delta(a)+\frac{\partial m_a}{\partial T}(a) \widetilde{\delta}(a),\] where $m_a^\delta(T)$ is the polynomial over $K$ obtained  by differentiating each coefficient from $m_a(T)$ (so $\deg(m_a^\delta(T))<\deg(m_a(T))$ as $m_a(T)$ is monic).  Note that the field extension $K\subset K(a)$ is separable (since the characteristic is $0$), so $\frac{\partial m_a}{\partial T}(a)\ne 0$ and thus the value $\widetilde{\delta}(a)$ is uniquely determined. 

In particular, the algebraic closure of a differential field is again a differential field. 
\item\label{I:ext_lin_disj} Linear disjointness provides a general method to extend derivations. Recall that $K$ is \emph{linearly disjoint} from a subfield $L$ of $M$ over a common subfield $k$ if any elements $a_1,\ldots, a_n$ of $K$ are linearly independent over $L$ (seen as elements of $M$), whenever they are linearly independent over $k$.  Linear disjointness is a symmetric notion, despite the asymmetric definition. 

Given a  derivation $\delta_L$ on $L$ whose restriction to $k$ coincides with $\delta$ (so $(k, \delta\restr k)$ is a common differential subfield of $(K, \delta)$ and of $(L, \delta_L)$), if $K$ and $L$ are linearly disjoint over $k$, then there is a unique derivation $\widetilde{\delta}$ on the compositum field $K\cdot L$ extending both $\delta$ and $\delta_L$ with \[ \widetilde{\delta}(a\cdot b)= \delta(a)\cdot b + a\cdot \delta_L(b)  \text{ for all $a$ in $K$ and $b$ in $L$}.\]
\end{enumerate}
\end{rema}

\begin{defi}\label{D:diffpolys}
Given a differential ring $(R, \delta)$, the \emph{ring $R\{T\}$ of differential polynomials with coefficients in $R$} has as underlying set the ring $R[(T_n)_{n\in \N}]$ in infinitely many variables, where we identify the variable $T_0$ with $T$. It admits a natural derivation~$D$ extending $\delta$ imposing that $D(T_n)=T_{n+1}$, so we will often write~$T^{(n)}$ for~$T_n$ with $T^{(0)}=T$. Iterating this process, we obtain \emph{the ring of differential polynomials $R\{U_1,\ldots, U_m\}$ in the differential variables $U_1,\ldots, U_m$ with coefficients in~$R$}. 

The \emph{order} $\ord(P)$ of a differential polynomial $P(T)$ is the largest $n$ such that $T^{(n)}$ non-trivially occurs  in $P$ (if $P$ is a constant polynomial, then its  order is $-\infty$). If the order of $P(T)$ is a natural number $n$ in $\N$, then write \[ P(T)= \sum_{i=0}^{d} Q_i(T^{(0)}, \ldots, T^{(n-1)}) (T^{(n)})^i,\]
 for some $d\ge 1$ in $\N$ (called the \emph{degree} of $P$) and differential polynomials $Q_i$ of order strictly less than $n$ with $Q_d$ non-trivial. The \emph{separant} $s_P(T)$ is the differential polynomial $s_P=\frac{\partial P}{\partial T_n}$ obtained by formally differentiating the multivariate polynomial $P$ with respect to the variable $T^{(n)}=T_n$. Note that the separant $s_P(T)$ of $P(T)$ has either smaller order or smaller degree than $P$.

A \emph{differential ideal} is an ideal $I\subset R\{T\}$ closed under the action of the derivation~$D$, that is, an ideal such that  $D(P(T))$ belongs to~$I$ for every $P(T)$ in~$I$. A \emph{radical differential ideal}, resp.\ a \emph{prime differential ideal}, is a differential ideal which is  radical, resp.\ prime, as an ideal.

Given a subset $A$ of $R$, we denote by $(A)_\delta$ the \emph{differential ideal generated by} $A$. if  $A=\{a\}$ is a singleton,  write $(a)_\delta$  for $(A)_\delta$. 
\end{defi}

\begin{rema}\textup{(}\textcite[Lemma 1.3, Corollary 1.7  \& Lemma 1.8]{Marker}\textup{)}~\label{R:prime_minpol}
Consider a differential field $(K, \delta)$ and the induced derivation $D$ on the ring $K\{T\}$ of differential polynomials. 
\begin{enumerate}[(a)]
\item For every non-zero prime differential ideal $\primep$ of $K\{T\}$ there is some differential polynomial $P(T)$, called a \emph{minimal polynomial} of $\primep$, such that $P(T)$ is irreducible, as a multivariate polynomial, with $\primep=I(P)$,   where \[ I(P)=\{ Q \in K\{T\} \ | \ s_P^k\cdot Q \text{ belongs to $(P)_\delta$  for some $k$ in $\N$}\}.\] 
\item For every irreducible differential polynomial $P$, the above set $I(P)$ is a prime differential ideal with minimal polynomial $P(T)$. Moreover  \parencite[Corollary 1.2.10]{Tressl}, \[ I(P)\cap K[T^{(0)}, \ldots,T^{(\ord(P))}]= P\cdot K[T^{(0)}, \ldots,T^{(\ord(P))}].\] 
\end{enumerate} 
\end{rema}
The above remark yields that the minimal polynomial of a prime differential ideal is essentially unique, up to rescaling by an element of $K^*$. 
\begin{coro}\label{C:unique_minpol}
Given a non-zero prime differential ideal $\primep$, any two minimal polynomials of $\primep$ differ by a non-zero scalar, so the order $n$ of a minimal polynomial only depends on $\primep$. There is no non-trivial differential polynomial in $\primep$ of order strictly less than $n$. 
\end{coro}
\begin{proof}

Assume the prime ideal $\primep$ has two irreducible minimal polynomials  $P$ and $P_1$. We may assume that $\ord(P)\le \ord(P_1)$, so \begin{multline*}  P\in \primep\cap  K[T^{(0)}, \ldots,T^{(\ord(P_1))}]= \\ = I(P_1)\cap  K[T^{(0)}, \ldots,T^{(\ord(P_1))}] \stackrel{\ref*{R:prime_minpol} \, (b)}{=} P_1 \cdot K[T^{(0)}, \ldots,T^{(\ord(P_1))}].\end{multline*} Thus, there is some  $H$ in $K[T^{(0)}, \ldots,T^{(\ord(P_1))}]$ with $P=P_1\cdot H$. Irreducibility of $P$ (since $P_1$ is not a unit) yields that $H$ is a unit in $K[T^{(0)}, \ldots,T^{(\ord(P_1))}]$ and thus a non-zero scalar of $K$, as desired.  In particular, we have that $\ord(P)=\ord(P_1)$. 

If a differential polynomial $Q$  in $\primep$ has order $\ord(Q)<\ord(P)$, then we deduce from the above discussion that $Q=P(T)\cdot H$ for some $H$ in $K[T^{(0)}, \ldots,T^{(\ord(P))}]$.  Since $T^{(\ord(P))}$ occurs non-trivially in $P$, we conclude that $H$, and thus $Q$, is the zero polynomial, as desired. 
\end{proof}
Given a differential field extension $K \subset L$ and an element $a$ in $L$, denote by $\str{K}{a}_\text{Ring}$ the differential ring generated by $K\cup\{a\}$ (so the differential subfield $\str{K}{a}$ of $L$ generated by $K\cup\{a\}$ is the field of fractions of $\str{K}{a}_\text{Ring}$). There is a  natural differential epimorphism given by \emph{evaluation on the element $a$} \[ \begin{array}{ccc} 
	K\{T\}& \to & \str{K}{a}_\text{Ring}\\[2mm]
	Q(T, T^{(1)},\ldots, T^{(\ord(Q))}) & \mapsto & Q(a, \delta(a), \ldots, \delta^{(\ord(Q))}(a))
\end{array}.\] 
\begin{defi}\label{D:difftr_alg}
Consider a differential field extension $K \subset L$ as well as an element $a$ in $L$. The ideal $I(a/K)$ is the kernel of the evaluation morphism, that is, the prime ideal of all differential polynomials over $K$ which vanish on $a$. 

The element $a$ is \emph{differentially transcendental} over $K$ if $I(a/K)=\{0\}$.   

Otherwise, the element $a$ is \emph{differentially algebraic} over $K$. The \emph{differential rank} of $a$ over $K$ is $\RD(a/K)=\ord(P_a)$, for some  minimal polynomial $P_a$ of the non-zero prime ideal $I(a/K)$ (this is well-defined by Corollary~\ref{C:unique_minpol}). 
\end{defi}
\begin{rema}\label{R:iso_diff}
With the notation of the previous definition, the evaluation morphism induces a differential isomorphism between 
	$\mathrm{Quot}(K\{T\}/I(a/K))$ and $\str{K}{a}$. If $a$ is differentially algebraic, then the ideal $I(a/K)$ is completely determined by a minimal polynomial $P_a(T)$ of $I(a/K)$. In particular, if $\varphi\colon K\to K'$ is a differential isomorphism and $a'$ is an element (in some differential field extension $L'$ of $K'$) whose minimal polynomial over $K'$ is  $\varphi(P_a)(T)$, then there is a differential field isomorphism $\widetilde{\varphi}\colon \str{K}{a}\to \str{K'}{a'}$  extending $\varphi$  and mapping $a$ to $a'$. 
\end{rema}

\begin{coro}\label{C:diffrank_trdeg}
Given a differential field extension $K \subset L$ and an element $a$ in $L$, we have the following:
\begin{itemize}
	\item If $a$ is differentially transcendental over $K$, then $\str{K}{a}$  is isomorphic to the field of fractions of the differential ring $K\{T\}$, and thus has infinite transcendence degree over $K$. 
	\item If $a$ is differentially algebraic, then the differential rank $\RD(a/K)$, that is, the order of a minimal polynomial $P_a(T)$ of $I(a/K)$, equals the transcendence degree of $\str{K}{a}$ over $K$. 
\end{itemize}
\end{coro}
\begin{proof}
If $a$ is differentially transcendental over $K$,  then the ideal $I(a/K)$ is trivial. By Remark~\ref{R:iso_diff}, the field $\str{K}{a}$  is isomorphic to the field of fractions of the differential ring $K\{T\}$.

Assume now that the differential rank $\RD(a/K)=n=\ord(P_a)$ is a  natural number.  By Corollary~\ref{C:unique_minpol}, the elements $a, a'=a^{(1)}=\delta(a),\ldots, a^{(n-1)}=\delta^{(n-1)}(a)$ must be algebraically independent over $K$. Moreover, the element $a^{(n)}$ is algebraic over $K(a,\ldots, a^{(n-1)})$, by Definition~\ref{D:diffpolys}. We deduce from Remark~\ref{R:extension_diff} that the derivatives $a^{(n+k)}$, for $k$ in $\N$, are all  algebraic over $K(a,\ldots, a^{(n-1)})$, so $n$ equals the transcendence degree of $\str{K}{a}$ over $K$, as desired, 
\end{proof}

A commutative ring with identity is \emph{noetherian} if every ideal is finitely generated, or equivalently, if every increasing chain of ideals stabilizes. Clearly, given a differential field $(K, \delta_K)$,  the ring $K\{T\}$ of differential polynomials is not noetherian, as witnessed by the chain \[ (T) \subset (T, T^{(1)}) \subset \cdots \subset (T, T^{(1)}, \ldots, T^{(n)} ) \subset \cdots \]
However, the theorem of Ritt--Raudenbush below, see  \textcite[Theorem 1.16]{Marker}, can be seen as a differential analogue of Hilbert's basis theorem once we  restrict our attention to radical differential ideals.

\begin{fait}\textup{(Theorem of Ritt--Raudenbush}\textup{)}~
Consider a fixed differential field $(K, \delta_K)$.  Every increasing chain of  radical differential ideals in $K\{U_1,\ldots, U_m\}$ stabilizes. In particular, every  prime differential ideal  $\primep$ of $K\{U_1,\ldots, U_m\}$ is finitely generated as a radical ideal, that is, there is  a finite set $A$ of differential polynomials of $K\{U_1,\ldots, U_m\}$  such that \[ \primep= \{ P \in K\{U_1,\ldots, U_m\} \ | \ P^n \in (A)_\delta \text{ for some $n$ in $\N$}\}.\]
 \end{fait}

\begin{defi}\label{D:Kolchintop}
A subset $V$ of $K^m$ is \emph{Kolchin closed} if there are finitely many differential polynomials $P_1,\ldots, P_n$ of $K\{U_1,\ldots, U_m\}$ such that \[V=\{\bar a\in K^m \ | \ P_1(\bar a)=\ldots=P_n(\bar a)=0\}.\]

The Kolchin topology on $K^m$ is the topology whose closed subsets are exactly the Kolchin closed subsets. Note that it is indeed a topology, by the theorem of Ritt--Raudenbush, as the vanishing ideal $I(\bar a/k)$ of a finite tuple $\bar a=(a_1,\ldots, a_m)$ of $K^m$ over a differential subfield $k$ of $K$ is a prime differential ideal. 
\end{defi}

\begin{coro}\label{C:Kolchintop}
The Kolchin topology on $K^m$ is noetherian. In particular, given a tuple $\bar a$ in $K^m$ and a differential subfield $k$ of $K$, there exists a smallest Kolchin closed subset $V$ of $K^m$ containing $\bar a$ given by differential polynomials with coefficients in $k$. We refer to $V$ as the \emph{Kolchin locus} of $\bar a$ over $k$. 
\end{coro}


\section{A crash-course in the basic model theory of differentially closed fields}\label{S:DCF}

The purpose of this section is to provide a concise introduction to the basic model-theoretic properties of  differentially closed fields of characteristic $0$. For the sake of the presentation, we will  reformulate some general notions from mathematical logic in the particular context of differential fields. Whilst we are well aware that the terminology we  provide is not \emph{standard},  we nonetheless hope that it will render the presentation more intuitive. 

As in the previous section, all fields and rings are commutative with identity, of characteristic $0$ and contain $\Q$. 

\begin{defi}\label{D:language}
We denote by $\LL_\delta=\{0,1,+, -, \cdot, \delta\}$  the language of differential rings. Note that every differential ring can be seen as an \emph{$\LL_\delta$-structure} with the obvious \emph{interpretations} of the symbols. 

A \emph{language expansion} $\LL$ of  $\LL_\delta$ consists of the symbols of $\LL_\delta$ plus possibly new constant symbols $(c_i)_{i\in I}$. A choice of an $\LL$\emph{-structure} for a differential field $K$ consists of a sequence  $(d_i)_{i\in I}$ of elements in $K$ for the \emph{interpretations} of the constant symbols $c_i$'s. We often use the notation $c_i^K$ for the interpretation $d_i$. 
\end{defi}
The reader should be aware that a constant symbol need not be interpreted as a constant element of derivative $0$. The double meaning of the word \emph{constant} is indeed unfortunate, but we do not want to introduce unnecessary non-standard terminology. 

\begin{defi}\label{D:formula}
Given a language expansion $\LL=\LL_\delta\cup\{c_i\}_{i \in I}$ of the differential ring language, a   \emph{differential (Kolchin-)constructible} (or \emph{quantifier-free}) $\LL$-formula in the variables $x_1,\ldots, x_n$ is a boolean combination of expressions of the form $P(x_1,\ldots, x_n, c_{i_1}, \ldots, c_{i_m})=0$, with $m$ in $\N$ and $P$ a differential polynomial in $\Z\{U_1,\ldots, U_n, V_1,\ldots, V_m\}$. 

The class of all $\LL$-formulae is the smallest collection of (formal) expressions obtained from the differential constructible formulae which is stable under boolean combinations, \emph{existential quantifications} (or \emph{projections}) and logical implications and equivalences, that is, if $\varphi$ and $\psi$ are $\LL$-formulae and $y$ is some variable, then  \[ \neg \varphi, (\varphi\lor \psi), (\varphi\land \psi), (\varphi\Rightarrow \psi), \  (\varphi\Leftrightarrow \psi), \ \exists y \varphi \] are again $\LL$-formulae. In particular, the abbreviation $\forall y \varphi= \neg\exists y \neg \varphi$ will also be treated as an $\LL$-formula. 

A  \emph{bound variable} of an $\LL$- formula $\varphi$ is a variable  $y$ occurring  within the scope of a quantifer $\exists$ or $\forall$. A variable which is not a bound variable is called \emph{free}. We use the notation $\varphi=\varphi(x_1,\ldots, x_n)$ to denote that the free variables which actually occur in $\varphi$ are among the variables in the list $x_1,\ldots, x_n$. A \emph{sentence} is a formula all whose variables are bound variables. 

Given an $\LL$-formula $\varphi(x_1,\ldots, x_n)$ and a tuple $\bar a=(a_1,\ldots, a_n)$ in a differential field~$K$ (seen as an $\LL$-structure), we say that $\bar a$~\emph{realizes} (or \emph{satisfies}) the formula~$\varphi$ in~$K$, denoted by $K\models \varphi(a_1,\ldots, a_n)$, if the conditions determined by~$\varphi$ hold in the differential field~$K$ (with the natural intepretations), once we replace each free variable~$x_i$ by the element~$a_i$. 

Given a subset $B$ of parameters of $K$ as well as an  $\LL$-formula $\varphi(\bar x, \bar y)$ and a tuple $\bar b$ in $B^{|\bar y|}$,  a \emph{$B$-instance} $\varphi(\bar x, \bar b)$  is the subset  \[ \{\bar a \in K^{|\bar x|} \ | \ K\models\varphi(\bar a, \bar b)  \}. \] 
\end{defi}
\begin{exem}\label{E:DF}
There exists a countable collection of $\LL_\delta$-sentences, denoted DF$_0$, expressing that the underlying $\LL_\delta$-structure is a differential field (of characteristic $0$). Indeed, it suffices to consider the finite set of (commutative) ring axioms together with the following axioms: 
\begin{description}
\item[Inverses exist] $\forall x\exists y (x\ne 0\Rightarrow x\cdot y=1)$. 
\item[Fields are not trivial] $1\ne 0$.
\item[Derivation is additive] $\forall x\forall y \ \delta(x+y)=\delta(x)+\delta(y)$
\item[Leibniz rule] $\forall x\forall y \ \delta(x\cdot y)=x\delta(y)+y\delta(x)$. 
\end{description}
and the following infinite list of sentences expressing that the characteristic is not a prime number.
\begin{description}
\item[Characteristic different from $p>0$] $\underbrace{1+\cdots+1}_p\ne 0$.   
\end{description}
Every differential field $K$, seen as an $\LL_\delta$-structure, is a \emph{model} of DF$_0$, denoted by $K\models \mathrm{DF}_0$. More generally, given an expansion $\LL$ of $\LL_\delta$ and an $\LL$-\emph{theory} $T$, that is, a collection of $\LL$-sentences, we say that an $\LL$-structure $K$ is a \emph{model} of $T$ if $K\models \chi$ for every $\LL$-sentence $\chi$ in $T$. 
\end{exem}

\begin{rema}\label{R:log_aquiv}
	
An $\LL_\delta$-formula $\psi(x_1,\ldots, x_n)$ is in \emph{prenex normal form} if \[ \psi(x_1,\ldots, x_n)=Q_1 y_1\cdots Q_m y_m \chi(x_1,\ldots, x_n, y_1,\ldots, y_m) \] for some differential constructible formula $\chi$ such that each symbol $Q_j$ is either the quantifier $\forall$ or the quantifier $\exists$.    

Modulo the above theory DF$_0$, every $\LL_\delta$-formula $\varphi(x_1,\ldots, x_n)$ is \emph{logically equivalent} to 
a formula $\psi$ in prenex normal form, that is, for every differential field $K$ (seen as an $\LL_\delta$-structure), we have that \[ K\models \forall x_1\cdots\forall x_n \ \big( \varphi(x_1,\ldots, x_n) \Leftrightarrow \psi(x_1,\ldots, x_n) \big).\] 
\end{rema}

\begin{defi}\label{D:substr_iso}
Consider a language expansion $\LL=\LL_\delta\cup\{c_i\}_{i \in I}$ and two differential fields $K$ and $L$, seen both as $\LL$-structures.

An \emph{embedding} of $K$ in $L$ is a differential monomorphism $F\colon K\to L$ which is compatible with the interpretations, so $F(c_i^K)=c_i^L$ for all $i$ in $I$. An \emph{isomorphism} of $\LL$-structures is a surjective embedding. If the embedding $F$ is the set-theoretic inclusion of fields, we say that $K$ is an $\LL$-substructure of $L$. 

Given $\LL$-substructures $k$ of $K$ and $k'$ of $K'$, an isomorphism $F\colon k\to k'$ is a \emph{partial elementary map between $K$ and $K'$}  if $F$ \emph{preserves satisfaction}, that is, for every $\LL$-formula $\varphi(x_1,\ldots, x_n)$ and every tuple $a_1,\ldots, a_n$ in $k$, we have that \[ K \models \varphi(a_1,\dots, a_n) \ \iff \ L \models \varphi(F(a_1),\ldots, F(a_n)).\] In particular, the fields $K$ and $L$ must satisfy the same $\LL$-sentences. 

A substructure $k$ of $K$ is an \emph{elementary substructure} if  the inclusion map $k\to K$ is a partial elementary map. 
\end{defi}
Consider the differential field $\Q(T)$ with $\delta(T)=1$ as an $\LL_\delta$-structure. Whilst  $\Q$ is a differential subfield of $\Q(T)$, it is not an elementary substructure of $\Q(T)$, since the satisfaction of the sentence $\forall y (\delta(y)=0)$ is not 
preserved.
\begin{rema}\label{R:iso_qf}
	Consider a language expansion $\LL=\LL_\delta\cup\{c_i\}_{i \in I}$ and an $\LL$-isomorphism $F\colon k\to k'$ of differential fields, seen as $\LL$-structures. Assume furthermore that there are $\LL$-structures $K$ and $K'$ such that $k$, resp.\ $k'$, is an $\LL$-substructure of $K$, resp.\ $K'$.  The map $F$ induces a map from $k$-instances $\varphi(\bar x, \bar b)$ of differential constructible formulae to $k'$-instances  $\varphi(\bar x, F(\bar b))$, applying $F$ to the parameters defining the instance. 
	
	Given a tuple $\bar a$ in $K$, consider the collection \[ \Sigma(x_1,\ldots, x_n)=\{ \varphi(\bar x, F(\bar b)) \ | \ \varphi \text{	diff. constructible } \& \  K\models \varphi(\bar a, \bar b)\}.\]
	
	A tuple $\bar a'$ in $K'$ realizes $\Sigma(\bar x)$, that is, the tuple $\bar a'$ realizes in $K'$ every instance in $\Sigma$, if and only if $F$ extends to an $\LL$-isomorphism $G\colon  \str{k}{\bar a}\to \str{k'}{\bar a'}$ mapping $\bar a$ to $\bar a'$, by a straightforward adaptation of Remark~\ref{R:iso_diff}. 
	
	In particular, if $K$ is an $\LL$-substructure of $K'$, a tuple $\bar a$ in $K$ realizes a differentially constructible $k$-instance in $K$ if and only if $\bar a$ realizes the $k$-instance in $K'$. 
\end{rema}

\begin{defi}\label{D:sat}
	Given a language expansion $\LL=\LL_\delta\cup\{c_i\}_{i \in I}$, consider a differential field $K$, seen as an $\LL$-structure, as well as a differential subfield $k$ of $K$. A \emph{partial $n$-type} $\Sigma(x_1,\ldots, x_n)$ over $k$ is a (possibly infinite) collection of $k$-instances $\varphi(\bar x, \bar b)$ with $\bar x=(x_1,\ldots, x_n)$ and $\bar b$ in $k$ such that $\Sigma(\bar x)$ is \emph{finitely consistent in $K$}, that is,  for every finite subset $\{\varphi_i(\bar x, \bar b_i)\}_{i=1}^m$ of instances in $\Sigma$ there exists a common realization $\bar a=(a_1,\ldots, a_n)$ in $K^n$: \[ K\models \bigwedge_{i=1}^m \varphi_i(\bar a, \bar b_i).\]
	
	A  \emph{partial $n$-type} $\Sigma(x_1,\ldots, x_n)$ over $k$  is realized by $\bar a=(a_1,\ldots, a_n)$ if $K\models \varphi(\bar a, \bar b)$ for every $k$-instance $\varphi(\bar x, \bar b)$ of $\Sigma$. 
	
	The differential field $K$ is \emph{$\aleph_1$-saturated} if it realizes every partial $n$-type over every countable differential subfield. 
\end{defi}

We now define the class of differential fields which will play a crucial role in the work of \textcite{FJM22}.
\begin{defi}\label{D:DCF}
	A differential field $K$ is \emph{differentially closed} if for every pair of non-trivial differential polynomials $P$ and $Q$ in one variable with $\ord(Q)<\ord(P)$, there exists an element $a$ in $K$ with $P(a)=0\ne Q(a)$. 
\end{defi}
Notice that a differentially closed field~$K$ is in particular algebraically closed, since on one hand, every polynomial can be seen as a differential polynomial (or order at most~$0$) and on the other hand, the order of the constant polynomial~$1$ is $-\infty$. It follows immediately from Remark~\ref{R:extension_diff}.(a) that the constant subfield~$\CC_K$ of~$K$ is also algebraically closed. 
\begin{rema}\label{R:DCF0_Einbettung}
	\begin{enumerate}[(a)]
		\item Every differential field can be embedded into a differentially closed field, by a standard chain argument, using Remark~\ref{R:prime_minpol}.(b) \parencite[Lemma 2.2]{Marker}. 
		\item As in Example~\ref{E:DF}, there exists an infinite collection of $\LL_\delta$-sentences, denoted by DCF$_0$, such that a differential field $K$ is a model of DCF$_0$ if and only if it is a differentially closed field. The theory  DCF$_0$ consists of adding to the theory DF$_0$ the  collection of sentences $(\chi_{n, d, d_1})_{n, d, d_1}$ described as follows: For every triple $(n,  d, d_1)$, choose an enumeration $(M^0_\alpha, \ldots, M^k_\alpha)$ of all monomials in the variables $T, T^{(1)}, \ldots, T^{(n-1)}$ of total degree at most $d$. Similarly, for every $m<n$, let $(M'_{\beta(m)})$ be an enumeration of all monomials in the variables $T, T^{(1)}, \ldots, T^{(m)}$ of total degree at most $d_1$. The 
		sentence $\chi_{n, d, d_1}$ encodes the particular axiom of differentially closed fields for a given (shape of a) differential polynomial $P$ of order $n$ and degree at most $d$, so 
		\begin{multline*} \chi_{n, d, d_1}= \bigwedge\limits_{m<n}  \forall \bar y^0_\alpha \ldots \forall\bar y^k_\alpha \forall \bar z_{\beta(m)} \exists x \Big( \big( \bigvee_{\beta(m)} z_\beta\ne 0 \land \bigvee_{j=1}^k  \bigvee_\alpha y^j_\alpha\ne 0  \land \bigwedge_{r=j+1}^k y^r_\alpha=0 \big) \Rightarrow \\  \big( \sum_{\alpha, j} y^j_\alpha M^j_\alpha(x, \ldots, \delta^{n-1}(x)) \delta^n(x)^j=0 \land \sum_\beta z_{\beta(m)} M'_{\beta(m)}(x, \ldots, \delta^m(x))\ne 0 \big) \Big)  .\end{multline*}
	\end{enumerate}
\end{rema}

\begin{exem}\label{E:types} 
We will exhibit two examples of partial $1$-types in a differentially closed field $K$  over a countable subfield $k$. 
\begin{enumerate}[(a)]
\item Given a differential polynomial $P$ with coefficients in $k$ of order $\ord(P)$, the following collection of $k$-instances 
\[ \Sigma_P(x)=\{P(x)=0\} \cup \{Q(x)\ne 0 \ | \ 0\ne Q \in k\{T\} \text{ with } \ord(Q)<\ord(P)\}\] is a partial $1$-type over $k$. Indeed, we need to show that every collection of finitely many instances in $\Sigma_P(x)$ has a common realization in $K$. We may assume without loss of generality that the instances we are given consist of the equation $P(x)=0$ together with  instances requiring that a finite number of non-trivial differential polynomials $Q_1(x),\ldots, Q_m(x)$ over $k$ do not vanish.  Now, the differential polynomial $Q=\prod_{j=1}^{m} Q_j$ over $k$ is again non-trivial of order strictly less than $\ord(P)$. Since $K$ is differentially closed, there exists an element $b$ in $K$ such that $P(b)=0$ yet $Q(b)\ne 0$, so $Q_j(b)\ne 0$ for $1\le j\le m$. Hence, the collection of instances in $\Sigma$ is finitely consistent, as desired. 

If the differential polynomial $P$ is irreducible, note that $b$ in $K$ realizes  $\Sigma_P$ if and only if $P$ is the minimal differential polynomial of the vanishing differential ideal $I(b/k)$. Such a realization need not in general exist (for example if $K=k$ is countable and $n\ge 1$). 
\item Similarly, the collection of $k$-instances 
\[ \Sigma_\text{diff.tr}(x)= \{Q(x)\ne 0 \ | \ 0\ne Q(T) \in k\{T\} \}\] is a partial $1$-type over $k$. Indeed,  finitely many $k$-instances in $\Sigma_\text{diff.tr}$ correspond to finitely many non-trivial differential polynomials $Q_1(x),\ldots, Q_m(x)$ over $k$. If we consider again the non-trivial differential polynomial $Q(T)=\prod_{i=1}^m Q_j$ of order $\ord(Q)=\ell$, the differentially closed field $K$ has an element $b$ with $\delta^{\ell+1}(b)=0$, yet $Q(b)\ne 0$, so  thus $Q_j(b)\ne 0$ for $1\le j\le m$, as desired. 

An element $b$ in $K$ realizes  $\Sigma_\text{diff.tr}$ if and only if it is differentially transcendental over $k$. 
\end{enumerate}
\end{exem}
\begin{rema}\label{R:sat_1types}
In order to show that sone differential field $K$ is $\aleph_1$-saturated, seen as an $\LL$-structure, it suffices to show that every partial $1$-type over every countable differential subfield $k$ of $K$ has a realization in $K$. Indeed, we argue by induction on the number of variables, so consider a partial $n$-type $\Sigma(x_1,\ldots, x_n)$ over $k$. Now,  the collection of instances \[ \{ \exists x_1 \bigwedge\limits_{j=1}^m \varphi_{i_j}(x_1,\ldots, x_n, \bar b_{i_j}) \ | \ m\in \N \ \& \  \varphi_{i_j}(x_1,\ldots, x_n, \bar b_{i_j}) \in \Sigma(\bar x) \} \] is a partial $(n-1)$-type in the free variables $x_2,\ldots, x_n$, which admits a realization $(a_2,\ldots, a_n)$ in $K$ by induction on the number of variables. Our choice of $a_2,\ldots, a_n$ yields that the set of instances \[ \{ \varphi(x_1, a_2,\ldots, a_n, \bar b) \ | \ \varphi(\bar x, \bar b) \in \Sigma \}  \] is now a partial $1$-type over the countable differential subfield $\str{k}{a_2,\ldots, a_n}$. By assumption, this partial $1$-type is realized by some element $a_1$ in $K$. The tuple $(a_1,\ldots, a_n)$ in~$K$ realizes~$\Sigma$, as desired. 
\end{rema}
Non-principal ultrapowers always yield $\aleph_1$-saturated structures \parencite[Theorem 6.1.1]{ChangKeisler}, whenever the language is countable. 
\begin{rema}\label{R:ultraprod_sat}
Given a countable language expansion $\LL=\LL_\delta\cup\{c_i\}_{i \in I}$, every differential field $K$, seen as an $\LL$-structure, admits an elementary extension which is $\aleph_1$-saturated, namely, the 
 \emph{ultrapower} $\prod_{\mathcal U} K$, where $\mathcal U$ is a non-principal ultrafilter on $\N$, that is, a finitely additive probability measure on $\N$ only  taking the values $0$ and $1$ such that every finite set has measure $0$.  With this identification in mind, we will often use the expression \emph{for $\mathcal{U}$-almost all $n$ in $\N$} if the collection of such $n$'s has measure $1$, that is, it belongs to the ultrafilter $\mathcal U$. 

The ultrapower $\prod_{\mathcal U} K$ consists of the equivalence classes of
sequences $(a_n)_{n\in \N}$ in $\prod_{n\in \N} K$, where we identify two
sequences $(a_n)_{n\in \N}$ and $(a'_n)_{n\in \N}$, if the corresponding
entries $a_n$ and $a'_n$ agree for $\mathcal U$-almost all $n$. The ultrapower
$\prod_{\mathcal U} K$ becomes a differential field, setting for the
equivalence classes $[(a_n)]_\mathcal U$ and $[(b_n)]_\mathcal U$ in
$\prod_{\mathcal U} K$ the following:
\begin{align*}
  [(a_n+b_n)]_{\mathcal U} &= [(a_n)]_{\mathcal U}+[(b_n)]_{\mathcal U};  \quad    [(a_n\cdot b_n)]_{\mathcal U} = [(a_n)]_{\mathcal U}\cdot [(b_n)]_{\mathcal U};  \\
0_{\prod_{\mathcal U}K} &= [(0_K)]_{\mathcal U}; \quad  1_{\prod_{\mathcal U}K}  =[(1_K)]_{\mathcal U};  \quad    c_i^{\prod_{\mathcal U}K} =[(c_i^K)]_{\mathcal U};    \quad
  \delta([(a_n)]_{\mathcal U}) = [(\delta(a_n))]_\mathcal U.
\end{align*}

Given an $\LL_\delta$-formula $\varphi(x_1,\ldots, x_m)$, we deduce \L{}o\'s's theorem in this particular set-up, by an easy 
 induction argument on the number of quantifiers of its prenex normal form (see Remark~\ref{R:log_aquiv}):  for all equivalence classes 
 $[(a^1_n)]_\mathcal U, \ldots,  [(a^m_n)]_\mathcal U$, 
 we have that 
 \[ \prod_{\mathcal U} K\models \varphi\left( [(a^1_n)]_\mathcal U, \ldots,  [(a^m_n)]_\mathcal U  \right)  \  \iff \  K\models \varphi(a_n^1,\ldots, a_n^m) \text{ for $\mathcal U$-almost all $n$}.\]
In particular, the differential field $K$ is an elementary substructure of the ultrapower, via the natural embedding \[ \begin{array}[t]{rcl}
K&\to& \prod_{\mathcal U} K \\[1mm]
a & \mapsto & [(a, \ldots, a, \ldots)]_\mathcal{U}
\end{array}.\]

In order to prove that $\prod_{\mathcal U} K$ is $\aleph_1$-saturated, consider a countable differential subfield~$k$ of $\prod_{\mathcal U} K$ as well as a partial $\ell$-type $\Sigma(x_1,\ldots, x_\ell)$ over~$k$. Since both~$k$ and the language expansion~$\LL$ are countable, there are only countably many $k$-instances. In order to simplify the notation,   list all the instances in $\Sigma(\bar x)$ as $(\varphi_m(\bar x, \bar b_m))_{m\ge 1}$ for some finite tuple~$b_m$ with entries in~$k$ consisting of equivalence classes of sequences. Now, for every $1\le m$ in $\N$, we have that \[\prod_{\mathcal U} K\models \exists \bar x \bigwedge_{k=1}^m \varphi_k(\bar x, \bar b_k),\] so  by \L{}o\'s's Theorem  the set of indexes \[X_m=\{ n \in \N \ | \ m\le n \ \& \  K\models  \exists \bar x \bigwedge_{k=1}^m \varphi_k(\bar x, \bar b_k(n)), \} \] has $\mathcal U$-measure $1$, as  the finite subset $\{0,\ldots, m-1\}$ of $\N$ has $\mathcal U$-measure $0$.  Set $X_0=\N$. Given $n$ in $\N$, choose $m_n$ in $\N$ the largest index with $n$ in $X_{m_n}$ (such an index  $m_n$ always exists, since $\bigcap_m X_m=\emptyset$).  If $m_n=0$, then choose $\bar a(n)$ an arbitrary $\ell$-tuple of $K^\ell$; otherwise let $\bar a(n)$ be a realization in $K^\ell$ of \[ \bigwedge_{k=1}^{m_n} \varphi_k(\bar x, \bar b_k(j)).\] 
The saturation of $\prod_{\mathcal U} K$ will follow once we show that the equivalence class $\bar a=[(\bar a(n))]_\mathcal U$ realizes every formula $\varphi_m(\bar x, \bar b_m)$ in $\Sigma(\bar x)$. By \L{}o\'s's theorem, it suffices to show that \[ X_m \subset \{n\in \N \ | \ K\models \varphi_m(\bar a(n), \bar b_m(n))\},\] for $X_m$ has $\mathcal U$-measure $1$. Now, if $n$ belongs to $X_m$, then $1\le m\le m_n$ (since $m_n$ is largest with $n$ in $X_{m_n}$). Hence, the tuple $\bar a(n)$ in $K^n$ realizes the conjunction $\bigwedge_{k=1}^{m_n} \varphi_k(\bar x, \bar b_k(j))$, and in particular the instance $\varphi_m(\bar x, \bar b_m(n))$, as desired. 
\end{rema}
We will now state a  version of Gödel's compactness theorem \parencite[Corollary 1.2.12]{ChangKeisler} customized for our purposes. The proof of the following Fact is an easy application of the classical compactness theorem and Remark~\ref{R:ultraprod_sat}. 
\begin{fait}\textup{(}The compactness theorem \textup{)}~\label{F:compact} Consider a countable language expansion $\LL=\LL_\delta\cup\{c_i\}_{i \in I}$ and an $\LL$-theory $T$ of differential fields (see Example~\ref{E:DF}) which is finitely consistent, that is, such that for every finite subset $T_0$ of sentences in $T$, there is a differential field (seen as an $\LL$-structure) which satisfies every sentence in $T_0$. Then, there exists an $\aleph_1$-saturated differential field $K$ which satisfies every sentence in $T$. 	
\end{fait}

Our next goal is to introduce the fundamental model-theoretic property of \emph{quantifier elimination} for the particular case of differentially closed fields. Indeed,  Theorem~\ref{T:DCF_QE}  yields that every differential formula is (logically equivalent) to a differentially constructible formula. Quantifier elimination is a syntactic notion which is strongly related to the semantic notion of Ehrenfeucht--Fra\"iss\'e games (which most model-theorists refer to simply as Back-\&-Forth).
\begin{defi}\label{D:BackForth}
Consider a language expansion $\LL=\LL_\delta\cup\{c_i\}_{i \in I}$ and two differential fields $K$ and $K'$ (seen as $\LL$-structures). We say that a collection $\mathcal S$ of partial isomorphisms between countable differential subfields of $K$ and $K'$ is a \emph{Back-\&-Forth system} if it satisfies the following two conditions for every partial isomorphism $F\colon k\to k'$ in $\mathcal S$, where $k$, resp.\ $k'$, is a countable differential subfield of $K$, resp.\ of $K'$. 
\begin{description}
\item[Forth] For every element $a$ in $K$, there exists some extension of $F$ to a partial isomorphism $G\colon  \str{k}{a}\to \str{k'}{a'}$ in $\mathcal S$ for some $a'$ in $K'$. 
\item[Back] For every element $b'$ in $K'$, there exists some extension of $F$ to a partial isomorphism $H\colon  \str{k}{b}\to \str{k'}{b'}$ in $\mathcal S$ for some $b$ in $K$. 
\end{description}
\end{defi}
\begin{rema}\label{R:BackForth_elementary}
\begin{enumerate}[(a)]
	\item It follows immediately from Remark~\ref{R:iso_qf} by induction on the number of quantifiers of the prenex normal form (see Remark~\ref{R:log_aquiv}) that every partial isomorphism $F\colon k\to k'$ of a Back-\&-Forth system $\mathcal S$ between $K$ and $K'$ is a partial elementary map, as in Definition~\ref{D:substr_iso}. 
\item Every global isomorphism $F\colon K\to K'$ induces a Back-\&-Forth system by considering the restrictions of $F$ to countable differential subfields of $k$ of $K$. In particular, isomorphisms are elementary maps and thus preserve satisfaction. 
\end{enumerate}
\end{rema}
\begin{defi}\label{D:QE_complete}
	Assume that a  theory of differential fields $T$ in the language  $\LL_\delta$ of differential rings admits a \emph{model}, that is, a differential field which satisfies every sentence (or \emph{axiom}) of the theory $T$. 
\begin{enumerate}[(a)]
\item The theory $T$ \emph{eliminates quantifiers} if for every $\LL_\delta$-formula $\varphi(x_1,\ldots, x_n)$ there exists a differentially constructible formula $\theta(x_1,\ldots, x_n)$ such that \[ K\models \forall x_1 \ldots \forall x_n \big(  \varphi(x_1,\ldots, x_n) \Leftrightarrow \theta(x_1,\ldots, x_n) \big) \] for every model $K$ of $T$. 
\item We say that $T$ is \emph{complete} if every $\LL_\delta$-sentence which holds in some model of $T$ holds in every model of $T$. 
\end{enumerate}
\end{defi}
We can now give a customized version of the \emph{separation lemma} \parencite[Lemma 3.1.1]{TentZiegler} adapted to our setting. Whilst Back-\&-Forth is relatively easy to show in practice, it does not provide an effective method to describe the resulting quantifier-free formula, due to the use of compactness in the proof of the next result. Effective quantifier elimination is usually achieved via a more syntactic approach. 
\begin{prop}\label{P:QE_sat}
	Consider a theory of differential fields $T$ in the language $\LL_\delta$ such that $T$ admits a \emph{model}. Assume that for every two $\aleph_1$-saturated models $K$ and $K'$ of $T$  the collection of all partial isomorphisms between countable differential subfields $k$ of $K$ and $k'$ of $K'$ is a non-empty Back-\&-Forth system. We have that the theory $T$ is complete and eliminates quantifiers. 
\end{prop}
\begin{proof}
We first show that $T$ eliminates quantifiers, so consider an $\LL_\delta$-formula $\varphi(x_1,\ldots, x_n)$. Choose new constant symbols $c_1,\ldots, c_n$  and consider now $T$ as a theory in the countable language expansion $\LL=\LL_\delta\cup\{c_1,\ldots, c_n\}$. Every model of $T$ can be seen as an $\LL$-structure (by choosing a particular tuple for the interpretations of the $c_i$'s). Given a differential constructible formula $\theta(x_1,\ldots, x_n)$, in order to show that   \[ K\models \forall x_1 \ldots \forall x_n \big(  \varphi(x_1,\ldots, x_n) \Leftrightarrow \theta(x_1,\ldots, x_n) \big) \] for every model $K$ of $T$, it suffices to show that  $\theta(\bar c)$ is \emph{logically equivalent} to the sentence $\varphi(\bar c)$ modulo $T$, that is, 
\[ K\models \left(  \varphi(\bar c)  \Leftrightarrow \theta(\bar c) \right) \] for every model $K$ of the $\LL$-theory $T$, seeing the differential field $K$ as an $\LL$-structure. So we need only show that there is some differential constructible $\LL_\delta$-formula $\theta(\bar x)$ such that the $\LL$-sentence $\theta(\bar c)$ is logically equivalent to the $\LL$-sentence $\varphi(\bar c)$ modulo $T$.

Clearly, if the $\LL$-sentence $\varphi(\bar c)$ does not hold in any model of $T$, then it is equivalent to the differential constructible formula $0\ne 0$. Therefore,  we may assume that $T\cup \{\varphi(\bar c)\}$ admits a model. By compactness (Fact~\ref{F:compact}), fix some $\aleph_1$-saturated differential field $K$, seen as an $\LL$-structure, satisfying every axiom of $T$ as well as $\varphi(\bar c)$.  Let $\bar a$ in $K^{|\bar c|}$ be the interpretation of $\bar c$ in $K$. 

The $\LL$-theory \[T_1= T\cup \{\neg\varphi(\bar c)\}\cup\{\theta(\bar c) \ | \  \theta(\bar x) \text{ is diff. constructible } \& \  K\models \theta(\bar a)\} \] cannot be finitely consistent. Assume otherwise for a contradiction, so the above collection of $\LL$-sentences has an $\aleph_1$-saturated model $K'$ by compactness (Fact~\ref{F:compact}). In particular, the tuple $\bar a'=\bar c^{K'}$ given by the interpretation of $\bar c$ in $K'$ realizes every differential constructible $\LL$-formula which is realized by $\bar a$ in $K$. By Remark~\ref{R:iso_qf} there is an $\LL$-isomorphism $F\colon \str{\Q}{\bar a}\to \str{\Q}{\bar a'}$ mapping $\bar a$ to $\bar a'$. By  assumption together with Remark~\ref{R:BackForth_elementary}.(a), we conclude that $F$ is a partial elementary map between $K$ and $K'$, yet $K\models \varphi(\bar a)$ but $K'\not\models \varphi(F(\bar a))$, which gives the desired contradiction.

Since $T_1$ is not finitely consistent,  there are finitely many differentially constructible formulae $\theta_1(\bar x), \ldots, \theta_m(\bar x)$ with $K\models \bigwedge_{j=1}^m\theta_j(\bar a)$ such that for every differential field $L$ (seen as an $\LL$-structure), if $L$ is a model of $T$, then \[ L \models \big( \bigwedge_{j=1}^m \theta_j(\bar c)\Rightarrow \varphi(\bar c)\big) \tag*{(1)}\] Set now $\theta_K(\bar x)=\bigwedge_{j=1}^m \theta_j(\bar x)$, which is a differentially constructible $\LL_\delta$-formula.  Notice that our $\aleph_1$-saturated model $K$ satisfies the $\LL$-sentence $\theta_K(\bar c)$.

Running now over all possible choices of $\aleph_1$-saturated models $K$ of our theory $T$  and all possible differentially constructible formulae $\theta_K$ as above, consider the $\LL$-theory  \[ T_2=T\cup \{\varphi(\bar c)\}\cup\{\neg\theta_K(\bar c) \ |  \ K  \text{ is $\aleph_1$-saturated  } \& \  K\models \varphi(\bar c)\}.\] 
Notice that $T_2$ is indeed a set of sentences, as our language $\LL$ is fixed, though the collection of all possible models of $T$ is a proper class and not a set!. 

We now show that the theory $T_2$ 
is not be finitely consistent: Assume otherwise for a contradiction. By compactness (Fact~\ref{F:compact}), the theory $T_2$ admits an $\aleph_1$-saturated model $K'$, so $K'$ is a model  of $T$ with $K'\models \varphi(\bar c)$ yet the corresponding formula $\theta_{K'}(\bar c)$ does not hold  by construction, which gives the desired contradiction. As before, we deduce that there are  finitely many differentially constructible formulae $\theta_{K_1}, \ldots, \theta_{K_r}$ such that for every differential field $L$ (seen as an $\LL$-structure), if $L$ is a model of $T$, then \[ L \models \big( \varphi(\bar c) \Rightarrow \bigvee_{j=1}^r \theta_{K_j}(\bar c)\big) \tag*{(2)}.\] 

Now,  the differential constructible formula $\theta(\bar x)=\bigvee_{j=1}^r \theta_{K_j}(\bar x)$ is as desired. Indeed, by~(1) and~(2),   for every differential field~$L$ (seen as an $\LL$-structure), if $L$~is a model of~$T$, then \[ L\models  \left( \varphi(\bar c) \Leftrightarrow \theta(\bar c)\right).\] 

In order to conclude, let us now show that $T$ is complete. Assume otherwise for a contradiction, so there exists an $\LL_\delta$-sentence $\chi$ which holds in some model of $T$ but  not in every model in $T$. In particular, both theories $T\cup\{\chi\}$ and $T\cup\{\neg\chi\}$ are (finitely) consistent, so compactness (Fact~\ref{F:compact}) yields two $\aleph_1$-saturated differential fields $K$ and $K'$, each one a model of $T$, such that $K\models \chi$ yet $K'\models \neg\chi$. By assumption, there are countable isomorphic differential  subfields $k$ of $K$ and $k'$ of $K'$, as the Back-\&-Forth system between $K$ and $K'$ is not empty. Every map in a  Back-\&-Forth system is a partial elementary map, by Remark~\ref{R:BackForth_elementary}.(a), and thus preserves satisfation of sentences,  by  Definition~\ref{D:substr_iso}.  We obtain the desired contradiction, so $T$ is complete. 
\end{proof}

\begin{theo}\label{T:DCF_QE}
Given two $\aleph_1$-saturated differentially closed fields $K$ and $K'$, the collection of partial isomorphisms between countable differential subfields $k$ of $K$ and $k'$ of $K'$ is a non-empty Back-\&-Forth system. In particular, the theory DCF$_0$ of differentially closed fields of characteristic $0$ is complete and eliminates quantifiers. Every differential monomorphism  of differentially closed fields is elementary. 

Therefore, given a differentially closed field $K$, the collection of Kolchin constructible sets (that is, given by instances of Kolchin constructible formulae) is closed under projections. 
\end{theo}
\begin{proof}
	Completeness and quantifier elimination as well as the last assertion in the statement follow by Proposition~\ref{P:QE_sat}, once we show that, given two $\aleph_1$-saturated differentially closed fields $K$ and $K'$, the collection of partial isomorphisms between countable differential subfields $k$ of $K$ and $k'$ of $K'$ is a non-empty Back-\&-Forth system. It is clearly not empty, as both $K$ and $K'$ have characteristic $0$, so $\Q$   is a common differential subfield of both $K$ and $K'$, since $\Q$ only admits the trivial derivation. 
	
	Let us now show that {\bf Forth} holds, since the proof for {\bf Back} is analogous. Consider therefore a partial isomorphism $F\colon k\to k'$ as well as an element $a$ in $K$. If $a$~belongs to~$k$, set $G=F$. Assume thus that $a$ does not lie in $k$. If $a$ is differentially algebraic over $k$ with minimal polynomial $P_a(T)$, then set $F(P_a)(T)$ the polynomial over $k'$ obtained  from $P_a(T)$ by applying $F$ to each coefficient. The polynomial $F(P_a)(T)$ is clearly irreducible over $k'$ of order $\ord(P)$. By Remark~\ref{R:iso_diff}, we need only show that $K'$ contains an element $a'$ whose minimal polynomial over $k'$ equals  $F(P_a)(T)$. As in Example~\ref{E:types}.(a), the collection of $k'$-instances 
		\[ \Sigma_{F(P_a)}(x)=\{F(P_a)(x)=0\} \cup \{Q(x)\ne 0 \ | \ 0\ne Q \in k'\{T\} \text{ with } \ord(Q)<\ord(P)\}\] 
	is a partial $1$-type in the differentially closed field $K'$. Since $K'$ is $\aleph_1$-saturated and $k'$ is countable, we deduce that $K'$ contains a realization $a'$ of $\Sigma_{F(P_a)}$. By Remark~\ref{R:prime_minpol}.(b) and Corollary~\ref{C:unique_minpol}, we deduce that the minimal polynomial of $a'$ over $k'$ is indeed $F(P_a)(T)$, as desired. 
	
	If $a$ is differentially transcendental over $k$, it suffices to show by the same argument as above that $K'$ contains a differentially transcendental element $a'$ over $k'$. As in Example~\ref{E:types}.(b), the collection of $k'$-instances 
		\[ \Sigma_\text{diff.tr}(x)= \{Q(x)\ne 0 \ | \ 0\ne Q(T) \in k'\{T\} \}\]  
	is a partial $1$-type over $k'$ in the differentially closed field
        $K'$. Since $K'$ is $\aleph_1$-saturated, there is a realization $a'$
        in $K'$ of $\Sigma_\text{diff.tr}$. We conclude as before that the
        isomorphism $F\colon k\to k'$ extends to $\str{k}{a}$ and $\str{k'}{a'}$, as desired. 
\end{proof}
Quantifier elimination yields the following characterization of differentially closed fields as existentially closed among differential field extensions. 
\begin{coro}\label{C:DCF_existclosed}
A differential field $K$ is differentially closed if and only if it is  \emph{existentially closed} in the class of differential fields, that is, for every differentially constructible $\LL_\delta$-formula $\varphi(x_1,\ldots, x_n, y_1,\ldots, y_m)$  and every tuple $\bar b$ in $K^{|\bar y|}$, if the instance $\varphi(\bar x, \bar b)$ has a realization in some differential field extension $L$ of $K$, then there is a realization of $\varphi(\bar x, \bar b)$ already in $K$. 
\end{coro}
\begin{proof} 
\textcite[Theorem 2.1.3 (a)]{Tressl} gives an explicit proof that existentially closed differential fields are differentially closed, using a straightforward adaptation of Remark~\ref{R:iso_diff}. Let us therefore show the remaining implication, that is, that  a differentially closed field $K$ is existentially closed.  Thus,  consider a $K$-instance $\varphi(\bar x, \bar b)$ of a differentially constructible $\LL_\delta$-formula and assume that  there is a tuple $\bar a$ in some  differential field extension $L$ of $K$ realizing the instance (in $L$ seen as a $\LL_\delta$-structure). By Remark~\ref{R:DCF0_Einbettung}.(a), embed $L$ into a differentially closed field $L_1$.  By Theorem~\ref{T:DCF_QE}, the $\LL_\delta$-formula $\psi(\bar y)=\exists \bar x \varphi(\bar x, \bar y)$ is logically equivalent to a differentially constructible formula $\theta(\bar y)$ modulo DCF$_0$, that is, in every differentially closed field, the sets defined by $\psi(\bar y)$ and by $\theta(\bar y)$ coincide. Now, the instance $\varphi(\bar x, \bar b)$ defines a  Kolchin constructible set, regardless whether we work in  $L$ or in the differential extension $L_1$. Thus, we have that $L_1\models \varphi(\bar a, \bar b)$ and hence the tuple $\bar b$ 
realizes $\psi(\bar y)$, and thus $\theta(\bar y)$ in the differentially closed field $L_1$. 
Since $K$ is a differential subfield of $L_1$ and $\theta(\bar y)$ is a differentially constructible formula, we have that 
 the tuple $\bar b$ realizes $\theta$ in the differential subfield $K$ of $L_1$ by  Remark~\ref{R:iso_qf}. The field $K$ itself is differentially closed, so $K\models \psi(\bar b)$ and thus $K\models \exists \bar x \varphi(\bar x, \bar b)$. We deduce that there exists a realization in $K$ of the instance $\varphi(\bar x, \bar b)$, as desired. 
\end{proof}

To avoid any possible confusion, we would like to stress out that every differentially closed field $K$ admits proper field extensions $L$ which are \emph{differentially algebraic}, that is, every element $a$ of $L$ is differentially algebraic over $K$. Indeed, take a new transcendental element $u$ over $K$ and consider the polynomial ring $K[U]$ in the variable $U$. We can endow $K[U]$, and thus $K(U)$, with a derivation extending the derivation of $K$ such that $\delta(U)=0$, so $U$ is a new constant element which does not lie in $\CC_K$! 

However, many classical results of Galois theory can be shown for differentially algebraic field extensions of differentially closed fields. For example, the classical theorem of the primitive element has a differential analogue, shown first by \textcite[p. 728 with $m=1$]{Kolchin_primitive}, and more recently improved by \textcite[Theorem 2]{Pogudin_primitive}, which is valid over every differentially closed field $K$, as $K$ contains non-constant elements.  Indeed, choose a solution $a$ in $K$ to the differential algebraic equation $\delta(T)-1=0$ (setting $Q(T)=1$ in Definition~\ref{D:DCF}). The element $a$ is not a constant and is moreover transcendental over $\Q$, by Remark~\ref{R:extension_diff}.(a).


\section{Universal models and differential closure}\label{S:Univ}

The goal of this section is to introduce the notion of  \emph{saturated} and \emph{prime} models, in order to relate the syntactic notion of types to the Galois theoretic notion of orbits under the group of global automorphisms. Saturated models can be seen as universal differentially closed fields in the sense of Kolchin.  

From now on, we work inside an ambient $\aleph_1$-saturated differentially closed field $\UU$, see Fact~\ref{F:compact}. Unless explicitly stated, all subfields  are countable, all tuples are finite and they  will all be taken within $\UU$. 

\begin{defi}\label{D:types}
  Two tuples~$\bar a$ and~$\bar a'$ have \emph{the same type over} a differential subfield~$K$ of~$\UU$, denoted by $\bar a\equiv_K \bar a'$, if $I(\bar a/K)=I(\bar a'/K)$. This is  equivalent to (each of) the following properties (by Remark~\ref{R:iso_diff}): 
	\begin{itemize}
		\item There exists a partial isomorphism $F\colon \str{K}{\bar a}\to \str{K}{\bar a'}$ fixing $K$ pointwise and mapping $\bar a$ to $\bar a'$. 
		\item The tuples $\bar a$ and $\bar a'$ belong to the same differential constructible subsets defined over $K$. 
	\end{itemize} 
\end{defi}
Note that having the same type over $K$ is an equivalence relation. By an abuse of notation, we refer to the equivalence class of the tuple $\bar a$ with respect to the equivalence relation $\equiv_K$ as \emph{the type of $\bar a$ over $K$}, and denote this equivalence class by $\tp(\bar a/K)$. An element $\bar a'$ in this equivalence class is a \emph{realization} of the type of $\bar a$ over $K$.

\begin{defi}\label{D:isol}
The type of a tuple $\bar a$ over the differential subfield $K$ is \emph{isolated} by the $K$-instance $\varphi(\bar x, \bar b)$ with $\bar b$ in $K$ of a differentially constructible formula $\varphi(\bar x, \bar y)$ if \[ \UU \models \varphi(\bar a', \bar b) \ \iff \ \bar a'\equiv_K \bar a \]   for all $\bar a'$ in $\UU^{|\bar a|}$. 
\end{defi}

It  follows from Remark~\ref{R:iso_diff} and Theorem~\ref{T:DCF_QE} (or rather its proof using Back-\&-Forth) that the theory DCF$_0$ is \emph{$\omega$-stable}: there are only countably many types of singletons over a given countable subfield $K$ of $\UU$. Indeed, the type of an element $a$ over $K$ is uniquely determined by its minimal (differential) polynomial $P_a(T)$ in $K\{T\}$. Since there are only countably many differential polynomials over $K$, we immediately deduce the  $\omega$-stability of DCF$_0$. 

A consequence of $\omega$-stability is the existence of \emph{prime models}: 
\begin{fait}\label{F:diff_closure}
	Given a countable differential subfield $K$ of $\UU$, there exists a countable differentially closed subfield $\widehat{K}$ of $\UU$ containing $K$, called the \emph{differential closure} of~$K$, which is unique up to $K$-isomorphism among the differential field extensions of~$K$  satisfying that the type of every finite tuple~$\bar a$ in~$\widehat{K}$ over~$K$ is isolated. 
	
	If the $K$-instance $\varphi(\bar x, \bar b)$ isolates the type of $\bar a$ over $K$, then   $\widehat{K}\models \varphi(\bar a, \bar b)$ by Quantifier Elimination (Theorem~\ref{T:DCF_QE}). 
\end{fait}
Note that every isolated type over~$K$  must be realized in the differential closure~$\widehat{K}$ of~$K$. Indeed, if $\varphi(\bar x, \bar b)$ is the isolating $K$-instance of the type of $\bar d$ over $K$, then Corollary~\ref{C:DCF_existclosed} yields a realization $\bar d'$ in $\widehat{K}$, so $\bar d$ and $\bar d'$ have the same type over $K$. 

\begin{exem}\label{E:isolation}
As discussed at the end of Section~\ref{S:DCF}, let  $a$ in $\UU$ be such that $\delta(a)=1$. The type of $a$ over $\Q$ is isolated by the differential constructible formula $(\delta(x)=1)$.  Indeed, every solution $a'$ of the equation $\delta(T)=1$  is again transcendental over $\Q$. Thus, the differential fields $\str{\Q}{a}=\Q(a)$ and  $\str{\Q}{a'}=\Q(a')$ are isomorphic, so $a'$ has the same type as $a$ over $\Q$. 
\end{exem}

In contrast to the field algebraic closure or the real closure for ordered fields, the differential closure need not be minimal, as shown by  \textcite{Rosenlicht}. Nevertheless its  field of constants does not increase. 
\begin{coro}\label{C:diff_closure_noconstants}
	If the field of constants $\CC_K$ of $K$ is  algebraically closed, then  there are no new constants in $\widehat{K}$:  if $c$ is a constant element in $\widehat{K}$, then $c$ belongs to $\CC_K$. 
\end{coro}
Recall that the field of constants $\CC_K$ of $K$ is always  algebraically closed  whenever the field $K$ is algebraically closed, by Remark~\ref{R:extension_diff}.(a). 
\begin{proof}
	
	We first begin with an easy observation: 
	
	\begin{claimstar}
If $\CC_{\widehat{K}}$ is contained in the algebraic closure $K^{alg}$ of $K$, then $\CC_{\widehat{K}}=\CC_K$. 
	\end{claimstar}
	
	\begin{claimstarproof}
Assume that every constant element in $\CC_{\widehat{K}}$ is algebraic over $K$.  A constant element $c$ in $\CC_{\widehat{K}}$ has then minimal polynomial $m_c(T)$ with coefficients in $K$. An easy application of Remark~\ref{R:extension_diff}.(a) yields that $m_c^\delta(c)=0$ (since $\delta(c)=0$), which implies that every coefficient of $m_c^\delta$ must equal $0$, as the degree of $m_c^\delta(T)$ is strictly smaller than the degree of $m_c(T)$. Hence, the minimal polynomial $m_c(T)$ has coefficients in the algebraically closed field $\CC_K$, which gives the desired result. 
	\end{claimstarproof}
	
	Let us now prove that every constant element $c$ in $\CC_{\widehat{K}}$ is algebraic over $K$.  By Fact~\ref{F:diff_closure}, the type of $c$ over $K$ is isolated by a $K$-instance $\varphi(x, \bar b)$ of a differentially  constructible formula $\varphi(x, \bar y)$ with $\bar b$ in $K$. Assume for a contradiction that $c$~is not algebraic over~$K$, so the instance $\varphi(x, \bar b)$ can only involve a finite number of differential polynomial inequalities $Q_1(x)\ne 0,\ldots, Q_m(x)\ne 0$ with coefficients in $K$. Now, a differential polynomial evaluated in a constant element translates into a polynomial expression, so we may assume that each $Q_i(T)$ is a classical polynomial. The polynomial \[ Q(T)=\prod_{i=1}^{m} Q_i(T)\] is non-trivial (as $Q(c)\ne 0$), so there exists some rational number $q$ in $\Q$ with $Q(q)\ne 0$. Isolation yields that $c$ and $q$ must have the same type over $K$, which is a blatant contradiction (as $\Q$ is contained in $K$ yet $c$ is not algebraic over $K$). 
\end{proof}

We will now use the existence of the differential closure in order to produce a \emph{universal} differentially closed field. For this, we first need the following fact regarding the \emph{regularity} of the first uncountable cardinal $\aleph_1$.
\begin{fait}\label{F:aleph1_reg}
Consider an \emph{increasing chain} $(K_\alpha)_{\alpha<\aleph_1}$ of countable fields, that is, such that $K_\alpha\subset K_{\beta}$ if $\alpha<\beta<\aleph_1$ and $K_\gamma=\bigcup\limits_{\alpha<\gamma} K_\alpha$ if $\gamma<\aleph_1$ is a limit ordinal. For every countable subset $A$ of $\bigcup\limits_{\alpha<\aleph_1} K_\alpha$, there is some $\alpha<\aleph_1$ with $A\subset K_\alpha$. 
\end{fait}
\begin{prop}\label{P:univ}
\emph{Universal differentially closed fields} exist, that is, there is an $\aleph_1$-saturated differentially closed subfield $\UU'$ of $\UU$ whose cardinality is exactly $\aleph_1$. Moreover, for every partial isomorphism $F\colon K\to K'$,  with both $K$ and $K'$ countable subfields of~$\UU'$, there is a global automorphism $F'\colon \UU'\to \UU'$ whose restriction $F'\restr K$ equals $F$. 
\end{prop}
\begin{proof}
We will obtain $\UU'$ as the union of an increasing chain $(K_\alpha)_{\alpha<\aleph_1}$ of countable differentially closed subfields of $\UU$ such that at every successor stage, the differential field $K_{\alpha+1}$ contains a realization of every possible $1$-type over $K$ of an element $a$ of $\UU$. Set  $K_0=\widehat \Q$ the differential closure of $\Q$ (as a differential subfield of $\UU$ equipped with the trivial derivation). Assume that $(K_{\beta})_{\beta<\alpha}$ has been constructed for $\beta<\alpha$. If $\alpha$ is a limit ordinal, set $K_{\alpha}$ the differential closure of $\bigcup_{\beta<\alpha} K_\beta$. If $\alpha=\gamma+1$ for some $\gamma$, we know by $\omega$-stability that there are only countably many  types of elements of $\UU$ over $K_\gamma$, so choose an enumeration $(a_n)_{n\in\N}$ in $\UU$ of representatives of all these equivalence classes. Hence, for every $a$ in $\UU$ there is some $n$ in $\N$ with $a\equiv_{K_\gamma} a_n$. Set now $K_{\gamma+1}$ the differential closure of the countable differential subfield $\str{K_\gamma}{(a_n)_{n\in\N}}$. 

By construction, using Example~\ref{E:types}.(a), the uncountable field $\UU'=\bigcup_{\alpha<\aleph_1} K_\alpha$ is differentially closed. Moreover, its cardinality is exactly $\aleph_1$. Let us show that $\UU'$ is $\aleph_1$-saturated: Indeed, by Remark~\ref{R:sat_1types}, we need only show that $\UU'$ realizes every  partial $1$-type $\Sigma(x)$ with parameters over a countable differential subfield $k$ of $\UU'$. By Fact~\ref{F:aleph1_reg}, there exists some $\gamma<\aleph_1$ such that $k$ is a subfield of $K_\gamma$. The partial $1$-type $\Sigma$ is realized in $\UU$ by some element $a$, for $\UU$ is $\aleph_1$-saturated. By construction, there exists some representative $a'$ in $K_{\gamma+1}\subset\UU'$ of the type of $a$ over $k$, so $a'$ realizes $\Sigma$, as desired.  

For the last assertion of the statement, consider a partial isomorphism $F\colon K\to K'$ of some countable differential subfields $K$ and $K'$ of $\UU'$. Choose enumerations $(a_\alpha)_{\alpha<\aleph_1}$ and $(a'_\alpha)_{\alpha<\aleph_1}$ of $\UU'\setminus K$ and $\UU'\setminus K'$. Repeating the proof of the Back-\&-Forth argument in Theorem~\ref{T:DCF_QE} within the $\aleph_1$-saturated differentially closed subfield $\UU'$, we obtain  an increasing chain $(F_\alpha\colon L_\alpha\to L'_\alpha)_{\alpha<\aleph_1}$ of partial isomorphisms of countable differential subfields $L_\alpha$ and $L'_\alpha$ of $\UU'$ with $F_0=F\colon K\to K'$ such that for every $\alpha<\aleph_1$ the elements $a_\alpha$ and $a'_\alpha$ lie in $L_{\alpha+1}$ and in $L'_{\alpha+1}$, respectively. The map  $F'=\bigcup_{\alpha<\aleph_1} F_\alpha$ is a global  automorphism of $\UU'=\bigcup_{\alpha<\aleph_1} L_\alpha=\bigcup_{\alpha<\aleph_1} L'_\alpha$ extending $F$, as desired. 
\end{proof}
\begin{coro}\label{C:univ}
Two tuples $\bar a$ and $\bar a'$ of a universal differentially closed field $\UU'$ have the same type over a countable differential subfield $K$ of $\UU'$ if and only if $\bar a'=\sigma(\bar a)$ for some automorphism $\sigma$ of the differential field $\UU'$ fixing $K$ pointwise. \qed
\end{coro}
One of the main features of automorphisms in a universal differentially closed field is the following result, which can be seen as some sort of Galois correspondence. The proof follows the lines of the use of the classical separation lemma in our customized version in Proposition~\ref{P:QE_sat}.
\begin{lemm}\label{L:invariant_definition}
Given a universal differentially closed field~$\UU'$,  a differentially constructible subset~$X$ of $(\UU')^n$ and a countable differential subfield~$K$ of~$\UU'$, we have that $X$~is given by a $K$-instance $\psi(\bar x, \bar a)$ of a differentially constructible formula with parameters~$\bar a$ in~$K$  if and only if $X$~is setwise invariant under the action of the group  $\Aut_\delta(\UU'/K)$  of differential automorphisms of~$\UU'$ fixing~$K$ pointwise.
\end{lemm} 
\begin{proof}
If $X$~is given by a $K$-instance, then it is clearly $\Aut_\delta(\UU'/K)$-invariant. Therefore, we need only show the converse, so assume that the differentially constructible set~$X$ is $\Aut_\delta(\UU'/K)$-invariant. 
\begin{claimstar}
For every $\bar b$ in $X$, there exists some $K$-instance $\varphi_{\bar b}(\bar x, \bar a)$ with parameters in $K$ such that $\UU'\models \varphi_{\bar b}(\bar b, \bar a)$ and \[ \UU'\models \forall \bar x (\varphi_{\bar b}(\bar x, \bar a) \Rightarrow \bar x \in X),\] where we identify the expression $``\bar x \in X"$ with the differential constructible instance defining $X$. 
\end{claimstar}
\begin{claimstarproof}
Given $\bar b$ in $X$, the collection of $K$-instances \[\Sigma(\bar x)= \{ \psi(\bar x, \bar a) \ | \ \bar a \in K \ \& \ \UU'\models \psi(\bar b, \bar a)\} \cup\{ \bar x\notin X\} \] cannot be finitely consistent in $\UU'$. Assume otherwise for a contradiction.  By $\aleph_1$-saturation of $\UU'$, as the differential subfield $K$ is countable, the partial type $\Sigma(\bar x)$  has a realization $\bar b'$ in $\UU'$. Now, the tuple $\bar b'$ has the same type as $\bar b$ over $K$, so by Corollary~\ref{C:univ} there is an automorphism $\sigma$ of $\Aut_\delta(\UU'/K)$ mapping $\bar b$ to $\bar b'$. However, the tuple  $\bar b'$ does not lie in the  $\Aut_\delta(\UU'/K)$-invariant set $X$, which yields the desired contradiction. 

 Therefore, there are finitely many $K$-instances $\psi_i(\bar x, \bar a_i)$, each realized by $\bar b$, such that every realization of $\varphi_{\bar b}(\bar x,\bar a)=\bigwedge_{i} \psi_i(\bar x, \bar a_i)$ in $\UU'$ lies in $X$, as desired. 
\end{claimstarproof}

Now, since $K$ is countable, there are only countably many $K$-instances $\varphi_{\bar b}$ as in the Claim, even if $X$ may be possibly of cardinality $\aleph_1$. Consider the collection of $K$-instances \[ \Sigma_1(\bar x)= \{\bar x\in X\}\cup\{\neg \varphi_{\bar b}(\bar x, \bar a) \ | \ \varphi_{\bar b}\text{ as in the above claim}\}_{\bar b \in X}.\] 

Assume for a contradiction that $\Sigma_1(\bar x)$ is finitely consistent. By $\aleph_1$-saturation of $\UU'$, as $K$ is countable, there is a  realization $\bar b_1$ in $\UU'$. The element $\bar b_1$ belongs to $X$, yet  $\UU'\not\models \varphi_{\bar b_1}(\bar b_1, \bar a)$ by construction, contradicting the choice of $\varphi_{\bar b_1}$ in the Claim.

Hence, the the collection of $K$-instances $\Sigma_1$ is not finitely consistent, so there are finitely many $K$-instances $\varphi_{\bar b_i}(\bar x,\bar a_i)'s$ such that every $\bar b$ in $X$ must realize the disjunction $\varphi(\bar x, \bar a)=\bigvee_{i}\varphi_{\bar b_i}(\bar x,\bar a_i)$. Hence, the differentially constructible set $X$ is defined by the $K$-instance $\varphi(\bar x,\bar a)$, as desired. 
\end{proof}

We will conclude this section with two easy observations, which will be useful all throughout these notes. They provide an algebraic description of the  model-theoretic \emph{definable} and \emph{algebraic closures} of a subset~$A$ of the universal differentially closed field~$\UU'$. 
\begin{coro}\label{C:dcl}
Given a finite tuple $\bar a$ in a universal differentially closed field $\UU'$ and a countable differential subfield $K$ of $\UU'$, either $\bar a$ belongs to $K$ or there exists some tuple $\bar a'\ne \bar a$ whose type over $K$ equals the type of $\bar a$ over $K$. 
\end{coro}
\begin{proof}
By automorphisms, using Corollary~\ref{C:univ}, we may assume that the tuple $\bar a$ consists of a single element $a$. Assume therefore that $a$ does not belong to $K$. If $a$ is differentially transcendental over $K$, choose some $a'$ in $\UU'$ differentially transcendental over $\str{K}{a}$ by $\aleph_1$ saturation, so $a'\ne a$ yet $a\equiv_K a'$, as desired. 

If $a$~is differentially algebraic over~$K$, its minimal polynomial $P_a(T)$ has either order~$0$, yet  is not linear (as a polynomial in~$T$) or its order is $n>0$. In the first case, choose another root $a'\ne a$ of $P_a$ in (the algebraically closed field)~$\UU'$. If the order~$n$ of~$P_a$ is strictly positive, Example~\ref{E:types} yields that the collection of $K$-instances 
\[ \Sigma(x)=\{P_a(x)=0\} \cup \{Q'(x)\ne 0 \ | \ 0\ne Q'(T) \in \str{K}{a}\{T\} \text{ with } \ord(Q')<\ord(P)\}\] is a partial $1$-type over $\str{K}{a}$ (as in Definition~\ref{D:sat}). Now, since $\UU'$ is $\aleph_1$-saturated, the partial $1$-type $\Sigma(x)$  admits a realization $a'\ne a$ in $\UU$ whose  minimal polynomial over $K$ is $P_a(T)$, so $a\equiv_K a'$, as desired.   
\end{proof}
Mimicking the above proof, we immediately deduce the following result.  
\begin{coro}\label{C:acl}
	Given a finite tuple $\bar a$ in a universal differentially closed field $\UU'$ and a countable differential subfield $K$ of $\UU'$, either $\bar a$ belongs to $K^{alg}$ or the orbit of $\bar a$ under $\Aut_\delta(\UU'/K)$ is infinite, where $\Aut_\delta(\UU'/K)$ denotes the subgroup of all differential automorphisms of $\UU'$ fixing $K$ pointwise. \qed
\end{coro}
We provide a quick sketch of the proof: after reducing to a single element $a$, which we may immediately assume to be  differentially algebraic (since the differentially transcendental case is easy), it suffices to notice that if  $a$ is not algebraic over $K$, then its minimal differential polynomial $P_a(T)$ over $K$ must have order $n>0$. 


\section{Basics of stability and independence}\label{S:Stab}

Whilst we introduce in the previous section the notion of ($\omega$-)stability in terms of the number of types, one of the key features of this notion is that it is canonically equipped with an abstract notion of independence (defined in purely combinatorial terms). In the particular case of differentially closed fields, this notion of independence can be easily described in algebraic terms, yet one of the core goals of \emph{geometric model theory} is the study and development of the independence notion at an abstract level.

From now on, we work inside an ambient universal $\aleph_1$-saturated differentially closed field $\UU$, as in Proposition~\ref{P:univ}. All subfields and tuples will be taken within $\UU$. Unless explicitly stated, all proper (differential) subfields of $\UU$ we consider will be countable.

\begin{defi}\label{D:indep}
Consider differential subfields $K$, $L$  and $k$ of $\UU$ with $k\subset K\cap L$. We say that $K$ and $L$ are \emph{independent over $k$}, denoted by $K\ind_k L$, if $K$ and $L$ are \emph{algebraically independent} over $k$, that is, if for every tuple $(a_1,\ldots, a_n)$ of $K$ and every non-trivial  polynomial $P(T_1,\ldots, T_n)$ with coefficients in $L$ such that $P(a_1,\ldots, a_n)=0$, we have that $Q(a_1,\ldots, a_n)=0$ for some non-trivial  polynomial $Q(T_1,\ldots, T_n)$ with coefficients in $k$. 

Two tuples $\bar a$ and $\bar b$ are \emph{independent over} the differential  subfield $k$ of $\UU$ if the differential subfields $\str{k}{\bar a}$ and $\str{k}{\bar b}$ of $\UU$ are independent over $k$. 
\end{defi}

\begin{rema}\label{R:indep_char}
	\begin{enumerate}[(a)]
		\item Linear disjointness, introduced in Remark~\ref{R:extension_diff} (\ref{I:ext_lin_disj}), always implies algebraic independence. These two notions coincide whenever the base subfield $k$ is (relatively) algebraically closed  
		 \parencite[Chapter VIII, Theorem 4.12]{Lang_Alg}.  
		
		In particular, Remark~\ref{R:extension_diff} (\ref{I:ext_lin_disj}) yields that, whenever~$K$ and~$L$ are independent over the algebraically closed differential subfield~$k$, the compositum field $K\cdot L$ as a differential subfield of~$\UU$ has a unique derivation extending both the derivations on~$K$ and on~$L$.  
		\item Consider a differential subfield $k$ of $K$ and a   tuple $\bar a$ of $\UU$. We have that $\str{k}{\bar a}$ is linearly disjoint from $K$ over $k$ if and only if  the differential ideal $I(\bar a/K)$ can be generated (as a radical ideal) by finitely many differential polynomials with coefficients in $k$ \parencite[Chapter III, \S 2, Theorem 8]{Lang_AlgGeo}. 
		
		The field $\str{k}{\bar a}$ is always linearly disjoint from $K$ over $k$ whenever  $\str{k}{\bar a}\ind_k K$ and $k$~is algebraically closed, by the previous discussion. 
	\end{enumerate}
\end{rema}

\begin{rema}\label{R:indep_prop1}
The above notion of independence satisfies the following properties \parencite[Chapter VIII, \S 3]{Lang_Alg} for every triple of differential subfields $K$, $L$ and $k$ of $\UU$ as above:
\begin{description}
\item[Invariance]  Given an automorphism $\sigma$ of the differential field $\UU$, we have that $K\ind_k L$ if and only if $\sigma(K)\ind_{\sigma(k)} \sigma(L)$. 
\item[Symmetry] The independence $K\ind_k L$ holds if and only if $L\ind_k K$. 
\item[Finite Character] The independence $K\ind_k L$ holds if and only if  all  finite tuples $\bar a$ of $K$ and $\bar b$ of $L$ are independent over $k$. 
\item[Monotonicity \& Transitivity] Given a differential subfield $M$ of $\UU$ with $L\subset M$,  denote by $\str{\Q}{K\cup L}$  the differential field generated by $K$ and $L$. We  then have the following: 
\[K\ind_k M \ \iff \ \begin{cases} K\ind_k L \\[1mm] \hskip2mm \text{ and }\\[1mm]
	\str{\Q}{K\cup L} \ind_{L} M \\[1mm]
\end{cases}\] 
\item[Algebraic closure] The independence $K\ind_k L$ holds if and only if $K^{alg} \ind_{k^{alg}} L^{alg}$. Moreover, if an element~$a$ is independent from $\str{k}{a}$ over~$k$, then $a$~is algebraic over~$k$. 
\end{description}
\end{rema}

Consider a differential subfield~$K$   and a finite tuple~$\bar a$ of~$\UU$ such that $\str{k}{\bar a}\ind_k K$, where $k$~is a differential subfield of~$K$. It follows immediately from {\bf Invariance} that this independence does not depend on the particular representative of the type of~$\bar a$ over~$K$, that is, for every $\bar a'\equiv_K \bar a$, we have that $\str{k}{\bar a'}\ind_k K$, 
by Corollary~\ref{C:univ}. In this case, we say that $\bar a$~is \emph{independent from~$K$ over~$k$}.

\begin{rema}\label{R:indep_prop2}
	\begin{enumerate}[(a)]
\item The independence notion defined above satisfies the following principle of { \bf Stationarity}: Assume that $K\ind_k L$ with $k$ algebraically closed. Given  partial isomorphisms $F\colon K\to K'$ and $G\colon L\to L'$ which agree on $k$,  if  $K'\ind_{F(k)} L'$, then the composita fields $K\cdot L$ and $K'\cdot L'$ are isomorphic as  differential fields, by Remark~\ref{R:indep_char}.(a). Thus, there exists a global  automorphism $\sigma$ of the differential field~$\UU$ extending both $F$ and $G$ such that $\sigma(K\cdot L)= K'\cdot L'$. 

By Remark~\ref{R:indep_char}.(a), stationarity also holds if $k$~is relatively algebraically closed in~$L$ (and thus the field $F(k)=G(k)$ is also relatively algebraically closed in~$L'$, by Proposition~\ref{P:univ}). 

\item The independence notion of Definition~\ref{D:indep} satisfies the {\bf Extension} property: Given a finite tuple $\bar a$ and a (countable) differential subfield $K$ of $\UU$ with a differential subfield $k\subset K$, there exists some $\bar a'$ in $\UU$ with $\bar a'\equiv_k \bar a$ and $\str{k}{\bar a'}\ind_k K$. Indeed, we may assume by {\bf Algebraic closure} that $k$ is algebraically closed. Working inside some ambient algebraically closed field (containing our universal differentially closed field $\UU$), we can find a differential field extension $L$ of $k$, possibly outside of our ambient model $\UU$, which is linearly disjoint from $K$ over $k$ such that $L$ contains a tuple $\bar b$ whose differential vanishing ideal over $k$ is $I(\bar a/k)$ \parencite[Chapter III, \S 2, Corollaries 1 \& 2]{Lang_AlgGeo}. The compositum $K\cdot L$ is a differential field extending $K$, by Remark~\ref{R:extension_diff} (\ref{I:ext_lin_disj}).  Moreover, the tuple $\bar b$ has the same differential vanishing ideal as $\bar a$ over $k$ and $\str{k}{\bar b}\ind_k K$. We need only find a realization $\bar a'$ in $\UU$ which belongs to the same Kolchin constructible subsets of $\UU^{|\bar a|}$ defined over $K$ as those induced by the tuple $\bar b$ of $K\cdot L$ over $K$, and thus $\bar a\equiv_k \bar a'$, by Definition~\ref{D:types}. Indeed, if we have find such a tuple $\bar a'$, we 
immediately have that  $\str{k}{\bar a'}\ind_k K$, for this independence is encoded by Kolchin constructible subsets of $K\cdot L$ defined over $K$ containing $\bar b$.

 By Remark~\ref{R:DCF0_Einbettung}.(a), we can embed  $K\cdot L$ (and thus $K$) into some differentially closed field $M$.  Using Remark~\ref{R:iso_qf} and Corollary~\ref{C:DCF_existclosed}, we conclude that 
 the collection of $K$-instances \[ \{P(x_1,\ldots, x_n)=0 \}_{P(\bar T) \in I(\bar b/K)} \cup \{ Q(x_1,\ldots, x_n)\ne 0\}_{Q(\bar T) \notin I(\bar b/K)} \] is a partial $n$-type over $K$ in $\UU$, where $n=|\bar b|$. By $\aleph_1$-saturation of $\UU$, there exists a realization $\bar a'$ of the above partial type in $\UU$, as desired. 
\end{enumerate}
\end{rema}
The use of stationary types, which will be defined below, extends the principle of stationarity to types whose parameter set need not be algebraically closed. 
\begin{defi}\label{D:stat}
The type of a finite tuple $\bar a$ over a differential subfield $K$ of $\UU$ is \emph{stationary} if  the field extension $K\subset \str{K}{\bar a}$ is \emph{regular}, that is, if $ \str{K}{\bar a}$ and $K^{alg}$ are linearly disjoint over $K$. 
\end{defi}
The above definition does not depend on the choice of representative for the type of~$\bar a$ over~$K$. Indeed, given another tuple~$\bar a'$ with the same type as~$\bar a$ over~$K$, Corollary~\ref{C:univ} yields an automorphism~$\sigma$ in $\Aut_\delta(\UU/K)$ with $\bar a'=\sigma(\bar a)$.  Note that $\sigma$~permutes $K^{alg}$, so $ \str{K}{\bar a'}$ and $K^{alg}$ are linearly disjoint over~$k$ by invariance of linear disjointness. 

Types over algebraically closed differential subfields are  stationary, directly from the definition. 
\begin{rema}\label{R:stat}
	
	\begin{enumerate}[(a)]
		
		Consider a differential field extension $K\subset L$ and a tuple $\bar a$ in $\UU$. 
\item Transitivity of linear disjointness \parencite[Chapter VIII, Proposition 3.1]{Lang_Alg} yields that stationary types satisfy the {\bf Stationarity} principle of Remark~\ref{R:indep_prop2}.(a), that is, if the type of $\bar a$ over $K$ is stationary, whenever  $\str{K}{\bar a}\ind_K L$ and $\str{K}{\bar a'}\ind_K L$ with $\bar a'\equiv_K\bar a$, then $\bar a$ and $\bar a'$ have the same type over $L$. 

\item If $\bar a$ is independent from $L$ over $K$ and the type of $\bar a$ over $K$ is stationary, then  the type over $\bar a$ over $L$ is again stationary. 

\item If both the types of $\bar a$ and of the tuple $\bar b$ of $\UU$ over $K$ are stationary with $\bar a$ and $\bar b$ independent over $K$, then  the type of the tuple $(\bar a, \bar b)$ is stationary over $K$. 
	\end{enumerate}
\end{rema}
\begin{proof}
	
	For (a),  we may assume by {\bf Algebraic Closure}, that $L$ is algebraically closed. Using the transitivity of linear disjointness, we have that   $\str{K}{a}$ and $L$  are linearly disjoint over $K$, and likewise for $\str{K}{a'}$ and $L$, by Remark~\ref{R:indep_char}.(a). Hence, the partial isomorphism mapping $\bar a$ to $\bar a'$ and fixing $K$ pointwise extends to an isomorphism $\str{L}{\bar a}\to \str{L}{a'}$ fixing~$L$ pointwise. Thus it follows that $\bar a\equiv_L \bar a'$, as desired.

The proof of (b) is similar:  Using  {\bf Algebraic Closure}, we have that $\bar a$ is independent from $L^{alg}$ over $K$, so $\str{K}{\bar a}$ and $L^{alg}$ (and hence $L$) are linearly disjoint over $K$, by regularity of the extension  $K\subset \str{K}{a}$ (for the type of $\bar a$ over $K$ is stationary). Now, the compositum $\str{K}{a}\cdot L$ is a differential field which equals  $\str{L}{\bar a}$.  Using transitivity of linear disjointness, we conclude that the extension $L\subset \str{L}{\bar a}$ is regular, as desired. 

For (c), note that the independence of $\bar a$ and $\bar b$ over $K$ yields that $\str{K^{alg}}{\bar a}$ and $\str{K^{alg}}{\bar b}$ are linearly disjoint over $K^{alg}$. Notice that  $\str{K^{alg}}{\bar a}=K^{alg}\cdot \str{K}{\bar a}$ and similarly $\str{K^{alg}}{\bar b}=K^{alg}\cdot \str{K}{\bar b}$. Transitivity of linear disjointness, using that the  extension $K\subset \str{K}{\bar a}$ is regular, yields that $\str{K}{\bar a}$ and $\str{K^{alg}}{\bar b}=K^{alg}\cdot \str{K}{\bar b}$ are linearly disjoint over $K$. In particular, by  transitivity of linear disjointness, using now that the  extension $K\subset \str{K}{\bar b}$ is regular, we conclude that 
the differential field $\str{K}{\bar a, \bar b}=\str{K}{\bar a}\cdot \str{K}{\bar b}$ is linearly disjoint from $K^{alg}$ over $K$, so the extension $K\subset \str{K}{\bar a, \bar b}$ is regular. 
\end{proof}
The notion of a Morley sequence is ubiquitous in model theory. Whilst the definition of Morley sequences in the abstract setting requires the additional notion of \emph{indiscernibility}, this property will follow automatically from the {\bf Stationarity} principle in our context, see Remark~\ref{R:MS}.(c) below. 
\begin{defi}\label{D:MS}
Given a stationary type of a tuple $\bar a$ over a  differential subfield $K$ of $\UU$ and a  natural number $m\ge 1$, a \emph{Morley sequence} (of length $m$) of the type of $\bar a$ over $K$  is a sequence $(\bar a_1,\ldots, \bar a_m)$ of realizations of the type of $\bar a$ over $K$ (so $\bar a_i\equiv_K \bar a$) with \[ \str{K}{\bar a_i} \ind_{K} \str{K}{\bar a_1,\ldots, \bar a_{i-1}} \text{ for all $2\le i\le m$}.\]
\end{defi}
\begin{rema}\label{R:MS}
	Consider a  Morley sequence $(\bar a_1,\ldots, \bar a_m)$ of the stationary type of $\bar a$ over $K$. 
	\begin{enumerate}[(a)]
\item  For every permutation $\tau$ of $\{1,\ldots, m\}$, the sequence $(\bar a_{\tau(1)},\ldots, \bar a_{\tau(m)})$ is also a Morley sequence of the type of $\bar a$ over $K$,  by a straightforward application of {\bf Monotonicity \& Transitivity}.

\item The type of the tuple $(\bar a_1,\ldots, \bar a_m)$ over $K$ is again stationary, by Remark~\ref{R:stat}.(c). 
 
\item Consider a  differential field extension $L$ of $K$ with \[ \str{K}{\bar a_1,\ldots, \bar a_m} \ind_K L.\] The sequence   $(\bar a_1,\ldots, \bar a_m)$  is a Morley sequence of the type of $\bar a_1$ over $L$ (which is stationary by Remark~\ref{R:stat}), by {\bf Stationarity} as well as {\bf Monotonicity \& Transitivity}. 

In particular, for every $1\le i\le j\le m$, we have that $\bar a_i$ and $\bar a_j$ have the same type over $\str{K}{\bar a_1,\ldots, \bar a_{i-1}}$, since \[ \str{K}{\bar a_i, \bar a_j} \ind_{K} \str{K}{\bar a_1,\ldots, \bar a_{i-1}}.\] Thus, the subsequence $(\bar a_i, \ldots, \bar a_m)$ is a Morley sequence of the (stationary) type of $\bar a_i$ over $\str{K}{\bar a_1,\ldots, \bar a_{i-1}}$. 
\item Every stationary type admits a Morley sequence of length $m$ for every $m$ in $\N$ by {\bf Extension}~\ref{R:indep_prop2}.(b). Any two Morley sequences of the same length have the same type over the base subfield, again by {\bf Stationarity}. Thus,  we can map one Morley sequence of a given fixed length to another one of the same length by an automorphism fixing the base subfield. 
	\end{enumerate}
\end{rema}

In order to introduce the fundamental notion of the canonical base of a stationary type, we first recall Weil's \emph{field of definition} of a differential ideal. 
\begin{defi}\label{D:fielddef}
Given a finite tuple $\bar a$ and a differential subfield $K$ of $\UU$, the (differential) \emph{field of definition} of the type of $\bar a$ over $K$ is the smallest differential field $k\subset K$ such that the vanishing ideal $I(\bar a/K)$ admits generators over $k$, that is, with $I(\bar a/K)=I(a/k)\cdot K\{T\}$. Such a field exists and is fixed pointwise exactly by the automorphisms of $\UU$ which fix $I(\bar a/K)$ setwise (see \cite[Chapter VIII, \S 2, Theorem 7]{Lang_AlgGeo}  \& \cite[Theorem 2.3.5]{Tressl}). 
\end{defi}

\begin{prop}\label{P:CB}
Given a differential field $K$ and a tuple $\bar a$ in $\UU$ such that the type of $\bar a$ over $K$ is stationary, set the \emph{canonical base} $\cb(\bar a/K)$ to be the differential field of definition of the type of $\bar a$ over $K$.  The canonical base has the following properties:
\begin{enumerate}[(a)]
	\item The type of $\bar a$ over the differential subfield $k=\cb(\bar a/K)$ is again stationary and the field $\str{k}{\bar a}$ is linearly disjoint, and thus independent,  from $K$ over $k$. 
\item For every differential subfield $K_1$ of $K$, we have that $\str{K_1}{\bar a}$ and $K$ are linearly disjoint over $K_1$ if and only if $\cb(\bar a/K)\subset K_1$. In particular, it follows from {\bf Algebraic Closure} that 
\[ \str{K_1}{\bar a}\ind_{K_1} K \ \iff \ \cb(\bar a/K)\subset K_1^{alg}.\] 

	\item If $\bar a$~is independent from~$L$ over~$K$, where $L$~is a differential field extension of~$K$, then $\cb(\bar a/L)=\cb(\bar a/K)$ (Note that the type of~$\bar a$ over~$L$ is stationary by Remark~\ref{R:stat}.(b)).
		\item Canonical bases are finitely generated (as differential fields), that is, there exists a finite tuple $\bar b$ of $K$ such that $\cb(\bar a/K)=\str{\Q}{\bar b}$. 
\item There exists some $m$ in $\N$ such that the canonical base $\cb(\bar a/K)$  is algebraic over $\str{\Q}{\bar a_1,\ldots, \bar a_m}$ for every Morley sequence $(\bar a_1,\ldots, \bar a_m)$ of the type of $\bar a$ over $K$. 
\end{enumerate} 
\end{prop}

\begin{proof}
Set $k=\cb(\bar a/K)$. Remark~\ref{R:indep_char}.(b) yields that $\str{k}{\bar a}$ is linearly disjoint from $K$ over $k$ and thus  $\str{k}{\bar a}\ind_k K$, by Remark~\ref{R:indep_char}.(a). Since the extension 
$K\subset \str{K}{\bar a}=\str{k}{\bar a}\cdot K$ is regular,  transitivity of linear disjointness gives that $k\subset \str{k}{\bar a}$ is also regular, which gives  (a).

 We deduce (b) in a similar fashion. For the first part, use Remark~\ref{R:indep_char}  and the definitional property of the field of definition. The last part follows from {\bf Algebraic Closure}.

In order to prove (c), assume that $K\subset L$ is a differential field extension with $\str{K}{\bar a}\ind_K L$. By {\bf Algebraic Closure}, using Remark~\ref{R:indep_char}.(a), transitivity of linear disjointness  and that the extension $K\subset \str{K}{\bar a}$ is regular, we deduce that $\str{K}{\bar a}$ and $L^{alg}$, and thus $L$, are linearly disjoint over $K$. Again by transitivity of linear disjointness, the fields $\str{k}{\bar a}$ and $L$ are linearly disjoint over $k$, which yields that $k_1=\cb(\bar a/L)$ is contained in $\cb(\bar a/K)$. The other inclusion is immediate using that $\str{k_1}{\bar a}$ and $L$ are linearly disjoint over $k_1\subset K$, so $\str{k_1}{\bar a}$ and $K$ are also linearly disjoint over $k_1$, which yields that $k\subset k_1$ by part (b) above. 

For (d),  we need only find some finite tuple $\bar b$ in $k=\cb(\bar a/K)$ such that $\str{\Q}{\bar b, \bar a}$ is linearly disjoint from $k$ over $\str{\Q}{\bar b}$ (using (b)). Split the tuple $\bar a=\bar a_1\bar a_2$, where $\bar a_1$ is a maximal subtuple of $\bar a$ consisting of differential transcendental elements and independent over $k$ (as in Definition~\ref{D:difftr_alg}), whereas each coordinate in the tuple $\bar a_2$ is differentially algebraic over $\str{k}{\bar a_1}$.

 For a tuple $\bar d$ of $\UU$,  its \emph{$n^{th}$-prolongation} is the tuple $\nabla_n(\bar d)=(\bar d, \delta(\bar d),\ldots, \delta^n(\bar d))$. Choosing a suitable integer $n$, the  prolongation $\nabla_n(\bar a)$ witnesses all differential algebraicities of $\bar a_2$ over the field  $k(\nabla_n(\bar a_1))$ (without taking further derivatives). Remark~\ref{R:extension_diff}.(a) yields that $\str{k}{\bar a}=k(\nabla_n(\bar a))(\{\nabla_m(\bar a_1)\}_{m\ge n})$. 
The (algebraic) vanishing ideal of $\nabla_n(\bar a)$ over $k$ has a field of definition $\Q(\bar b)$ for some finite tuple $\bar b$ of $k$ (by Hilbert's Nullstellensatz).  In particular, the fields $\str{\Q}{\bar b}(\nabla_n(\bar a))$ and $k$ are linearly disjoint over $\str{\Q}{\bar b}$, as in part (a) of this proof. Since the family $(\nabla(\bar a_1))_{m\ge n}$ consists of  algebraically independent elements over $k$, we deduce from \parencite[Chapter VIII, \S 3, Proposition 3.3]{Lang_Alg}  that the fields $\str{\Q}{\bar b, \bar a}=\str{\Q}{\bar b}(\nabla_n(\bar a), \{\nabla_m(\bar a_1)\}_{m\ge n})$ and $k$ are linearly disjoint over $\str{\Q}{\bar b}$, as desired. 

For the last item (e), set $\bar b$ a tuple of generators of $\cb(\bar a/K)$, by part (d) above.  Replace the tuple $\bar a$ with a suitable prolongation $\nabla_n(\bar a)$ as before, so every further derivative is already algebraic (in the field sense) over $K(\nabla_n(\bar a))$. We need only show that the field of definition $\Q(\bar b)$  (in the algebraic sense) of our new tuple $\bar a$ over $K$ lies in the algebraic closure of the field generated by $m$-many independent realizations $\bar a_1,\ldots, \bar a_m$ of the type of $\bar a$ (or rather of $\nabla_n(\bar a)$) over $K$, where the integer $m=\tr(\bar b)$. Assume otherwise, so for $k\le m$, we have that the tuple $\bar b$  does not lie in $\Q(\bar a_1,\ldots, \bar a_k)^{alg}$ (Note that we are working in classical commutative algebra without taking differential fields). 

Since the sequence $\bar a_1,\ldots, \bar a_m$ is independent over $K$, we have that \[ K(\bar a_k)\ind_K K(\bar a_1,\ldots, \bar a_{k-1})  \text{ for all } k\le m,\] so \[\tag*{(1)} \Q(\bar b, \bar a_k)\ind_{\Q(\bar b)}  \Q(\bar b, \bar a_1,\ldots, \bar a_{k-1}) \text{ for all } k\le m,\] by {\bf Monotonicity \& Transitivity}, for $\bar a_k$ and $\bar a$ have the same type over $K$. Thus, the field $\Q(\bar b)$ is also the field of definition of the vanishing ideal of $\bar a_k$ over $\Q(\bar b, \bar a_1,\ldots, \bar a_{k-1})$, similarly as in the proof of (c)  of this Proposition. If we have  for some $1\le k\le m$ that  \[ \Q(\bar b, \bar a_k)\ind_{ \Q(\bar a_1,\ldots, \bar a_{k-1})}  \Q(\bar b, \bar a_1,\ldots, \bar a_{k-1}),\] it would follow from  part (b) that $\bar b$ is algebraic over $\Q(\bar a_1,\ldots, \bar a_{k-1})$, contradicting our assumption. Thus, the algebraic dependence \[\tag*{(2)} \Q(\bar b, \bar a_k)\nind_{ \Q(\bar a_1,\ldots, \bar a_{k-1})}  \Q(\bar b, \bar a_1,\ldots, \bar a_{k-1}) \text{ for all } k\le m\] yields that for all $1\le k\le m$ \begin{multline*}\tag*{(3)} s=\tr(\bar a_k/K)  \tr(\bar a_k/\Q(\bar b)) \stackrel{(1)}{=} \\ = \tr(\bar a_k/\Q(\bar b, \bar a_1,\ldots, \bar a_{k-1})) \stackrel{(2)}{<} \tr(\bar a_k/\Q( \bar a_1,\ldots, \bar a_{k-1})), \end{multline*} and thus 
\[ \tag*{(4)} s+1 \le \tr(\bar a_k/\Q(  \bar a_1,\ldots, \bar a_{k-1})) \text{ for all } k\le m.\]
The following computation
\begin{multline*}m(s+1)=m+m\cdot s=\tr(\bar b) + \sum_{k=1}^m s \stackrel{(3)}{=} \tr(\bar b) + \sum_{k=1}^m \tr(\bar a_k/\Q(\bar b, \{\bar a_j\}_{j<k}))= \\ = \tr(\bar b, \{\bar a_k\}_{k\le m}) = \tr(\{\bar a_k\}_{k\le m})+ \tr(\bar b/\Q(\{\bar a_k\}_{k\le m}) = \\ = \sum_{k=1}^{m} \tr(\bar a_k/\Q(\{\bar a_j\}_{j< k})) + \tr(\bar b/\Q(\{\bar a_k\}_{k\le m}) \stackrel{(4)}{\ge} \\ \sum_{k=1}^m s+1 + \tr(\bar b/\Q(\{\bar a_k\}_{k\le m}) \ge  m(s+1)+ 1 > m(s+1)
   \end{multline*}
yields the desired contradiction. 
\end{proof}
\begin{rema}\label{R:CB_dcl}
It is not difficult to show that for some $m$ in $\N$ the canonical base $\cb(\bar a/K)$ of $\bar a$ over $K$ belongs to the differential field $\str{\Q}{\bar a_1,\ldots, \bar a_m}$ for every Morley sequence  $(\bar a_1,\ldots, \bar a_m)$. Indeed, it suffices to show this for the coefficients of the field of definition of the ideal in the classical algebraic sense. After choosing a suitable enumeration of the monomials, every realization $\bar a_i$ induces a linear dependence for the coefficients. Choosing the prolongation index $n$ sufficiently large, this translates into an invertible matrix, whose entries can be computed using Cramer's rule. 
\end{rema}
Proposition~\ref{P:CB}.(a) and (d) yield immediately the following result: 
\begin{coro}\label{C:local_char}
The independence notion for the theory of differentially closed fields introduced in Definition~\ref{D:indep} satisfies a \emph{strong} version of {\bf Local Character}: For every finite tuple $\bar a$ of $\UU$ and every differential subfield $K$ of $\UU$, there exists a \emph{finite} tuple $\bar b$ of $K$ such that $\bar a$ is independent from $K$ over the differential field generated by $\bar b$. \qed
\end{coro}
As shown by \textcite{KimPillay} (see also the work of \textcite{HarnikHarrington} for stable theories), the independence notion of Definition~\ref{D:indep} is canonical and coincides with  \emph{non-forking independence} as introduced by \textcite{Shelah}. 

\section{Geometric stability and minimality}\label{S:Intern}

The purpose of this section is to introduce \emph{minimal} types, which are the building blocks to analyze types (of finite rank) in differentially closed fields. Zilber's trichotomy principle (Fact~\ref{F:dichotomy}) for minimal types as well as the semi-minimal analysis (Proposition~\ref{P:nonorth_internal})  will play a fundamental role in the proof of \textcite{FJM22} in Section~\ref{S:D2}. 

By Proposition~\ref{P:univ}, we work inside an ambient universal $\aleph_1$-saturated differentially closed field $\UU$. All finite tuples and proper subfields are taken within $\UU$ and are assumed to be countable, unless explicitly stated. 

\begin{defi}\label{D:fterank}
The type of a finite tuple $\bar a$ over a differential subfield $K$ of $\UU$ has \emph{finite rank} if every coordinate of $\bar a$ is differentially algebraic over $K$ (see Definition~\ref{D:difftr_alg}), or equivalently, if the differential field $\str{K}{\bar a}$ has finite transcendence degree over $K$, by Corollary~\ref{C:diffrank_trdeg}. 
\end{defi}
The above definition does not depend on the representative of the type of $\bar a$ over $K$, by Corollary~\ref{C:univ}.

\begin{rema}\label{R:fterank}
	\begin{enumerate}[(a)]
\item If  $\bar a$ has finite rank over $K$, then so does $\bar b$, whenever $\bar b$ belongs to $\str{K}{\bar a}^{alg}$. 

\item  If  $\bar a$ has finite rank over $K$ and $K\subset L$ is a differential field extension, then  $\bar a$ 
has also finite rank over $L$. Moreover, we have that \[ \str{K}{\bar a} \ind_K L \ \iff \ \tr(\str{K}{\bar a}/K)=\tr(\str{L}{\bar a}/L).\]  
Therefore, we will now define the \emph{Lascar} or \emph{U-rank} of the type of $\bar a$ over $K$, denoted by $\U(\bar a/K)$, as \emph{the number  of times} that  $\bar a$  can become dependent from some differential field extension $L$ of $K$, that is, we set $\U(\bar a/K)\ge n+1$ if there is some differential field extension $L$ of $K$ with \[ \U(\bar a/L)\ge n \ \text{ and} \ 
\str{K}{\bar a} \nind_K L,\] whereas the inequality $\U(\bar a/K)\ge 0$ always holds. Note that $\U(\bar a/K)\le \tr(\str{K}{\bar a}/K)$, so the U-rank $\U(\bar a/K)$  is a well-defined positive integer whenever $\bar a$ has finite rank over $K$. However, equality need not hold: if the element $a$ is differentially algebraic (and hence of finite rank) over $K$, its U-rank can be strictly smaller than the differential order of any minimal differential polynomial of $a$ over $K$. For example,  it follows from Corollary~\ref{C:sm}, that the \emph{generic type} (as in Remark~\ref{R:D2_generictype}) of Poizat's equation $\delta^2(T)T=\delta(T)$, which already appeared in the Introduction,  has U-rank $1$, but differential order $2$. 

It is not difficult to see that U-rank \emph{witnesses independence}: whenever $L$ is a differential field extension of $K$, we have that \[ \U(\bar a/L)\le \U(\bar a/K) \ \text{ and  moreover, } \ \U(\bar a/L)=\U(\bar a/K) \ \iff \  \str{K}{\bar a} \ind_K L.\]  Therefore, we have that $\U(\bar a/K)=0$ if and only if $\bar a$ is algebraic over $K$.  Furthermore, if the tuple $\bar b$ belongs to $\str{K}{\bar a}^{alg}$, {\bf Monotonicity \& Transitivity} yield  the following special case of \emph{Lascar's inequalities} \parencite[Theorem 8]{Lascar}:
\[ \U(\bar a/K)=\U(\bar b/K)+ \U(\bar a/\str{K}{\bar b})\ge \U(\bar b/K).\]
\end{enumerate}
\end{rema}

\begin{defi}\label{D:minimaltype}
Given a finite tuple $\bar a$ over a differential subfield $K$ of $\UU$, we say that the type of $\bar a$ over $K$ is: 
\begin{itemize}
	\item \emph{non-algebraic} if at least one coordinate of $\bar a$ is transcendental over $K$, or equivalently, if $\U(\bar a/K)\ne 0$, by Remark~\ref{R:fterank}.(b);
\item \emph{minimal} if it is non-algebraic and for every differential field extension $L$ of $K$, we have that $\bar a$ is either algebraic over $L$ or independent from $L$ over $K$. Notice that this does not depend on the representative for the type of $\bar a$ over $K$, by Corollary~\ref{C:univ}. 
\end{itemize}
\end{defi}

\begin{rema}\label{R:minimaltypes}
\begin{enumerate}[(a)]
	
	\item If the type of $\bar a$ over $K$ is minimal and the tuple $\bar b$ belongs to $\str{K}{\bar a}^{alg}$, either $\bar b$ belongs to $K^{alg}$ or the tuples $\bar a$ and $\bar b$ are \emph{interalgebraic} over $K$, that is $\str{K}{\bar a}^{alg}=\str{K}{\bar b}^{alg}$: Indeed, if $\bar b$ is not algebraic over $K$, then $\bar a$ is not independent from $\str{K}{\bar b}$ over $K$ by {\bf Algebraic Closure}. Minimality of the type now yields that  $\bar a$  belongs to $\str{K}{\bar b}^{alg}$, as desired.  
	
	If $\bar a$ and $\bar b$ are interalgebraic over $K$, then the type of $\bar b$ over $K$ is minimal whenever the type of $\bar a$ over $K$ is.
	
	\item Minimal types have finite rank:  Assume for a contradiction that  one coordinate $b$ of the minimal type of $\bar a$ over $K$ is differentially transcendental over $K$. Consider the differential field extension $K\subset L=\str{K}{\delta(b)}$. Clearly, the tuple $\bar a$ is not algebraic over $L$, as $b$ does not belong to $\str{K}{\delta(b)}^{alg}$. By minimality, we deduce that $\bar a$ is independent from $L$ over $K$ and hence so is $\delta(b)$. By {\bf Algebraic Closure}, we conclude that $\delta(b)$ is algebraic over $K$, which gives the desired contradiction. 
	
	\item It follows directly from the definitions (using Remark~\ref{R:fterank}.(b)) that a type of finite rank is minimal if and only if it has U-rank $1$.
\end{enumerate}
\end{rema}
Zilber's trichotomy principle establishes that minimal types can be divided, up to \emph{non-orthogonality} (see Definition~\ref{D:orth}), into three categories, one of which consists of the (unique) minimal type  given by transcendental constant elements. 
\begin{exem}\label{E:constants}
It follows from Example~\ref{E:types}.(a), with $k=\Q$  and $P(T)=\delta(T)$, that there are transcendental constant elements $\CC_{\UU}$. Any two such elements have the same type over $\Q$ by Remark~\ref{R:iso_diff}, so we will refer to the type of any transcendental constant element $c$ as \emph{the type of the constants}. Observe that this type is clearly stationary, since $\str{\Q}{c}=\Q(c)$ is a purely transcendental extension of $\Q$,  and hence regular. 

Given a realization $c$ of the type of the constants, if $c$ is transcendental over the differential field $K$, or equivalently, if $c$ is independent from $K$ over $\Q$, we have that the type of $c$ over $K$ is again stationary by Remark~\ref{R:stat}.(b). We will refer to this (unique) type as the \emph{the type of the constants over $K$}. 

For every differential field $K$, the type of the constants  over $K$ is minimal: It is clearly non-algebraic. Consider now a constant element $c$ transcendental over $K$ and a differential extension $K\subset L$ such that   $c$ is not independent from $L$ over $K$. We need to show that  $c$ belongs to $L^{alg}$, which follows  immediately from Remark~\ref{R:fterank}.(b), since the differential field $\str{K}{c}=K(c)$ has transcendence degree $1$ over $K$. 
\end{exem}
Another relevant class of minimal types in the trichotomy principle is given by those types with \emph{trivial geometry}, as explained below. Triviality roughly says that three realizations are independent whenever they are pairwise independent. 
\begin{defi}\label{D:trivial}
The stationary type of a finite tuple~$\bar a$ over the differential subfield~$K$ of~$\UU$ is \emph{trivial} if 
 \[ \str{K_1}{\bar a_{3}} \ind_{K_1} \str{K_1}{\bar a_{1}, \bar a_{2}}\] 
for every differential field extension $K_1$ of $K$ and realizations $\bar a_1, \bar a_2$ and $\bar a_3$ of the type of $\bar a$ over $K$ realizing the following conditions:
\begin{itemize}
\item each $\bar a_i$ is independent from $K_1$ over $K$;
\item For $i\ne j$ in $\{1, 2, 3\}$, the tuples $\bar a_i$ and $\bar a_j$ are independent over $K_1$.
\end{itemize} 
\end{defi}
\begin{rema}\label{R:trivial}
	As shown by  \textcite[Lemma 1]{Goode} (the  \emph{alter ego} of a well-known model theorist), whenever the stationary type of $\bar a$ over $K$ is trivial, then every  sequence $\bar a_1,\ldots, \bar a_m$ of realizations of $\bar a$ over $K$ satisfying the following two conditions: 
	\begin{itemize}
		\item $\str{K}{\bar a_i}\ind_K K_1$ for each $1\le i\le m$;
		\item $\str{K_1}{\bar a_i}\ind_{K_1} \str{K_1}{\bar a_j}$ for $1\le i\ne j\le m$;
	\end{itemize} is a Morley sequence of the type of $\bar a$ over $K_1$, that is, \[ \str{K}{\bar a_i} \ind_{K} \str{K}{\bar a_1,\ldots, \bar a_{i-1}} \text{ for all $2\le i\le m$}.\]
\end{rema}

\begin{rema}\label{R:gp_nottrivial}
	\begin{enumerate}[(a)]
\item A \emph{definable group} is given by a Kolchin constructible set $G$ in some cartesian product of $\UU$ equipped with a group law $\cdot$ such that the map (or rather its graph) \[ \begin{array}{ccc}
G\times G &\to& G\\[1mm]
(x,y) &\to& x\cdot y\inv
\end{array}\] is Kolchin constructible. Archetypal examples of definable groups are differential algebraic groups (where the group law is given by differential polynomials, or rather by differential rational functions). As shown by \textcite[Theorem 21]{Pillay_chunks}, the category of definable groups and the category of differential algebraic groups (in differentially closed fields) are equivalent. Moreover, definable groups definably embed into algebraic groups, as shown by \textcite[Corollary 4.2]{Pillay_gpsDCF} (see also the work of \textcite[Theorem 4.1]{KowalskiPillay_gps}). 

\item A \emph{definable action} of a definable group $G$  on a Kolchin constructible set $X$ is  a group action \[ \begin{array}{ccc}
	G\times X &\to& X\\[1mm]
	(g,x) &\to& g\star x
\end{array}\]
whose graph is Kolchin constructible. The action is \emph{defined} over a given differential subfield $K$ of $\UU$, if $G$, $X$ and the  graph of the group action are all defined over  $K$. 

\item 
Assume that the definable group~$G$ is defined over the differential
subfield~$K$ of~$\UU$. A Kolchin constructible subset of~$G$ is \emph{generic}
(though model-theorists should definitely use the terminology \emph{syndetic})
if finitely many translates cover~$G$. An element~$g$ of~$G$ is \emph{generic}
over~$K$ if it only lies in generic Kolchin constructible subsets defined
over~$K$.  Genericity is a property of the type of an element and generic elements (or generic types) exist \parencite[Chapter 2, \S 1]{Poizat_stable}. Generic elements  are exactly the  generic elements in the sense of the Kolchin topology.  If the definable group~$G$  is defined over~$K$ and all elements of~$G$ have finite rank over~$K$, then an element~$g$ of~$G$ is generic if and only if its U-rank $\U(g/K)$ is maximal among the rank of elements of~$G$. Thus, if $G$~is infinite, then no generic element is algebraic over~$K$.

Assume now that $G$ is defined over $K$ and that all elements of $G$ have finite rank over $K$. It is easy to see that, if two elements $g$ and $g_1$ of $G$  are generic and independent over $K$, then their product $g\cdot g_1$ is again generic  and independent from each factor over $K$.  

\item 
A definable group~$G$ defined over~$K$ is \emph{connected} if it has no proper
definable subgroup of finite index.  Every definable group has a
\emph{connected component}, which is again a definable subgroup of~$G$ defined
over $K^{alg}$. Any two generic elements over~$K$ (or equivalently over $K^{alg}$) in the connected component have the same type over~$K$.  

\item Two tuples $\bar a$ and $\bar b$ are \emph{interalgebraic} over a differential subfield $K$ if $\str{K}{\bar a}^{alg}=\str{K}{\bar b}^{alg}$. 

A generic element $g$ of an infinite group $G$ defined over $K$ cannot be interalgebraic over $K$ with a tuple $\bar a$ whose type over $K$ is trivial. Assume otherwise for a contradiction.  Set $\bar a_1=\bar a$ and $g_1=g$. It is easy to show that we may assume the base field $K$ to be algebraically closed and that the generic element belongs to  the connected component of $G$ (since the latter has finite index in $G$).  By {\bf Extension} and Corollary~\ref{C:univ} we can find another generic $g_2$, independent from $g$ over $K$. Since the connnected component only has one generic type, we conclude that such that $g_2$ must also be interalgebraic over $K$ with a tuple $\bar a_2$ with $\bar a_2\equiv_K\bar a$. The product  $g_3=g_1\cdot g_2$ is again an element of the connected component and is generic over $K$. Thus, by automorphisms (Corollary~\ref{C:univ}),  we find a  realization $\bar a_3$ of the type of $\bar a$ over $K$ such that $g_1\cdot g_2$ is interalgebraic with $\bar a_3$.  Since the product is independent from each factor, we deduce that $\bar a_i$ 
and $\bar a_j$ are independent over $K$ for $i\ne j$.  Since the type of $\bar a$ is trivial, we have that $\str{K}{\bar a_3}$ is independent from $\str{K}{\bar a_1, \bar a_2}$ over~$K$. 
{\bf Algebraic Closure} gives now the desired contradiction, since the product $g_3=g_1\cdot g_2$  is not independent from  $\str{K}{g_1, g_2}$ over $K$.
	\end{enumerate}
\end{rema}
It follows from the above remark that the type of the constants is not trivial. Theorem~\ref{T:Main} will provide a criterion to determine when the type of a differentially algebraic element is trivial of \emph{degenerated geometry}.

Trivial types are particular examples of \emph{$1$-based types}, for which we can choose  $m=1$ in Proposition~\ref{P:CB}.(e). 
\begin{defi}\label{D:1based}
	Consider a differential subfield $K$ of $\UU$ and a finite tuple $\bar a$ such that the type  of $\bar a$ over $K$ is stationary. This type is \emph{$1$-based} if for every differential field extension $L$ of $K$ and every finite tuple $\bar b=(\bar a_1,\ldots, \bar a_n)$ consisting of realizations $\bar a_i$ of the type of $\bar a$ over $K$, 
	we have that  the canonical base $\cb(\bar b/L^{alg})$ is algebraic over $\str{K}{\bar b}$. 
	
	In particular, a single realization of the type of $\bar a$ over $K$ suffices to determine its canonical base, that is,  we have that $\cb(\bar a/L^{alg})$ is a differential subfield of $\str{K}{\bar a}^{alg}\cap L^{alg}$. 
\end{defi}

It follows directly from the definition that if the type of $\bar a$ over $K$ is $1$-based, then so is the type of $\bar a$ over $L$ for every differential field extension $L$ of $K$. 

The type of the constants of Example~\ref{E:constants} is not $1$-based: Indeed, by the above discussion, it suffices to show this working over the prime field $\Q$.  Choose algebraically independent transcendental constant elements $a$, $b$ and $c$ and set $d=ac+b$. The \emph{generic} point $(c, d)$ belongs to the \emph{generic} line given by the equation $y=ax+b$. In this particular case, the differential field $\str{\Q}{a, b}=\Q(a, b)$, so it is not difficult to see, along the lines of the proof of \textcite[Lemma 3.1]{Pillay_CMtrivial}, that 
$\str{\Q}{a, b}^{alg}\cap \str{\Q}{c, d}^{alg}=\Q^{alg}$, witnessing the failure of $1$-basedness. Indeed, if the type of $(a, b)$ were $1$-based over $\Q$, we would have that $\cb(a, b/\str{\Q}{c, d}^{alg})$ would be contained in \[\str{\Q}{a, b}^{alg}\cap \str{\Q}{c, d}^{alg}=\Q^{alg},\] yet \[ \str{\Q}{a, b}\nind_\Q \str{\Q}{c, d}.\]
Whilst the type of the constants is not $1$-based, trivial types are. 
\begin{lemm}\label{L:trivial_1based}
Trivial minimal stationary types are $1$-based. 
\end{lemm}
\begin{proof}
Assume that the minimal stationary type of $\bar a$ over $K$  is trivial, as in Definition~\ref{D:trivial}. In order to show that it is $1$-based, consider a differential field extension $L=L^{alg}$ of $K$ as well as a tuple $\bar b$ consisting of realizations $\bar a_i$ of the type of $\bar a$ over~$K$.  We want to show that $\cb(\bar b/L)$ is contained in the subfield $L_1=L\cap \str{K}{\bar b}^{alg}$.  By Proposition~\ref{P:CB}.(e), the canonical base $\cb(\bar b/L)$ is algebraic over $\str{\Q}{\bar b_1,\ldots, \bar b_m}$, for a Morley sequence $(\bar b_i)_{i\le m}$ of realizations of the type of $\bar b$ over $L$. In particular, the sequence is independent over $L$. Extend the Morley sequence to  $(\bar b_i)_{i\le m+1}$ by {\bf Extension} and notice that $L_1=L\cap \str{K}{\bar b_{j}}^{alg}$ for $1\le j\le m+1$,  as $\bar b$ and $\bar b_{j}$ have the same type over $L$.   

The independence \[ \str{L}{b_{m+1}}\ind_L \str{L}{\bar b_1,\ldots, \bar b_m}\] yields that $\cb(\bar b/L)=\cb(\bar b_{m+1}/L)\stackrel{\ref{P:CB}  \, (c)}{=}\cb(\bar b_{m+1}/\str{L}{\bar b_1,\ldots, b_{m}})$.  By Proposition~\ref{P:CB}.(b), we need only show that \[ \str{L_1}{\bar b_{m+1}}\ind_{L_1} \str{L}{\bar b_1,\ldots, \bar b_{m}} \tag*{($\lozenge$)},\] since $L_1\subset \str{L}{\bar b_1,\ldots, \bar b_{m}}$. 

Choose now a maximal subtuple $\bar b'_1=(\bar a_{i_1},\ldots, \bar a_{i_\ell})$ of $\bar b_1$ such that \[ \tag*{($\star$)} \str{K}{\bar a_{i_j}} \ind_{K} \str{L_1}{(\bar a_{i_s})_{s<j}} \ \text{ for all } j\le \ell.\] Since the type of  $\bar a$ over $K$ is minimal (that is, of Lascar rank $1$), it follows that $\bar b_1$ belongs to $\str{L_1}{\bar b'_1}^{alg}$, by maximality of the subtuple $\bar b'_1$. Now, the tuple $\bar b_i$, with $2\le i\le m+1$, has the same type over $L$ (and thus over $L_1$) as $\bar b_1$, so for $2\le i\le m+1$ the corresponding subtuple $\bar b'_i$ is independent from $L_1$ over $K$ and $\bar b_i$ belongs to $\str{L_1}{\bar b'_i}^{alg}$. In particular, the independence ($\lozenge$) will follow from {\bf Algebraic Closure}, once we show the following: \[ \str{L_1}{\bar b'_{m+1}}\ind_{L_1} \str{L}{\bar b'_1,\ldots, \bar b'_{m}} \tag*{($\blacklozenge$)}.\]
By construction, each tuple $\bar b'_i$ is an independent sequence over $L_1$ of realizations of the trivial type of $\bar a$ over $K$. Hence,  the independence ($\blacklozenge$) follows from Remark~\ref{R:trivial} if any two $\bar b'_i$ and $\bar b'_j$ with $i\ne j$ are independent over $L_1$. Again by triviality of the type of $\bar a$ over $K$, this amounts to showing that every two coordinates $\bar a_{s(i)}$ of $\bar b'_i$ and $\bar a_{t(j)}$ of $\bar b'_j$ are independent over $L_1$. Assuming otherwise for a contradiction, we have that $\bar a_{s(i)}$ must be algebraic over $\str{L_1}{\bar a_{t(j)}}\subset \str{L}{\bar b_j}$,  since \[ 1=\U(\bar a/K)=\U(\bar a_{s(i)}/K)\stackrel{(\star)}{=}\U(\bar a_{s(i)}/L_1).\]  Now, the Morley sequence $(\bar b_\ell)_{\ell\le m+1}$ is independent over $L$, so \[ \str{L}{\bar b_i} \ind_L \str{L}{\bar b_j}.\]

By {\bf Algebraic Closure}, the coordinate $ \bar a_{s(i)}$ of $\bar b_i$ must belong to the algebraically closed differential field $L$. As the tuples $\bar b_i$ and $\bar b_j$ have the same type over $L$, we deduce that the tuple  $\bar a_{s(i)}$ of $L$ also occurs in $\bar b_j$. Thus, the coordinate $\bar a_{s(i)}$ lies in $L\cap \str{K}{\bar b_{j}}^{alg}=L_1$, contradicting the independence ($\star$). 
\end{proof}
 \emph{Manin kernels} are the third possibility in Zilber's trichotomy (Fact~\ref{F:dichotomy}) for minimal types in differentially closed fields. We will provide a succinct account of Manin kernels, some of which appeared in the unpublished results of \cite{Hr_Sok}. For an algebraic description of Manin's construction, see the work of \textcite{Bertrand}. For a model-theoretic presentation, see the corresponding chapter of \textcite[Section 4]{Pillay_numberdiff} as well as the notes of \textcite[Sections  4 \& 5]{Marker_Manin}. 

\begin{exem}\label{E:Manin}
Consider an algebraically closed differential subfield $K$ of $\UU$ as well as a simple abelian variety $A$ defined over $K$ of dimension $d$. The \emph{universal extension} $\widehat A$ of $A$ by a vector group is an algebraic group of dimension $2d$, which we can see as a definable group (see Remark~\ref{R:gp_nottrivial}.(a)) in our ambient differentially closed field $\UU$. The \emph{Manin kernel} $A^{\#}$ of $A$ is the Kolchin closure (see Corollary~\ref{C:Kolchintop}) of the torsion points of $\widehat{A}$. Hence, the group $A^{\#}$ is an infinite differential algebraic group defined over $K$. As in Remark~\ref{R:gp_nottrivial}, we can consider generic types in the definable group $A^{\#}$. 

If $A$ does not \emph{descend} to the field of constants $\CC_\UU$, that is, if $A$ is not birationally isomorphic to an algebraic group defined over $\CC_\UU$, then any two generic elements over $K$ have the same type, so the generic type is unique and stationary (note that we assume $K$ to be algebraically closed). In this case, the generic type of the Manin kernel is \emph{$1$-based}  \parencite[Corollary 4.9 \& Theorem 4.10]{Pillay_numberdiff}, yet not trivial by Remark~\ref{R:gp_nottrivial}.(b). 
\end{exem}

We have hence introduced the three classes of types which will appear in Zilber's trichotomy (Fact~\ref{F:dichotomy}), up to \emph{non-orthogonality}. This notion coincides with the more general notion of non-indifference for minimal types, and plays a fundamental role in many results in geometric stability theory, from Hrushovski's analysis of a stable group to applications of geometric stability theory to problems in diophantine geometry \parencite{Bouscaren_book}. 

\begin{defi}\label{D:orth}
	Consider two differential fields $K$ and $L$ as well as two tuples $\bar a$ and $\bar b$. The stationary type of  $\bar a$ over $K$ is \emph{not indifferent} (often called \emph{not foreign} in the literature) to the type of $\bar b$ over $L$ if there exists a common differential field extension $M\supset K\cup L$ and realizations $\bar a'\equiv_K \bar a$ and $\bar b'\equiv_L \bar b$ with \[\str{K}{\bar a'}\ind_K M \ \text{ but } \ \str{M}{\bar a'}\nind_M \str{M}{\bar b'}.\]  
	
	 If both the   type of $\bar a$ over $K$ and of $\bar b$ over $L$ are stationary, then we say that these two types are \emph{non-orthogonal} if there exists a differential field extension $M\supset K\cup L$ and realizations $\bar a'\equiv_K \bar a$ and $\bar b'\equiv_L \bar b$ with \[\str{K}{\bar a'}\ind_K M, \ \str{L}{\bar b'}\ind_L M \ \text{ but } \ \str{M}{\bar a'}\nind_M \str{M}{\bar b'}.\] 
\end{defi}
Observe that non-orthogonality is a symmetric notion, whilst non-indifference need not be. However, these two notions coincide if the  type of  $\bar b$ over $L$ is minimal, by Remark~\ref{R:fterank}. Indeed, as the independence $\str{M}{\bar a'}\ind_M M$ always holds, the tuple $\bar b'$ cannot be algebraic over $M$, so $\str{L}{\bar b'}\ind_L M$, since \[0<\U(\bar b'/M)\le \U(\bar b'/L)=\U(\bar b/L)=1.\]
\begin{rema}\label{R:non_orth_minimal}
With the previous notation, if the two minimal stationary types of~$\bar a$ over~$K$ and of~$\bar b$ over~$L$ are non-orthogonal, witnessed by~$M$, $\bar a'$ and $\bar b'$ as in the above definition, then $\bar a'$ and $\bar b'$ are \emph{interalgebraic} over~$M$,  that is $\str{M}{\bar a'}^{alg}=\str{M}{\bar b'}^{alg}$, as in Remark~\ref{R:gp_nottrivial}.(e). Indeed, the tuples~$\bar a'$ and~$\bar b'$ are not independent over~$M$, so by Remark~\ref{R:fterank} \[0\le  \U(\bar a'/\str{M}{\bar b'})<\U(\bar a'/M)=\U(\bar a'/K)=\U(\bar a/K)=1.\] 
\end{rema}
The above remark together with an easy application of {\bf Extension} and {\bf Monotonicity \& Transitivity} yields the following result below for  a minimal stationary type. Indeed, it suffices to note that triviality is preserved under \emph{independent restrictions}, that is, if $K\subset L$ is a differential field extension and $\bar a$~is a   tuple which is independent from~$L$ over~$K$, whenever  the type of~$\bar a$ over~$L$ is trivial, then so is the type of~$\bar a$ over~$K$. 
\begin{coro}\label{C:non_orth_trivial}
Consider differential fields $K$ and $L$ as well as finite tuples $\bar a$ and $\bar b$ such that both the type of $\bar a$ over $K$ and the type of $\bar b$ over $L$  are minimal and stationary. If these two types are non-orthogonal and the type of $\bar b$ over $L$ is trivial, then the type of $\bar a$ over $K$ is also trivial. 

In particular, if the generic type of an infinite Kolchin constructible group $G$ is minimal, then it is orthogonal to all trivial types, by Remark~\ref{R:gp_nottrivial}.(e). \qed
\end{coro}

The work of \textcite{Zilber_sm, Zilber_trichotomy} on the structure of strongly minimal theories motivated Zilber's \emph{trichotomy principle}. Whilst the general principle does not hold for all  minimal types in arbitrary stable theories, as shown by \textcite{Hr_newsm} in the so-called \emph{ab initio} counter-example, the trichotomy principle has been shown to hold for (strongly) minimal types in various relevant classes of theories (see among others the work of \textcite{Hrushovski_Zilber_Zar1, Hrushovski_Zilber_Zar2} as well as the work of \textcite{PillayZiegler}). 

We will now state a  particular version of the trichotomy principle customized for the theory DCF$_0$ of differentially closed fields of characteristic $0$. It  first appeared in unpublished work by \textcite{Hr_Sok}. 
\begin{fait}\label{F:dichotomy}
	Consider a finite tuple $\bar a$  and a differential subfield $K$ of $\UU$ such that the type of $\bar a$ over $K$ is stationary and minimal. Exactly one of the three following mutually exclusive cases holds:
	\begin{itemize}
		\item The type of $\bar a$ over $K$ is trivial.
		\item The type of $\bar a$ over $K$ is non-orthogonal to the generic type of a Manin kernel of a simple abelian variety defined over $K^{alg}$ which does not descend to the field of constants $\CC_\UU$. 
		\item  The type of $\bar a$ over $K$ is non-orthogonal to the type of the constants over $K$.
	\end{itemize}
\end{fait}

In order to see that the three aforementioned cases are mutually exclusive, it suffices to show that the three kinds of types listed in the above fact are pairwise orthogonal, since non-orthogonality among minimal stationary types is a transitive relation \parencite[Chapter 1, Lemma 4.4.2]{Pillay_GeoStab}. By Corollary~\ref{C:non_orth_trivial}, we need only show that the generic type of a Manin kernel is  orthogonal to the constants. If these two types were non-orthogonal,  using similar arguments as the ones given by \textcite[Proof of Theorem 1.1]{Pillay_ML}, then  the (simple) abelian variety would descend to the constants, contradicting the definition of a Manin kernel. 

Since the field of constants $\CC_\UU$ is algebraically closed, we deduce immediately the following result from Fact~\ref{F:dichotomy}, which will play a fundamental role in the proof of \textcite{FJM22}.  
\begin{coro}\label{C:descent}
Given  a subfield $K$ of the field of constants $\CC_\UU$ of $\UU$,  if the stationary type of some tuple $\bar a$ over $K$ is minimal, then either the type of $\bar a$ over $K$ is trivial or it is non-orthogonal to the type of the constants over $K$. \qed
\end{coro}

Together with Lemma~\ref{L:trivial_1based} and Example~\ref{E:Manin}, Fact~\ref{F:dichotomy} yields Proposition~\ref{P:dichotomy} below. Whilst the proof may seem rather technical and obscure at first sight, we have decided to include it, for we believe that it highlights the power of the \emph{independence (or non-forking) calculus}, which is not difficult to master (The author of these notes is particularly fond of this sort of manipulations). 

\begin{prop}\label{P:dichotomy}
Consider a differential field~$K$ and a tuple~$\bar a$ such that the type of~$\bar a$ over~$K$ is not algebraic, but stationary and of finite rank. Either the type of~$\bar a$ over~$K$  is non-orthogonal to the type of the constants over~$K$ or  there exists a differential extension $L=\str{K}{\bar c}$ of~$K$, where $\bar c$~is a finite tuple, such that the type of~$\bar a$ over~$K$ is non-orthogonal to a $1$-based minimal stationary type based over~$L^{alg}$. 
\end{prop}
Recall (Definition~\ref{D:fterank}) that a tuple $\bar a$ has \emph{finite rank} over $K$ if every coordinate of $\bar a$ is differentially algebraic over $K$.   
\begin{proof}
We first start with the following  observation, which will be used twice in the proof.
	\begin{claimstar}
		Assume that there are finite tuples $\bar d_1$ and $\bar d_2$ as well as a differential field extension $M_1\subset M_2$ in $\UU$  satisfying all of the following conditions:
		\begin{enumerate}[(a)]
			\item We have that $\U(\bar d_1/M_1)=m\ge 1$ and  $\U(\bar d_1/\str{M_1}{\bar d_2})=m-1$;
			\item The tuple $\bar d_2$ is algebraic over $\str{M_2}{\bar d_1}$;
			\item The type of the tuple $\bar d_2$ over $M_2$ is minimal;
			\item We have the independence $\str{M_1}{\bar d_1, \bar d_2}\ind_{\str{M_1}{\bar d_2}} \str{M_2}{\bar d_2}$. 
		\end{enumerate}
		In this case, we have that the tuple $\bar d_1$ is independent from $M_2$ over $M_1$, or equivalently by Remark~\ref{R:fterank}.(b), the U-rank $\U(\bar d_1/M_1)=\U(\bar d_1/M_2)$. 
	\end{claimstar}
	\begin{claimstarproof}
		The proof is just an immediate application of Lascar's
                inequalities~\ref{R:fterank}.(b):  \begin{multline*}m=\U(\bar
                  d_1/M_1)\ge \U(\bar d_1/M_2)\stackrel{(b) \ \& \ \ref{R:fterank}.(b)}{=}\U(\bar d_1/\str{M_2}{\bar d_2})+ \U(\bar d_2/M_2)\stackrel{(d)}{=} \\ = \U(\bar d_1/\str{M_1}{\bar d_2}) + \U(\bar d_2/M_2) \stackrel{(c)}{=} \U(\bar d_1/\str{M_1}{\bar d_2}) + 1 \stackrel{(a)}{=} m-1+1 =m. \end{multline*}	\vskip-2mm\end{claimstarproof}
	
	In order to prove the statement of the proposition,  set $n=\U(\bar a/K)\ge 1$ (as the type of $\bar a$ over $K$ is not algebraic). If $n=1$, the type of $\bar a$ over $K$ is minimal, so we deduce the result immediately from Fact~\ref{F:dichotomy}, using Lemma~\ref{L:trivial_1based} and Example~\ref{E:Manin}, with $\bar c$ the empty tuple. Assume therefore that $n\ge 2$, and find by Remark~\ref{R:fterank} some differential field extension $K_1=K_1^{alg}$ of $K$ with $\U(\bar a/K_1)=n-1$. The canonical base $\cb(\bar a/K_1)$ is generated by a finite tuple $\bar b$ by Proposition~\ref{P:CB}.(d). Moreover, the  tuple $\bar b$  has finite rank over $K$, by Proposition~\ref{P:CB}.(e) and Remark~\ref{R:fterank}.(a). By construction, we have that \[ \str{K}{\bar a, \bar b} \ind_{\str{K}{\bar b}} K_1,\] so $\U(\bar a/\str{K}{\bar b})=n-1$. Now, the tuple $\bar b$ cannot be algebraic over $K$, for $\str{K}{\bar a}$ and $K_1$ are dependent over $K$, since $\U(\bar a/K_1)=n-1<\U(\bar a/K)$. 
	
	Choose some differential field extension $M$ of $K$ with \[ \U(\bar b/M)=1. \tag*{($\clubsuit$)}\] By {\bf Extension} (Remark~\ref{R:indep_prop2}.(b)), there exists some $\bar a_1\equiv_{\str{K}{\bar b}} \bar a$ with \[ \str{K}{\bar a_1, \bar b}\ind_{\str{K}{\bar b}} \str{M}{\bar b}, \ \text{ so }\U(\bar a_1/\str{M}{\bar b})=\U(\bar a_1/\str{K}{\bar b})=\U(\bar a/\str{K}{\bar b})=n-1. \]

Moreover,  \[ \cb(\bar a_1/\str{M}{\bar b}^{alg})\stackrel{\ref{P:CB}.(c)}{=} \cb(\bar a_1/\str{K}{\bar b}^{alg})=\cb(\bar a/\str{K}{\bar b}^{alg})=\cb(\bar a/K_1^{alg})=\str{\Q}{\bar b} \tag*{($\star$)}.\] Notice that $\bar b$ is algebraic over $\str{M}{\bar a_1}$. Indeed, since $\U(\bar b/M)=1$, the tuple $\bar b$ does not lie in $M^{alg}$, so  ($\star$)  and Proposition~\ref{P:CB}.(b) together with {\bf Symmetry} yield  \[ \str{M}{\bar b}\nind_M\str{M}{\bar a_1}, \ \text{ and thus } \bar b \text{ lies in }  \str{M}{\bar a_1}^{alg}.\] 

 By the Claim above, the tuple $\bar a_1$ is independent from $M$ over $K$. Choose some finite tuple $\bar c$ generating $\cb(\bar b/M^{alg})$ by Proposition~\ref{P:CB}.(d) and replace $M$ (or rather $M^{alg}$) with $L=\str{K}{\bar c}^{alg}$ everywhere (by {\bf Monotonicity}). In particular, the independence \[ \str{K}{\bar a_1}\ind_K L\] yields that the type of $\bar a_1$ over $L$ is again stationary, by Remark~\ref{R:stat}.(b). 

The type of $\bar b$ over $L=L^{alg}$ is stationary and minimal by ($\clubsuit$), so it is non-orthogonal to the type of some $\bar d$ over $L$, with $\tp(\bar d/L)$ a minimal type in one of the three classes listed in Fact~\ref{F:dichotomy}. By Remark~\ref{R:non_orth_minimal}, we find some differential field extension $L_1$ of $L$ and realizations $\bar b'$ and $\bar d'$, both independent from $L_1$ over $L$, which are interalgebraic over $L_1$. Since $\bar b'$ has the same type as $\bar b$ over $L$, there exists an automorphism $\sigma$ of $\UU$ fixing $L$ pointwise with $\sigma(\bar b)=\bar b'$. Set $\bar a'=\sigma(\bar a_1)$, so \[ \str{K}{\bar a', \bar b'}\ind_{\str{K}{\bar b'}} \str{L}{\bar b'},\] by {\bf Invariance}.  Possibly after composing with an automorphism fixing  $\str{L}{\bar b'}$ pointwise, we may assume by {\bf Extension} (Remark~\ref{R:indep_prop2}.(b)) that $\bar a'$ is independent from $\str{L_1}{\bar b'}$ over $\str{L}{\bar b'}$. Now 
\[ \cb(\bar a'/\str{L_1}{\bar d'}^{alg})=\cb(\bar a'/\str{L_1}{\bar b'}^{alg})\stackrel{\ref{P:CB}.(c)}{=} \cb(\bar a'/\str{L}{\bar b'})=\cb(\bar a'/\str{K}{\bar b'})=\str{\Q}{\bar b'} \tag{$\blacklozenge$},\] 
so \[ \U(\bar a'/\str{L_1}{\bar d'}^{alg})= \U(\bar a'/\str{L_1}{\bar b'}^{alg})\stackrel{(\blacklozenge) \ \& \ \ref{P:CB}.(c)}{=}\U(\bar a'/\str{K}{\bar b'})=\U(\bar a/\str{K}{\bar b})=n-1.\] 

Since $\U(\bar b'/L_1)=\U(\bar b'/L)=\U(\bar b/L)=1$,  we deduce that $\bar b'$ (and thus $\bar d'$) is algebraic over $\str{L_1}{\bar a'}$, using  the dependence \[ \str{L_1}{\bar b'}\nind_{L_1} \str{L_1}{\bar a'}.\] Indeed, assume for a contradiction the independence \[ \str{L_1}{\bar b'}\ind_{L_1} \str{L_1}{\bar a'}.\] The canonical base $\cb(\bar a'/\str{L_1}{\bar b'})\stackrel{(\blacklozenge)}{=}\str{\Q}{\bar b'}$ is then algebraic 
over $L_1$ by Proposition~\ref{P:CB}.(b), contradicting that $\U(\bar b'/L_1)=1$.

By the Claim again, we deduce that $\bar a'$~is independent from~$L_1$ over~$L$, and hence over~$K$ by {\bf Transitivity}. We conclude therefore that the type of~$\bar a$ over~$K$ is non-orthogonal to the minimal type of~$\bar d$ over $L^{alg}$, witnessed by the realizations~$\bar a'$ and~$\bar d'$, as desired. 
\end{proof}

Proposition~\ref{P:dichotomy} shows that every type of finite rank is non-orthogonal to some minimal type as in Fact~\ref{F:dichotomy}, possibly after adding additional parameters. The next 
fundamental notion in geometric model theory, called \emph{internality}, will be shown in Remark~\ref{R:internal} to occur   in the presence of non-orthogonality. Internality will play a major role in Section~\ref{S:binding} and also in the work of \textcite{FJM22}. 

\begin{defi}\label{D:internal}
	Given two differential fields~$K$ and~$L$, a \emph{$K$-conjugate} of a tuple~$\bar b$ over~$L$ is a tuple~$\bar b'$ given by an automorphism~$\sigma$ of $\Aut_{\delta}(\UU/K)$ with 
 	$\sigma(\str{L}{\bar b})=\str{L'}{\bar b'}$. We say that the  conjugate~$\bar b'$ of~$\bar b$ over~$L$ is \emph{based} over~$L'$. 
	
A stationary type of a tuple $\bar a$ over $K$ is \emph{internal} to the \emph{family of $K$-conjugates} of some tuple $\bar b$ over $L$ if there exists some  differential field extension $M$ of $K$  such that for every realization $\bar a'$ of the type of $\bar a$ over $K$, there are 
$K$-conjugates $\bar b_1,\ldots, \bar b_n$ of  $\bar b$ over $L$, each based over a subfield of $M$, such that $\bar a'$ belongs to the differential field $\str{M}{\bar b_1,\ldots, \bar b_n}$. 
\end{defi}

In the above definition of internality, we may always assume that the differential field $L$ is of the form $\str{K}{\bar c}$ for some finite tuple $\bar c$: Indeed, the canonical base $\cb(\bar b/\str{Q}{L\cdot K}^{alg})=\str{\Q}{\bar c}$ is finitely generated by Proposition~\ref{P:CB}.(d). Thus, the tuple~$\bar b$ is independent from $\str{\Q}{L\cdot K}^{alg}$ over $\str{K}{\bar c}$. Moreover,  the  $K$-conjugates of~$\bar b$ over~$L$ are in particular $K$-conjugates of~$\bar b$ over~$\str{K}{\bar c}$. 

 If the tuple $\bar b$ is based over a subfield of $K$, then 
so are its conjugates. However, we may still need additional parameters (coming from $M$) in order to witness internality. In particular, we can study the notion of internality if the tuple $\bar b$ above belongs to the field of constants $\CC_\UU$. In this case, every conjugate of a tuple in $\CC_\UU$ is again in $\CC_\UU$, so we say that the stationary type of $\bar a$ over $K$ is \emph{internal to the constants} if there exists some differential field extension $M$ of $K$ such that every realization $\bar a'$ of the type of $\bar a$ over~$K$ belongs to the differential subfield $M\cdot \CC_\UU$ of $\UU$. 

A typical example of types which are internal to the  constants are those given by tuples $\bar a$ in a given $K$-definable finite-dimensional $\CC_\UU$-vector space $V$. Indeed, consider~$M$ the differential field extension of~$K$ obtained after adding a basis of~$V$ to~$K$. Now, every realization $\bar a'$ of the type of $\bar a$ over $K$  is a $\CC_\UU$-linear combination of the basis, and thus belongs to  the differential subfield $M\cdot \CC_\UU$. Internality to the constants will be explored in further detail in Section~\ref{S:binding}.

\begin{rema}\label{R:internal}
\begin{enumerate}[(a)]

\item  Internality requires parameters $M$ which work for every realization of the stationary type of $\bar a$ over $K$. However,  it is equivalent to a  \emph{local} condition, which allows the  parameters to vary as we vary the realization. That is, with the notation of Defintion~\ref{D:internal}, the following are equivalent \parencite[Chapter 7,	\S 4, Lemma 4.2]{Pillay_GeoStab}: 
\begin{enumerate}[(1)]
\item The stationary type of $\bar a$ over $K$ is internal to  the family of $K$-conjugates of $\bar b$ over $L$.
\item There exists some differential field extension $M$ of $K$ and $K$-conjugates $\bar b_1,\ldots, \bar b_n$ of $\bar b$ over $L$, each based over a subfield of $M$, such that \[ \str{K}{\bar a}\ind_K M \ \text{ and }\ \bar a \in \str{M}{\bar b_1,\ldots, \bar b_n}.\]
\end{enumerate}
One direction is easy to prove:  If the type of $\bar a$ over $K$ is internal to the family of $K$-conjugates of $\bar b$ over $L$,  witnessed by the differential field extension $M$ of $K$, choose now by {\bf Extension} some $\bar a'\equiv_K \bar a$ independent from $M$ over $K$. 
 By assumption, there are $K$-conjugates $\bar b_1, \ldots, \ldots, \bar b_n$ of $\bar b$ over $L$, each based over a subfield of $M$, such that the realization $\bar a'$ lies in $\str{M}{\bar b_1,\ldots, \bar b_n}$. Now, if the  automorphism $\sigma$ of $\Aut_\delta(\UU/K)$ maps $\bar a'$ to $\bar a$, then we conclude that   \[ \str{K}{\bar a}\ind_K \sigma(M) \ \text{ and }\ \bar a \in \str{\sigma(M)}{\sigma(\bar b_1),\ldots, \sigma(\bar b_n)}, \] so $\sigma(M)$ is the desired field extension, since the $K$-conjugate $\sigma(\bar b_i)$ of $\bar b$ is based over a subfield of $\sigma(M)$. 
 
 We will prove the other direction under the additional assumption that the stationary type of $\bar a$ over $K$ has finite rank $\U(\bar a/K)\le \ell$. Without loss of generality (see the discussion right after Definition~\ref{D:internal}), we may assume that $L=\str{K}{\bar c}$ for some finite tuple $\bar c$.  
 
 By assumption,  there exists a differential field extension $M$ of $K$ and suitable $K$-conjugates $\bar b_1,\ldots, \bar b_n$ of the type of $\bar b$ over $L$ such that \[ \str{K}{\bar a}\ind_K M \ \text{ and }\ \bar a \in \str{M}{\bar b_1,\ldots, \bar b_n}.\] Choose a tuple $\bar m$ of $M$ containing the corresponding $K$-conjugates of the tuple $\bar c$ over $K$, so we may assume that all $\bar b_i$'s are based over $\str{K}{\bar m}$. Possibly after enlarging the tuple $\bar m$, we may also assume that the finite tuple $\bar a$ lies in $\str{K}{\bar m, \bar b_1,\ldots, \bar b_m}$. In particular, we have that \[ \str{K}{\bar a}\ind_K \str{K}{\bar m} \ \text{ and }\ \bar a \in \str{K}{\bar m,\bar b_1,\ldots, \bar b_n}. \tag{$\lozenge$}\]
  
 Consider now a Morley sequence $(\bar m_i)_{i\le \ell+1}$ of the stationary type of $\bar m$ over $K^{alg}$. Setting $M_1=\str{K}{\bar m_1,\ldots, \bar m_{\ell+1}}$, we need only show that every realization $\bar a'$ of the the type of $\bar a$ over $K$ belongs to the differential field generated over $M_1$ by $K$-conjugates of the type of $\bar b$ over $L$, where each $K$-conjugate is based over a differential subfield of $M_1$.  Given such a  realization $\bar a'$, it suffices to show  by {\bf Stationarity} and $(\lozenge)$ that there is some $1\le i\le \ell+1$ with $\bar a'$ independent from $\str{K}{\bar m_i}$ over $K$. 
 
 Assume otherwise, so \[  \str{K}{\bar a', (\bar m_j)_{j<i}}\nind_{\str{K}{(\bar m_j)_{j<i}}} \str{K}{(\bar m_j)_{j\le i}} \ \text{ for all } i\le \ell+1,\]
by   {\bf Symmetry and Transitivity}, for the $\bar m_i$'s are $K$-independent. Now, the strict inequalities 
 \begin{multline*} \mbox{}\hskip5mm \U(\bar a'/\str{K}{(\bar m_j)_{j\le \ell+1}})< \U(\bar a'/\str{K}{(\bar m_j)_{j<\ell}}) < \ldots   \\ < \U(\bar a'/\str{K}{\bar m_1}) < \U(\bar a'/K)=\U(\bar a/K)\le \ell.\end{multline*}
yield the desired contradiction.  
 
\item If the stationary type of~$\bar a$ over~$K$ is internal to the family of $K$-conjugates of~$\bar b$ over~$L$, then it admits a \emph{fundamental system of solutions}, that is, there exists realizations $\bar a_1,\ldots, \bar a_m$ of the type of~$\bar a$ over~$K$ such that for every realization~$\bar a'$ of the type of~$\bar a$ over~$K$, we have that~$\bar a'$ lies in $\str{K}{\bar a_1,\ldots, \bar a_m, \bar b_1,\ldots, \bar b_n}$, where the $\bar b_i$'s are $K$-conjugates of~$\bar b$ over~$L$. Moreover, we may find such a  fundamental system of solutions consisting of a Morley sequence of the type of~$\bar a$ over~$K$. 

In particular, the type of a fundamental system of solutions $(\bar a_1,\ldots, \bar a_m)$ over~$K$ is itself \emph{fundamental}: Every realization $(\bar a'_1,\ldots, \bar a'_m)$ of the type of $(\bar a_1,\ldots, \bar a_m)$ over~$K$ belongs to the field generated over~$M$ by $(\bar a_1,\ldots, \bar a_m)$ and $K$-conjugates of~$\bar b$ over~$L$.

The reader may notice  that the differential field extension $M= \str{K}{\bar a_1,\ldots, \bar a_m}$ of $K$ is not exactly as in Definition~\ref{D:internal}, for we do no longer impose that the $K$-conjugates $\bar b_i$'s are based over a subfield of $\str{K}{\bar a_1,\ldots, \bar a_m}$.

In order to show the existence of a fundamental system of solutions, assume that there exists some differential field extension $M=M^{alg}$ of $K$ with $\str{K}{\bar a}\ind_K M$ and a tuple $B=(\bar b_1, \ldots, \bar b_n)$ as in part (a) of this remark. Set $k=\cb(\bar a, B/M)$ the canonical base of $(\bar a, B)$ over $M$.  By Remark~\ref{R:CB_dcl}, the field $k$ a differential subfield of $\str{\Q}{\bar a_1, B_1,\ldots, \bar a_m, B_m}$ for some Morley sequence $(\bar a_i, B_i)$ of the type of $(\bar a, B)$ over $M$.  By {\bf Extension}, we can find some realization $(\bar a_{m+1}, B_{m+1})$ of the type of $(\bar a, B)$ over~$M$ independent from the previous sequence over~$M$. If $K_1$ denotes the differential subfield $\str{K}{\bar a_1,\ldots, \bar a_m}$, the field $k$ is a differential subfield of $\str{K_1}{B_1,\ldots, B_m}$. We will first show that $\bar a_{m+1}$ belongs to \[ \str{K_1}{B_1,\ldots, B_{m+1}}=\str{K}{\bar a_1,\ldots, \bar a_m, B_1,\ldots, B_{m+1}}.\] By construction, the fields $ \str{M}{\bar a_{m+1}, B_{m+1}}$ and  $\str{M}{(\bar a_i, B_i)_{i\le m})}$ are independent over $M$, and thus linearly disjoint over $M$, by Remark~\ref{R:indep_char}.(a). Proposition~\ref{P:CB}.(a) yields that the fields $\str{k}{\bar a_{m+1}, B_{m+1}}$ and $M$ are linearly disjoint over $k$, so we deduce that $\str{k}{\bar a_{m+1}, B_{m+1}}$ and $\str{M}{(\bar a_i, B_i)_{i\le m})}$ are linearly disjoint over $k$, by transitivity of linear disjointness. In particular, the fields $\str{K_1}{\bar a_{m+1}, B_1,\ldots, B_{m+1}}$ and $\str{M}{(\bar a_i, B_i)_{i\le m}, B_{m+1})}$ are linearly disjoint over $\str{K_1}{B_1,\ldots, B_{m+1}}\supset k$. Since $(\bar a, B)\equiv_M \bar a_{m+1}, B_{m+1}$, Corollary~\ref{C:univ} yields that $\bar a_{m+1}$ belongs to $\str{M}{B_{m+1}}$. The linear disjointness  \[  \str{K_1}{\bar a_{m+1}, B_1,\ldots, B_{m+1}}\ind^{ld}_{\str{K_1}{B_1,\ldots, B_{m+1}}} \str{M}{\bar a_1,\ldots, \bar a_m, B_1,\ldots, B_{m+1}} \] allows us to conclude that $\bar a_{m+1}$ belongs to \[ \str{K_1}{B_1,\ldots, B_{m+1}}=\str{K}{\bar a_1,\ldots, \bar a_m, B_1,\ldots, B_{m+1}},\] as desired. 

Note that the tuple $(\bar a_1,\ldots, \bar a_m)$ is a Morley sequence of the stationary type of $\bar a$ over $K$, by Remark~\ref{R:MS}.(b). 

Let us now show that every realization $\bar a'$ of the type of $\bar a$ over $K$ belongs to the field generated over $K_1$ by $K$-conjugates of $\bar b$ over $L$. By {\bf Extension}, find some realization $(\bar a_1',\ldots, \bar a_m')$ of the stationary type of $(\bar a_1,\ldots, \bar a_m)$ over $K$ with 
 \[  \str{K}{\bar a'_1,\ldots, \bar a'_m} \ind_K \str{K}{\bar a', \bar a_1,\ldots, \bar a_m}.\]
 By {\bf Stationarity}, for every $1\le i\le m$, the sequence $(\bar a_1,\ldots, \bar a_m, \bar a'_i)$ has the same type as $(\bar a_1,\ldots, \bar a_{m+1})$ over $K$. Thus, the tuple $\bar a'_i$ belongs to $\str{K}{\bar a_1, \ldots, \bar a_m, \bar d_i}$ for a suitable tuple $\bar d$ of $K$-conjugates of $B$ (and thus of $\bar b$) over $L$. Since   $(\bar a'_1,\ldots, \bar a'_m, \bar a')$  has the same type as $(\bar a_1,\ldots, \bar a_{m+1})$ over $K$ by {\bf Stationarity}, we conclude that  \[ \bar a' \in \str{K}{\bar a'_1, \ldots, \bar a'_m, \bar d'}\subset \str{K}{\bar a_1, \ldots, \bar a_m, \bar d_1,\ldots, \bar d_m, \bar d'},\] for suitable $K$-conjugates $\bar d_i$, as desired. 
\end{enumerate}
\end{rema}
Together with {\bf Algebraic closure}, the equivalence in Remark~\ref{R:internal}.(a) yields the following result:
\begin{coro}\label{C:internal_stat}
  The stationary type of $\bar a$ over $K$ is internal to  the family
 of $K$-conjugates of the type of $\bar b$ over $L$ if and only if the
 type of $\bar a$ over $K^{alg}$ is. \qed
\end{coro}
    
\textcite{Chatzidakis} showed in a previous version of her paper on the \emph{canonical base property} that $1$-basedness is preserved under internality (see below). This was later generalized by \textcite{Wagner} to \emph{higher analysis} beyond mere internality. 
\begin{fait}\label{F:internal_1based}
	Consider a tuple $\bar a$ whose type over $K$ is stationary and internal to the family of $K$-conjugates of some tuple $\bar b$ over $L$. If the type of $\bar b$ over $L$ is $1$-based, then  the type of $\bar a$ over $K$ is again $1$-based. 
	
	In particular, if $\bar a$ belongs to $\str{K}{\bar c_1,\ldots, \bar c_m}$ and the type of each $\bar c_i$ is $1$-based over $K$, then the type of $\bar a$ over $K$ is $1$-based. 
\end{fait}

We will finish this section establishing a connection between the notions of indifference and internality. 
\begin{prop}\label{P:nonorth_internal}
	Consider a differential field $K$ as well as a finite tuple $\bar a$ whose type over $K$ is stationary and not algebraic. 
	\begin{enumerate}[(a)]
\item Internality prevents indifference (or foreignness): If the  type of $\bar a$ over $K$ is internal to the family of $K$-conjugates of the type of $\bar b$ over the differential field $L$, then the type of $\bar a$ over $K$ cannot be  indifferent to the type of $\bar b$ over  $L$.
\item On the other hand, non-orthogonality  induces internality \parencite[Chapter~7, \S 4, Corollary 4.6]{Pillay_GeoStab}: If the  type of $\bar a$ over $K$ has finite rank, then there exists a finite tuple $\bar d$ in $\str{K}{\bar a}^{alg}$ whose type over $K^{alg}$ is non-algebraic (and stationary) such that the type $\bar d$ over $K^{alg}$ is internal to the family of $K$-conjugates of a tuple $\bar b$ over $\str{K}{\bar c}^{alg}$ for some finite tuple $\bar c$ in $\UU$, where the type of $\bar b$ over $\str{K}{\bar c}^{alg}$ is 
 one of the three classes of minimal types listed in Fact~\ref{F:dichotomy}. 
\item Furthermore, if the type of $\bar a$ over $K$ has finite rank and is orthogonal to the type of the constants, then we may find such a tuple $\bar d$ as in (b)  whose type over $K^{alg}$ is both minimal and $1$-based, as originally shown by \textcite[Theorem 2]{Hrushovski_locmod}. 
\end{enumerate}
\end{prop}
\begin{proof}
	For (a), assume that the type of~$\bar a$ over~$K$ is internal to the family of $K$-conjugates of~$\bar b$ over~$L$.  By Remark~\ref{R:internal}.(a), there are some differential field extension~$M$ of~$K$ and $K$-conjugates $\bar b_i$ of~$\bar b$ over~$L$, each based over a differential subfield of~$M$, such that \[  \str{K}{\bar a}\ind_K M \ \text{ and } \ \bar a\in \str{M}{\bar b_1,\ldots, \bar b_n}.\] Without loss of generality, we may assume that $M$ contains $L$, by an easy application of {\bf Extension} to the type of $\bar a$ over $M$.  In particular, {\bf Algebraic Closure} yields that $\bar a$ is not independent from $ \str{M}{\bar b_1,\ldots, \bar b_n}$ over $K$, since $\bar a$ does not belong to $K^{alg}$. Choose therefore $0\le i\le n-1$ maximal such that $\bar a$ is independent from  $M_1=\str{M}{\bar b_1,\ldots, \bar b_i}$ over $K$, so \[ \str{K}{\bar a}\ind_K M_1, \ \text{ yet } \str{M_1}{\bar a}\nind_{M_1} \str{M_1}{\bar b_{i+1}}\] by {\bf Monotonicity \& Transitivity}.  We deduce that the type of $\bar a$ over $K$ is not foreign to the type of $\bar b$ over $L$, as desired, witnessed by the realization $\bar b_{i+1}$ and the differential field extension $M_1$ of both $K$ and $L$.  
	
	For (b), notice that Proposition~\ref{P:dichotomy} yields that there are finite tuples $\bar b$ and $\bar c$ in $\UU$ such that the type of $\bar b$ over  the algebraic closure of $L=\str{K}{\bar c}$ is minimal and 
	the type of $\bar a$ over $K$ is non-orthogonal to the type of  $\bar b$ over $L^{alg}$. Using that non-orthogonality among minimal types is a transitive relation (see Fact~\ref{F:dichotomy} and the discussion thereafter), we may assume that the type of $\bar b$ over $L=\str{K}{\bar c}$ is one of the three minimal types listed in Fact~\ref{F:dichotomy}.
	
	By definition of non-orthogonality (Definition~\ref{D:orth}), there is 
 some differential field extension~$M$ of~$L$ (and thus of~$K$) as well as realizations~$\bar a'$ and~$\bar b'$, each independent from~$M$ over their corresponding base sets, such that $\bar a'$ and $\bar b'$ are not independent over~$M$. 
	
	Choose a finite tuple $\bar m$ in $M$  containing $\bar c$ as well as all the coefficients of the polynomial expressions needed to witness the dependence of $\bar a'$ and $\bar b'$ over $M$. Thus,   both $\bar a'$ and $\bar b'$ are each independent from $\str{K}{\bar m}$ over their base sets, yet \[  \str{K}{\bar m, \bar a'}\nind_{\str{K}{\bar m}} \str{K}{\bar m, \bar b'}.\] 	By Proposition~\ref{P:CB}.(d), there is some finite tuple $\bar d'$  which generates $\cb(\bar m, \bar b'/\str{K}{\bar a'}^{alg})$.  The tuple  $\bar d'$ belongs to 
	$\str{K}{\bar a'}^{alg}$, but cannot be contained in $K^{alg}$.  Otherwise, Proposition~\ref{P:CB}.(b) would yield that \[ \str{K}{\bar a'} \ind_K \str{K}{\bar m, \bar b'} \ \text{ and thus } \ \str{K}{\bar m, \bar a} \ind_{\str{K}{\bar m}} \str{K}{\bar m, \bar b'},\] contradicting our choice of $\bar m$. 
	
By Remark~\ref{R:CB_dcl}, there is some natural number $n$ such that \[ \cb(\bar m, \bar b'/\str{K}{\bar a'}^{alg})=\str{\Q}{\bar d'} \subset  \str{\Q}{\bar m_1,\bar b'_1,\ldots, \bar m_n,\bar b'_n} \subset \str{K}{\bar m_1,\bar b'_1,\ldots, \bar m_n,\bar b'_n},\] for some Morley sequence $(\bar m_1, \bar b'_1, \ldots, \bar m_n, \bar b'_n)$ of the type of the tuple $(\bar m, \bar b')$ over $\str{K}{\bar a'}$.  Notice that each $\bar b'_j$ is a $K$-conjugate of $\bar b$ over $L$. 

The tuples $\bar a'$ and $\bar m$ are independent over $K$, so it follows immediately from {\bf Invariance}, {\bf Stationarity}, {\bf Monotonicity \& Transitivity} that \[ \str{K}{\bar a'}\ind_K \str{K}{\bar m_1,\ldots, m_n}, \ \text{ and hence, }\ \str{K}{\bar d'}\ind_K \str{K}{\bar m_1,\ldots, m_n}, \] by {\bf Algebraic closure}.  Remark~\ref{R:internal}.(a) yields that the type of~$\bar d'$ over $K^{alg}$ is internal to the family of $K$-conjugates of~$\bar b$ over~$L$. Since $\bar a'$ and $\bar a$ have the same type over~$K$, there exists an automorphism~$\sigma$ mapping~$\bar a'$ to~$\bar a$. The tuple $\bar d=\sigma(\bar d')$ is contained in $\str{K}{\bar a}^{alg}$ and its type over $K^{alg}$ is also internal to the family of $K$-conjugates of~$\bar b$ over~$L$, as desired. 

For (c), assume that the type of $\bar a$ over $K$ is orthogonal to the constants. In the proof of part (b) above, it follows that the type of  $\bar b$ over $L^{alg}$ is either trivial or the generic type of a Manin kernel.  Both such types are $1$-based by Lemma~\ref{L:trivial_1based} and Example~\ref{E:Manin}. Fact~\ref{F:internal_1based}  implies that the stationary non-algebraic type of $\bar d$ over $K^{alg}$ is again $1$-based. 

We need only show that the type of $\bar d$ over $K$ is minimal. 
If $\U(\bar d/K)=1$, we are done. Otherwise, choose some differential field extension $K_1=K_1^{alg}$ of $K$ with $\U(\bar d/K_1)=\U(\bar d/K)-1$. By Proposition~\ref{P:CB}.(b) \& (d), there is a finite tuple $\bar f$ generating the canonical base $\cb(\bar d/K_1)$, so \[ \tag{$\lozenge$} \str{K}{\bar d, \bar f}\ind_{\str{K}{\bar f}} K_1.\] Since the type of $\bar d$ over $K$ is $1$-based, it follows that $\bar f$ belongs to $\str{K}{\bar d}^{alg}\subset \str{K}{\bar a'}^{alg}$. A straightforward application of Lascar inequalities (see Remark~\ref{R:fterank}.(b)) yields that 
\[ \U(\bar d/K)=\U(\bar d/\str{K}{\bar f}) + \U(\bar f/K) \stackrel{(\lozenge)}{=} \U(\bar d/K_1)+ \U(\bar f/K)= \U(\bar d/K)-1 +\U(\bar f/K).\] Hence, the type of $\bar f$ over $K$ is minimal. Fact~\ref{F:internal_1based} yields that the type of $\bar f$ over $K$ is $1$-based. Indeed,  the tuple $\bar f$ belongs to $\str{\Q}{\bar d_1,\ldots, \bar d_m}$ for some independent tuple of realizations $\bar d_i$ of the $1$-based type of $\bar d$ over $K_1^{alg}$, by Remark~\ref{R:CB_dcl}. Fact~\ref{F:internal_1based} yields that the type of $\bar f$ over $K$ is itself $1$-based, as desired.
\end{proof}

\section{Binding groups and Picard--Vessiot extensions}\label{S:binding}

 \textcite[\S 4]{Zilber_tt} and \textcite{Hr_unidim} noticed, beyond the particular  context of differential algebra, that internality produces the existence of a definable group of permutations, called the \emph{binding group}.  In the particular case of an internal type to the field of constants $\CC_\UU$ (within our universal differentially closed field $\UU$), the binding group can be identified with the \emph{elementary permutations} on the realizations of the internal type arising from differential automorphisms of $\UU$ fixing pointwise both the base set and the field $\CC_\UU$ of constants.  The connection between binding groups and differential Galois theory for differentially closed fields of characteristic $0$ was explored in more detail by \textcite{Poizat_Galois} and later on by \textcite{Pillay_Galois}. The algorithm provided by \textcite{Hrushovski_Gal} to effectively compute the binding group has been improved in the last decade by \textcite{Feng} and \textcite{Sun}.

Before we can introduce the binding group, we first need a couple of auxiliary results.  By Proposition~\ref{P:univ}, we work inside a universal $\aleph_1$-saturated differentially closed field $\UU$ with field of constants $\CC_\UU$.  All differential subfields of $\UU$ are countable, unless explicitly stated. 
\begin{rema}\label{R:wronsk}
	Given  a differential subfield $K$ of $\UU$, the elements $a_1,\ldots, a_n$ of $K$ are linearly independent over $\CC_K$ if and only if the \emph{Wronskian matrix} \[ \begin{pmatrix} a_1 & \ldots & a_n\\
		\delta(a_1) & \ldots &\delta(a_n)\\
		\vdots & \ddots & \vdots \\ 
		\delta^{n-1}(a_1) & \cdots & \delta^{n-1}(a_n)
	\end{pmatrix}\] is invertible \parencite[Lemma 4.1]{Marker}. As a consequence, the differential field $K$ is always linearly disjoint from $\CC_\UU$ over $\CC_K$, since the determinant of the above matrix does not depend on the ambient differential field containing the elements $a_i$. 
\end{rema}
Given a finite tuple  $\bar a$ and a differential subfield $K$ of $\UU$ such that the type of $\bar a$ over $K$ is stationary,  recall from the discussion after Definition~\ref{D:internal} that the type of $\bar a$ over $K$  is \emph{internal to the constants} (or \emph{$\CC_\UU$-internal}) if there exists some differential field extension $M$ of $K$ such that every realization of the type of $\bar a$ over $K$ belongs to the compositum field $M\cdot \CC_\UU$. The $\CC_\UU$-internal stationary type of $\bar a$ over $K$ is \emph{fundamental} if every realization $\bar a_1$ of the type of $\bar a$ over $K$ belongs to $\str{K}{\bar a}\cdot \CC_\UU$, or equivalently, if the tuple $\bar a$ is already a fundamental system of solutions, as in Remark~\ref{R:internal}.(b). 

\begin{rema}\label{R:binding_pn}
	Consider a finite tuple  $\bar a$ and a countable differential subfield $K$ of $\UU$ such that the type of $\bar a$ over $K$ is stationary and  internal to the constants. Given some $m\ge 1$ in $\N$ and a Morley sequence $(\bar a_1,\ldots, \bar a_m)$ of length $m$ of the type of $\bar a$ over $K$, the type of the tuple $(\bar a_1,\ldots, \bar a_m)$ over $K$ is internal to the constants, directly from the definitions. 
\end{rema}

More generally, we introduce the following notion:
\begin{defi}\label{D:set_internal}
	A Kolchin constructible subset~$X$ of~$\UU^n$ defined over a differential subfield~$K$ is \emph{internal to the constants} if there are finitely many (tuples of) rational functions $\bar R_1(\bar Y), \ldots, \bar R_k(\bar Y)$ with coefficients in some differential field extension~$M$ of~$K$ such that 
	\[ X\subset  \bigcup_{i\le k} \bar R_i(\CC_\UU),\] where we implicitly assume whenever we write $ \bar R_k(\bar c)$  that (every coordinate of) the rational function $ \bar R_k(\bar Y)$ is defined at the tuple $\bar c$. 
\end{defi}
An easy compactness (or rather an $\aleph_1$-saturation) argument yields the following result.
\begin{lemm}\label{L:internal_type_set_internal}
Given a finite tuple  $\bar a$ and a countable differential subfield $K$ of $\UU$ such that the type of $\bar a$ over $K$ is stationary,  we have that the  type of $\bar a$ over $K$ is internal to the constants if and only if $\bar a$ belongs to some Kolchin constructible subset $X$ of $\UU^{|\bar a|}$ defined over $K$ with $X$ internal to the constants. 
\end{lemm}
\begin{proof}
	One direction is immediate, so we need only show the existence of such a Kolchin constructible subset $X$ if the type of $\bar a$ over $K$ is internal to the constants. By assumption, there exists some countable differential field extension $M$ of $K$ such that every realization $\bar a_1$ of the type of $\bar a$ over $K$ belongs to the compositum field $M\cdot \CC_\UU$, so \[ \tag{$\star$} \bar a_1= \frac{\bar P_1(\bar m, \bar c)}{\bar P_2(\bar m, \bar c)}=\bar R(\bar m, \bar c)\] for some tuples $\bar m$ in $M$ and $\bar c$ in $\CC_\UU$ as well as tuples of polynomials $P_1$ and $P_2$ (all depending on $\bar a_1$) with integer coefficients  such that no coordinate of $\bar P_2(\bar m, \bar c)$ is $0$. In particular, for \[ \Sigma(\bar x)=\{ \varphi(\bar x, \bar b) \ | \  \UU\models \varphi(\bar a, \bar b), \varphi(\bar x, \bar y) \text{ diff. constr. formula }  \& \ \bar b \in K^{|\bar y|} \}\] the following collection of $M$-instances
	\[ \Sigma(\bar x) \cup \{\forall \bar y \left( \delta(\bar y)=0 \Rightarrow \bar x\ne \bar R(\bar m, \bar y) \right) \ | \ 
	\bar R \text{ rational over $\Z$}  \ \& \  \bar m \in M\}  \]
	cannot be finitely consistent (when we write such a rational function $\bar R$, we implicitly impose that the denominator does not vanish and rewrite the above expression in the language of differential rings).  Assume for a contradiction that it is finitely consistent.  By $\aleph_1$-saturation of $\UU$, there exists a realization $\bar a_1$ of the above collection of $M$-instances. Now, the tuple $\bar a_1$ has the same type as $\bar a$ over $K$, by Definition~\ref{D:types}, yet it cannot be written as in ($\star$), which gives the desired contradiction.

	Therefore, there exist finitely many $K$-instances $\varphi_i(\bar x, \bar b_i)$, each realized by $\bar a$, and finitely many rational functions $\bar R_1(\bar m_1, \bar Y), \ldots, R_k(\bar m_k, \bar Y)$  such that every realization of the Kolchin constructible set $X=\bigcap_{i=1}^m\varphi_i(\UU, \bar b_i)$ belongs to  $\bigcup_{i\le k} R_i(\bar m_k, \CC_\UU)$. Hence, the subset  $X$ of $\UU^{|\bar a|}$ defined over $K$ contains $\bar a$ and is  internal to the constants, as desired. 
\end{proof}

\begin{fait}\parencite[Theorem 4.4.5]{Buechler}\label{F:binding}
	Suppose that the Kolchin constructible set $X$ defined over the countable differential subfield $K$ is internal to the constants. There are
	\begin{itemize}
		\item a definable group $G$, that is, a differential algebraic group as in Remark~\ref{R:gp_nottrivial}.(a), defined over $K$; as well as 
		\item a definable group action $G\times X\to X$, as in Remark~\ref{R:gp_nottrivial}.(b), also defined over $K$;
	\end{itemize}
	such that the group $G$ and the action are isomorphic to the action of the group consisting of restrictions of automorphisms of $\Aut_\delta(\UU/K\cdot \CC_{\UU})$ (so each such automorphism fixes the compositum field $K\cdot \CC_\UU$) on the set $X$. In particular, the action is faithful. We will denote the group $G$ as $\Aut_\delta(X/K\cdot \CC_\UU)$. 
\end{fait}
Notice that the group $\Aut_\delta(\UU/K\cdot \CC_{\UU})$ acts on the set of realizations of the type of $\bar a$ over $K$, whenever this type is internal by Lemma~\ref{L:internal_type_set_internal}. However, the latter set need not be Kolchin constructible if the type is not isolated. In order to render the presentation of the binding group more accessible, we will only introduce it for isolated types (although the existence of the binding group of an internal type can be easily shown for arbitrary $\omega$-stable theories). In hindsight, Proposition~\ref{P:weak_orth_transitive} will  motivate the following definition. 

\begin{defi}\label{D:weakly_orth}
	We say that the stationary type of some finite tuple $\bar a$ over a countable differential subfield $K$ is \emph{weakly orthogonal to the constants} if $\CC_{\str{K}{\bar a}}=\CC_K$, or equivalently by Remark~\ref{R:wronsk}, if $\str{K}{\bar a}$ and $\CC_\UU$ are linearly disjoint over  $\CC_{K}$. 
\end{defi}
Note that the above definition does not depend on the representative of the type of $\bar a$ over $K$, by Corollary~\ref{C:univ}, since the set $\CC_\UU$ is invariant under all differentiaal automorphisms of $\UU$. 

\begin{rema}\label{R:weakly_orth}
		Consider a a countable differential subfield $K$  and a finite tuple $\bar a$ such that the type of  $\bar a$ over $K$ is stationary.
			\begin{enumerate}[(a)]
		 \item weak orthogonality to the constants is stable under base change to the algebraic closure: The type of $\bar a$ over $K$ is weakly orthogonal to the constants if and only if so is the type of $\bar a$ over  $K^{alg}$. One direction follows immediately from Remark~\ref{R:indep_char}.(a). For the converse, assume that the type of $\bar a$ over  $K^{alg}$ is weakly orthogonal to the constants. By Remark~\ref{R:wronsk} and transitivity of linear disjointness, the differential fields  $\str{K^{alg}}{\bar a}$ and $\CC_\UU\cdot K^{alg}$ are linearly disjoint over $K^{alg}$. Since the extension $K\subset \str{K}{\bar a}$ is regular, transitivity of linear disjointness yields also that $\str{K}{\bar a}$ and $\CC_\UU\cdot K^{alg}$, and thus $\CC_\UU\cdot K$, are  linearly disjoint over $K$. By Remark~\ref{R:wronsk}, the fields~$K$ and $\CC_\UU$ are linearly disjoint over $\CC_K$, so by transitivity of linear disjointness, we conclude that 
		 $\str{K}{\bar a}$ and $\CC_\UU$ are linearly disjoint over $\CC_K$, as desired. 
		
		\item The type of $\bar a$ over $K$  is weakly orthogonal to the constants if and only if $ \str{K}{\bar a}\ind_K K\cdot \CC_\UU$, that is, if $ \str{K}{\bar a}\ind_K K(\bar c)$ for every finite tuple $\bar c$ of $\CC_\UU$, since $\str{K}{\bar c}=K(\bar c)$. 	 Indeed, one direction follows easily from transitivity of linear disjointness and Remark~\ref{R:indep_char}.(a). Assume therefore that $\str{K}{\bar a}$ and $K\cdot \CC_\UU$ are independent over $K$. By part (a) above, we only need to show that the type of $\bar a$ over $K^{alg}$ is weakly orthogonal to the  constants. Remark~\ref{R:indep_char}.(a) and  {\bf Algebraic closure} yield that 
  the fields $\str{K}{\bar a}^{alg}$, and thus $\str{K^{alg}}{\bar a}$, and $K^{alg}\cdot \CC_\UU$ are linearly disjoint over $K^{alg}$, as desired.  
	
	\item  If the constant subfield $\CC_K$ of $K$ is algebraically closed, then the type of $\bar a$ over $K$ is weakly orthogonal to the constants if and only if $\str{K}{\bar a}^{alg}\cap \CC_\UU=\CC_K$. Indeed, one direction follows immediately from  Remark~\ref{R:indep_char}.(a). For the converse, assume that  $\CC_K=\str{K}{\bar a}^{alg}\cap \CC_\UU$.  Remark~\ref{R:wronsk} yields that the fields $\str{K}{\bar a}^{alg}$ and $\CC_\UU$ are linearly disjoint over $\CC_K$. In particular, the fields  $\str{K}{\bar a}$ and $\CC_\UU$ are linearly disjoint over $\CC_K$, as desired. 
 \end{enumerate}
\end{rema}

\begin{prop}\label{P:weak_orth_transitive}
	Consider a a countable differential subfield $K$  and a finite tuple $\bar a$ such that the type of  $\bar a$ over $K$ is stationary and  internal to the constants. The following conditions are equivalent: 
	\begin{enumerate}[(a)]
		\item  The type of $\bar a$ over $K$ is weakly orthogonal to the constants.
		\item For every Kolchin constructible set $X$ defined over $K$ containing $\bar a$ and internal to the constants, the action of the definable group $\Aut_\delta(X/K\cdot \CC_\UU)$ as in Fact~\ref{F:binding} on the set of realizations of the type of $\bar a$ over $K$ is transitive. 
		\item The type of $\bar a$ over $K$ is isolated (see Fact~\ref{F:diff_closure}) by a Kolchin constructible subset of $\UU^{|\bar a|}$ which is itself internal to the constants.  
	\end{enumerate}
\end{prop}
\begin{proof}
	For (a) $\Rightarrow$ (b), we need to show that  for every $\bar a'\equiv_K \bar a$ there exists an automorphism $\sigma$ of $\Aut_\delta(\UU/K\cdot \CC_\UU)$ mapping $\bar a$ to $\bar a'$ (despite the fact that $\CC_\UU$ is not countable!). Once such $\sigma$ has been constructed,  the restriction of $\sigma$ to the internal Kolchin constructible set $X$ defined over $K$ is an element of $\Aut(X/K\cdot \CC_\UU)$ mapping $\bar a$ to $\bar a'$, as desired. 
	
	By Corollary~\ref{C:univ}, there is a $K$-automorphism $\tau$  of $\UU$ mapping $\bar a$ to $\bar a'$. Since the type of $\bar a$ over $K$ is weakly orthogonal to the constants, the $K$-automorphism $\tau$ induces  a $K$-isomorphism $F_\tau\colon \str{K}{\bar a}\cdot \CC_\UU\to \str{K}{\bar a'}\cdot \CC_\UU$. As in the proof of \textcite[Chapter 10, Theorem 10.1.5]{TentZiegler}, the desired element $\sigma$ can be obtained from $F_\tau$ as a union of a chain of partial elementary maps (defined over countable subfields of $\UU$!) each of which fixes $K\cdot \CC_\UU$ pointwise and maps $\bar a$ to $\bar a'$.

	For (b) $\Rightarrow$ (c), Lemma~\ref{L:internal_type_set_internal} yields some Kolchin constructible set~$X$ defined over~$K$ containing~$\bar a$ with~$X$ internal to the constants.  Let $G=\Aut_\delta(X/K\cdot \CC_\UU)$ be the corresponding definable group with parameters in~$K$ as in Fact~\ref{F:binding}.  If $\star$~denotes the definable action of~$G$ on the set~$X$, then the Kolchin constructible set $G\star \bar a$  is invariant under $K$-automorphisms of~$\UU$.; Given an automorphism~$\sigma$ of~$\UU$ fixing~$K$ pointwise, it maps $\sigma(\bar a)$ to some tuple~$\bar a'$  which has the same type as~$\bar a$ over~$K$. In particular,  the tuple~$\bar a'$ lies in~$X$ by Definition~\ref{D:types}.  Transitivity of the action yields that  $\sigma(\bar a)=\bar a'=g\star \bar a$ for some~$g$ in~$G$, as desired.

	By Lemma~\ref{L:invariant_definition}, the Kolchin constructible set $X_1=G\star \bar a\subset X$  s definable by a $K$-instance $\varphi(\bar x, \bar e)$ of a Kolchin constructible formula.  It follows that the $K$-instance $\varphi(\bar x, \bar e)$  isolates the type of $\bar a$ over $K$, as desired.  Notice that $X_1$ is internal to the constants, since $X$ is. 
	
	For (c) $\Rightarrow$ (a),  we need only show that the type of $\bar a$ over ${K}^{alg}$ is weakly orthogonal to the constants, by Remark~\ref{R:weakly_orth}.(a). Let $\widehat{K}$ be the differential closure of $K^{alg}$, as in Fact~\ref{F:diff_closure}. By Corollary~\ref{C:diff_closure_noconstants}, the field of constants $\CC_{{K}^{alg}}$ is algebraically closed, so  $\CC_{\widehat{K}}=\CC_{{K}^{alg}}$. 
	
	By assumption, the type of $\bar a$ over $K$ is isolated by a $K$-instance $\varphi(\bar x, \bar b)$ for some tuple $\bar b$ in $K$. Now, this instance also isolates the type of $\bar a$ over $K^{alg}$: Indeed, if $\bar a_1$ belongs to $\varphi(\UU^{|\bar a|}, \bar b)$, then $\bar a_1$ and $\bar a$ have the same type over $K$. Stationarity yields that $\bar a_1\equiv_{K^{alg}} \bar a$, as desired. (For the readers who do not feel comfortable with this kind of argument, Quantifier Elimination (Theorem~\ref{T:DCF_QE}) allows to replace the algebraic parameters isolating the type of $\bar a$ over $K^{alg}$ with the coefficients of the minimal polynomial of a suitable primitive element). 
	
	Fact~\ref{F:diff_closure} yields that the type of $\bar a$ over $K^{alg}$ is realized in  $\widehat{K}$ by $\bar a'$.  Now,  \[ \str{K}{\bar a'}^{alg}\cap \CC_\UU\stackrel{\ref{R:wronsk}}{\subset} \CC_{\str{K}{\bar a'}^{alg}}\subset \CC_{\widehat{K}} =\CC_{{K}^{alg}}.\] Remark~\ref{R:weakly_orth}.(c) yields that  the type of $\bar a'$ (and thus of $\bar a$) over $K^{alg}$ is weakly orthogonal to the constants, as desired. 
\end{proof}
We can now introduce the notion of the binding group of an isolated internal stationary type. We would like to warn the reader that the notation $\Bind(\bar a/K)$ is not mainstream in model theory. 
\begin{defi}\label{D:binding}
Consider a  countable differential subfield $K$ and a tuple $\bar a$ whose type over $K$ is stationary,  internal to the  constants and weakly orthogonal to the constants.  By Proposition~\ref{P:weak_orth_transitive}, the set of realizations of 
type of $\bar a$ over $K$ 
is a Kolchin constructible set $X$ of $\UU^{|\bar a|}$ defined over $K$. 

The \emph{binding group} of the type of $\bar a$ over $K$, denoted by $\Bind(\bar a/K)$, is defined as the group $\Aut_\delta(X/K\cdot \CC_{\UU})$ of Fact~\ref{F:binding}. Notice that the group and the action on $X$, and thus on the set of realizations of the type of $\bar a$ over $K$, are both definable over $K$.  This action is faithful and transitive.  
\end{defi}
Transitivity of the action yields immediately the following result. 
\begin{coro}\label{C:nonalg_infinitebinding}
	With the assumptions of Definition~\ref{D:binding}, if the type of $\bar a$ over $K$ is non-algebraic, then the group $\Bind(\bar a/K)$ is infinite.\qed
\end{coro}

\begin{lemm}\label{L:fundsys_prime}
	Consider an algebraically closed countable differential subfield~$K$  and a tuple~$\bar a$ of~$\UU$ whose type over~$K$ is internal to the constants and weakly orthogonal to the constants. There exists a  fundamental system of solutions  $\bar d=(\bar a_1,\ldots, \bar a_m)$ as in Remark~\ref{R:internal}.(c) with each $\bar a_i\equiv_K\bar a$ such that the type of~$\bar d$ over~$K$ is again internal to the constants and weakly orthogonal to the constants. Moreover, both the group and the action of $\Bind(\bar d/K)$ on the set of realizations of the type of~$\bar b$ over~$K$ is isomorphic to the diagonal action of the group $\Bind(\bar a/K)$ on the set of realizations of the type of~$\bar d$ over~$K$. 
   	
	In particular, the type of $\bar d=(\bar a_1,\ldots, \bar a_m)$ over $K$ is \emph{fundamental} (see the paragraph before Definition~\ref{D:set_internal}): there are finitely many rational functions $\bar R_1 (\bar Y),\ldots, \bar R_\ell(\bar Y)$ with coefficients in $\str{K}{\bar d}$ such that every  realization $\bar d'=(\bar a'_1, \ldots, \bar a'_m)$ of the type of $\bar d$ over $K$ is of the form $\bar d'= \bar R_j( \bar c)$ for  some $1\le j\le \ell$ and some tuple $\bar c$ in $\CC_\UU$. 
\end{lemm}
Whilst we impose that the differential field $K$ is algebraically closed, this will only be used  in the proof in order to ensure that certain types (in particular the type of the tuple $\bar d$) over $K$ are stationary. 
\begin{proof}
	The type of $\bar a$ over $K$ is stationary, for $K$ is algebraically closed. Proposition~\ref{P:weak_orth_transitive} yields that this type is isolated by a $K$-instance $\varphi(\bar x, \bar b)$ for some $\bar b$ in $K$ of a Kolchin constructible formula.  By Remark~\ref{R:internal}.(c), there exists a fundamental system of solutions for the type of $\bar a$ over $K$ consisting of a Morley sequence $(\bar a_1,\ldots, \bar a_m)$ of the type of $\bar a$ over $K$. Hence, every realization  of the type of $\bar a$ over $K$ belongs to the differential field $\str{K}{\bar d_1,\ldots, \bar d_m}\cdot \CC_\UU$. 
	
	An $\aleph_1$-saturation argument as in the proof of Lemma~\ref{L:internal_type_set_internal} yields that there are finitely many rational functions $\bar H_1,\ldots, \bar H_t$ over $K$ with coefficients in $\mathbb Z$ witnessing that every realization  of the type of $\bar a$ over $K$ belongs to \[ \bigcup_{j\le t} \bar H_j(\bar a_1,\ldots, \bar a_m, \CC_\UU),\] so 
	\[ \UU\models \exists \bar x_1\ldots \bar x_m \Bigg( \bigwedge_{i=1}^m \varphi(\bar x_i, \bar b) \land \forall \bar x \Big(\varphi(\bar x, \bar b) \Rightarrow \exists \bar c \ \big( \delta(\bar c)=\bar 0\land  \bigvee_{j=1}^t \bar x=\bar H_j(\bar x_1\ldots \bar x_m, \bar c)\big) \Big)\Bigg).\] 
	 Quantifier Elimination~\ref{T:DCF_QE} yields that the above $K$-instance must hold in the differential closure $\widehat{K}$ of $K$. Hence, there exists a tuple  $\bar d$ in $\widehat{K}$ consisting of realizations of the type of $\bar a$ over $K$ such that \[ \widehat{K}\models \forall \bar x \Big(\varphi(\bar x, \bar b) \Rightarrow \exists \bar c \ \big( \delta(\bar c)=\bar 0\land  \bigvee_{j=1}^t \bar x=\bar H_j(\bar d, \bar c)\big) \Big).\] Quantifier-elimination (Theorem~\ref{T:DCF_QE}) yields that the same $\widehat{K}$-sentence must hold in $\UU$, so \[ \varphi(\UU^{|\bar x|}, \bar b)\subset \bigcup_{j\le t} H_j(\bar d, \CC_\UU).\]  It follows that the tuple $\bar d$ of $\widehat{K}$  is a fundamental system of solutions in $\UU$ of the type of $\bar a$ over $K$. Since $K$ is algebraicallly closed, the type of $\bar d$ over $K$ is stationary. It is clearly internal to the constants, since the type of $\bar a$ over $K$ is. The tuple $\bar d$ belongs to $\widehat{K}$, so its type over $K$ is isolated by Fact~\ref{F:diff_closure}. Hence, the type of $\bar d$ over $K$ is weakly orthogonal to the constants by Proposition~\ref{P:weak_orth_transitive}. 
	
	An element $\sigma$ of $\Bind(\bar a/K)$ is the restriction of a differential automorphism in $\Aut_\delta(\UU/K\cdot \CC_\UU)$ to the Kolchin constructible subset $X=\varphi(\UU^{|\bar x|}, \bar b)$. The group $\Aut_\delta(\UU/K\cdot \CC_\UU)$ acts diagonally in a natural way on the set of realizations of the type of $\bar d=(\bar a'_1,\ldots, \bar a'_m)$ over $K$, by construction. Moreover, if $\bar a'$ is some realization of the type of $\bar a$ over $K$, write $\bar a'=\bar H_j(\bar d, \bar c)$, so $\sigma(\bar a')=\bar H_j(\sigma(\bar d), \bar c)$.
	
We deduce that the group $\Bind(\bar d/K)$ and its action on the (set of realizations of the) type of $\bar d$ over $K$ is thus naturally isomorphic to the  diagonal action of $\Bind(\bar a/K)$ on the same set, as desired.  
	
	We will conclude this proof showing that the type of  $\bar d=(\bar a'_1,\ldots, \bar a'_m)$ over $K$ is fundamental. Since every realization $\bar d'$ of the type of $\bar d$ over $K$ consists of a tuple of realizations of the type of $\bar a$ over $K$, it is easy to produce the desired rational functions $R_1,\ldots, R_\ell$ with parameters in $\str{K}{\bar d}$ out of the rational functions $H_j$'s i such a way that $\bar d'=R_j(\bar c)$ for some $1\le j\le \ell$ and some tuple $\bar c$ in $\CC_\UU$, as desired. 
\end{proof}

Given a countable algebraically closed differential subfield $K$, or more generally, a field $K$ whose constant subfield $\CC_K$ is algebraically closed, Kolchin defines a differential field extension $L$ of $K$ to be \emph{strongly normal} if the following conditions hold:
\begin{itemize}
\item The field $L$ is finitely generated over $K$ (as a differential field), so $L=\str{K}{\bar a}$ for some tuple $\bar a$. 
\item The constant subfield $\CC_L=\CC_K$, so the type of $\bar a$ over $K$ is weakly orthogonal to the constants as in Definition~\ref{D:weakly_orth}. 
\item For every $\sigma$ in $\Aut_\delta(\UU/K)$, the differential field  $\sigma(L)$ is contained in the compositum $L\cdot \CC_\UU=\str{K}{\bar a}\cdot \CC_\UU$, so the type of $\bar a$ over $K$ is internal to the constants and fundamental. 
\end{itemize} \textcite[Chapter VI, Theorem 1]{Kolchin}  showed that the group $\Aut_\delta(L\cdot \CC_\UU/K\cdot \CC_\UU)$ is isomorphic to the group of $\CC_\UU$-rational points of an algebraic group defined over the constants. Moreover, since $\CC_K$ is algebraically closed, this group is connected \parencite[Chapter VI, Corollary 1]{Kolchin}. 

Summarizing, we obtain the following Galois correspondence, for which  \textcite[Theorem 2.3 \& \S 3.2]{OmarPillay} provided a purely model-theoretic account. 

\begin{fait}\label{F:Kolchin_Galois}
	Given an algebraically closed countable differential subfield $K$ of $\UU$, consider a finite tuple $\bar a$ such the type of $\bar a$ over $K$  is internal and weakly orthogonal to the constants with binding group $\Bind(\bar a/K)$. If $\bar d$ is  a fundamental system of solutions $\bar d$ for the type of $\bar a$ over $K$ as in Lemma~\ref{L:fundsys_prime}, then the following holds:
	
	\begin{enumerate}[(a)]
		\item The differential field extension $K\subset \str{K}{\bar d}$ is strongly normal in the sense of Kolchin. 
\item There exists a connected algebraic group $G$ defined over the constants and an isomorphism $\Phi\colon  \Bind(\bar a/K)\to G(\CC_\UU)$ whose graph is definable, that is, Kolchin constructible, over $\str{K}{\bar d}$  \parencite[Chapter VI, Theorem 1]{Kolchin}. In particular, given an element $g$ of $G(\CC_\UU)$, we denote by $\sigma(g)=\Phi^{-1}(g)$ the corresponding automorphism  of $\Aut_\delta(\UU/K\cdot \CC_\UU)$ (or rather, the restriction of $\sigma(g)$ to the set of realizations of the fundamental type of $\bar d$ over $K$).
\item Given a differential subfield $K\subset K_1\subset \str{K}{\bar d}$, the group \[ \mathrm{Fix}(K_1)=\{g\in G(\CC_\UU) \ | \ \sigma(g)\restr{K_1}=\mathrm{Id}_{K_1}\}  \] is again the subgroup of $\CC_\UU$-rational points of an algebraic subgroup of $G$ defined over the constants. Moreover, its fixed subfield  equals $K_1$. Furthermore, every algebraic subgroup of $G(\CC_\UU)$ defined over the constants occurs in such a fashion \parencite[Chapter VI, Theorem 3]{Kolchin}. 
\item An element $x$ of $\str{K}{\bar d}$  belongs to $K$ if and only if $x$  is fixed by every $\sigma(g)$ with $g$ in $G(\CC_\UU)$. 
	\end{enumerate}
\end{fait}
Archetypal  \emph{strongly normal} extensions are given by Picard--Vessiot extensions, which are described below. 
\begin{exem}\label{E:PV}
Consider an algebraically closed countable differential subfield~$K$ of~$\UU$ (so $\CC_K$ is algebraically closed as well). The \emph{Picard--Vessiot} extension of $K$ associated to the \emph{holonomic} differential polynomial \[ P(X)=\delta^n(X)+b_{n-1}\delta^{n-1}(X)+\cdots+ b_1\delta(X)+b_0,\] with all $b_i$'s in~$K$ is a finitely generated differential field extension $L=\str{K}{a_1,\ldots, a_n}$ of~$K$ with $P(a_i)=0$ for $1\le i\le n$ such that:
\begin{itemize}
	\item The solutions $a_1,\ldots, a_n$ of $P(X)=0$ are linearly independent over $\CC_K$. 
	\item The field of constants  $\CC_L=\CC_\UU\cap L=\CC_K$ (so the type of $(a_1,\ldots, a_n)$ over $K$ is weakly orthogonal to the constants).
\end{itemize}
By Remark~\ref{R:wronsk}, every other solution $a'$ is a $\CC_\UU$-linear combination $ a'=\sum_{i=1}^n c_i a_i$. Thus, the stationary type of $(a_1,\ldots, a_n)$ over $K$ is fundamental and the tuple $(a_1,\ldots, a_n)$ is a fundamental system of solutions over $K$. We can write the above equation in a matrix form using the variables $X_{i+1}=\delta^i(X)$ for $0\le i\le n$, so \[ \delta(\bar X)= B\cdot \bar X=\begin{pmatrix}
	0 & 1 & 0 &\ldots & 0\\
	0 & 0 & 1 &\ldots & 0\\ 
					\vdots & \ddots &\ddots& \ddots & \vdots  \\  0& 0 & \ldots& 0& 1\\
												-b_0 & -b_1 & \ldots &\ldots& -b_{n-1} \end{pmatrix} \cdot \bar X.\]			
It is now easy to see that the action of an element $\sigma$ of $\Aut_\delta(\UU/K\cdot \CC_\UU)$ on the fundamental system translates into an invertible $n\times n$-matrix $M_\sigma$ in $\mathrm{GL}_n(\CC_\UU)$ with \[ \sigma \underbrace{\begin{pmatrix} \bar a_1 | \bar a_2 | \cdots |\bar a_n
\end{pmatrix}}_{A}= A\cdot M_\sigma.\] The binding group $\Bind(a_1,\ldots, a_n/K)$ of the type of $(\bar a_1,\ldots, \bar a_n)$ over $K$ is linear, as it is definably isomorphic to a subgroup of  $\mathrm{GL}_n(\CC_\UU)$. The image of $\Bind(a_1,\ldots, a_n/K)$, as a subgroup of $\mathrm{GL}_n(\CC_\UU)$, is the \emph{differential Galois group} of the Picard--Vessiot extension $K\subset \str{K}{\bar a_1,\ldots, \bar a_n}$. 
\end{exem}
Using the above notation, the coefficients of the matrix $\delta(A)\cdot A^{-1}$ are fixed under the action of the elements of $\Bind(a_1,\ldots, a_n/K)$. This way we  recover the above matrix $B$ defining the Picard--Vessiot extension. Mimicking this idea, we can derive the following fact from a straightforward application of the Galois correspondence. 
\begin{fait}\label{F:linear_PV}
If the binding group of a strongly normal extension $L$ of an algebraically closed differential field $K$ is linear, then $L$ is a Picard--Vessiot extension over $K$ \parencite[Chapter VI, pp. 410\&411]{Kolchin}. 
\end{fait}
We will conclude this section with a result, first stated by \textcite[Proposition 4.9]{JaouiJimenezPillay}, with a more detailed proof given by \textcite[Theorem 3.9]{FJM22},  on the algebraic structure of the binding group for differential algebraic  \emph{holonomic} equations defined over the constants. 
\begin{rema}\label{R:PV_commutative}
Given a Picard--Vessiot extension defined over the algebraically closed subfield $K$ of $\CC_\UU$, its binding group is always commutative. Indeed,,  it is easy to explicitly compute from the Jordan normal form of the defining matrix $B$ in Example~\ref{E:PV} a fundamental system of solutions in order to deduce that the binding group of each block is commutative with unipotent radical of dimension at most $1$.
\end{rema}


\section{The property D$_2$ and functional transcendence}\label{S:D2}

We now have all the ingredients in order to prove one of the main results of \textcite[Theorem 3.9]{FJM22}. As in the previous sections, we work inside a universal $\aleph_1$-saturated differentially closed field $\UU$, by Proposition~\ref{P:univ}. All tuples and fields are taken within $\UU$. Moreover, all differential subfields are countable, unless explicitly stated. 

In order to include the examples listed in the Introduction, we will generalize the notion of \emph{equations} beyond mere polynomial equations. 
\begin{Notation}
By a \emph{differential algebraic equation} $P(T)=0$ we mean a differential rational function $P(T)=\frac{Q(T)}{R(T)}$ for some differential polynomials $Q(T)$ and $R(T)$ with $R(T)$ not the constant $0$ polynomial and $\ord(R(T))<\ord(Q(T))$. The equation is (defined) \emph{over a differential subfield } $K$ if both $Q(T)$ and $R(T)$ have coefficients in $K$. It is \emph{irreducible over} $K$ if the numerator $Q(T)$ is (as a multivariate polynomial).  The \emph{order} of the differential algebraic equation  is $\ord(Q(T))$. A \emph{solution}  of the equation $P(T)=0$ is an element $a$ in $\UU$ with $Q(a)=0\ne R(a)$. 

In particular, every differential polynomial induces a differential algebraic equation, setting $R(T)=1$. 
\end{Notation}

\begin{defi}\label{D:Dk}
Let $P(T)=0$ be a differential algebraic equation of order $n$ over a countable differential subfield $K$ of $\UU$. Given a natural number $m\ge 1$, the equation has \emph{Property D$_m$} if, for every $m$ many pairwise distinct solutions $a_1,\ldots, a_m$ in $\UU\setminus K^{alg}$, the set \[ a_1,\delta(a_1), \ldots, \delta^{n-1}(a_1), \ldots, a_m,\delta(a_m), \ldots, \delta^{n-1}(a_m)\] is an algebraically independent family  over $K$. 
\end{defi}
 \textcite[Definition 3.1]{FJM22} use the terminology \emph{Property C$_m$} for differential algebraic equations (and use the  corresponding \emph{Property D$_m$} for types). We have decided to  slightly change the terminology to avoid possible confusions with the arithmetic property C$_m$ on the existence of non-trivial solutions of homogeneous forms in $n$ variables of degree $d$ with $d^m\ge n+1$.

If the order of the equation is $n\ge 1$, then  $\aleph_1$-saturation and the axioms of differentially closed fields~\ref{D:DCF} yield that there are (uncountably many) non-algebraic solutions to the equation.  Thus,  the above Property is not an empty condition if $n\ge 1$. In this case, Property D$_{m+1}$  implies Property D$_m$ for every $k\ge 1$. 
\begin{rema}\label{R:D2_generictype}
Consider a differential algebraic equation $P(T)=0$ of order $n\ge 1$ over some differential subfield $K$ with Property $D_1$. If the equation is irreducible, then any two solutions $a$ and $a'$ of the equation, none of which is algebraic over $K$, must have the same type over $K$. We will refer to the type of any non-algebraic solution over $K$ as  the \emph{generic type of the equation over} $K$ (honoring the longstanding tradition in model theory of  repeatedly using the same word for different notions!). 
\end{rema} 
\begin{proof}
Write $P(T)=\frac{Q(T)}{R(T)}$ with $Q(T)$ an irreducible differential polynomial of order $n$ and $\ord(R(T))<n$. Suppose $a$ and $a_1$ are two solutions of the above equation, none of them algebraic over $K$, so $Q(a)=0=Q(a_1)$. In order to show that $a$ and $a_1$ have the same type over $K$, as in Definition~\ref{D:types},  it suffices to show by Remark~\ref{R:iso_diff} that $Q(T)$ is the minimal poynomial of the differential vanishing ideal of $a$, resp.\ of $a_1$, over $K$, as in Remark~\ref{R:prime_minpol}.(a).  Now, Property D$_1$ yields that $a, \ldots, \delta^{n-1}(a)$ are algebraically independent elements over $K$, so $a$ does not satisfy any algebraic differential polynomial equation of order strictly less than $n$. Since $Q(T)$ is irreducible over $K$, we deduce from Remark~\ref{R:prime_minpol}.(b) that $Q(T)$ is indeed the minimal polynomial of $a$ over $K$, as desired. 
\end{proof}
We will now show that the generic type of an irreducible differential algebraic equation having property D$_2$ satisfies a stronger version of \emph{the non-existence of proper fibrations}, a notion first introduced by \textcite[Definition 2.1]{MoosaPillay} in their model-theoretic study of the algebraic reductions of  (generalised) hyperkähler manifolds.  

In order to render the presentation simpler, we will from now on assume that the base differential subfield $K$ is algebraically closed, so the generic type of the equation is in particular stationary. 
\begin{lemm}\label{L:D2_fibration}
Consider an irreducible differential algebraic equation of order $n\ge 1$ having property D$_2$ over an algebraically closed differential subfield $K$ of $\UU$. Any two distinct generic solutions $a\ne a_1$ over $K$ form a Morley sequence of length $2$ of the generic type.   

In particular, if some tuple $\bar b$ in $\UU$ is algebraic over the differential field $\str{K}{a}$ generated by a generic solution $a$ over $K$, either $\bar b$ is already contained in $K=K^{alg}$ or $a$ belongs to the differential subfield $\str{K}{\bar b}$. 
\end{lemm}
\begin{proof}
Consider two distinct generic solutions $a\ne a_1$ over $K$ of the equation. Remark~\ref{R:D2_generictype} yields that $a$ and $a_1$ have both the same type over $K$ and furthermore this type is stationary (since $K$ is algebraically closed). In order to show that the pair  $(a, a_1)$ is a Morley sequence of length $2$ of the generic type of the equation, we need to show that the differential fields $\str{K}{a}$ and $\str{K}{a_1}$ are independent over $K$. Now, Property D$_2$ implies that the fields $K(a,\ldots, \delta^{n-1}(a))$  and $K(a_1,\ldots, \delta^{n-1}(a_1))$ are algebraically independent over $K$. Each derivative $\delta^{n+k}(a)$ is algebraic over $K(a, \delta(a), \ldots, \delta^{n-1}(a))$ for all $k$ in $\N$. Thus, the differential fields $\str{K}{a}$ and $\str{K}{a_1}$ are independent over $K$, by {\bf Algebraic closure}, as desired. 

Assume now that the tuple $\bar b$ of $\UU$ is algebraic over $\str{K}{a}$ yet $a$ does not belong to the differential subfield $\str{K}{\bar b}$. Without loss of generality, we may assume that $\bar b$ is a singleton $b$.  Corollary~\ref{C:dcl} yields some $a_1\ne a$ with the same type over $\str{K}{b}$ as $a$. In particular, the element $b$ is also algebraic over $\str{K}{a_1}$ and $a_1$ is a generic solution over $K$.  We deduce from the independence \[ \str{K}{a}\ind_K \str{K}{a_1}\] together with {\bf Algebraic closure}  that $b$ must be algebraic over $K$, as desired. 
\end{proof}
The last assertion of the above Lemma implies that the generic type of an equation with Property D$_2$ is always weakly orthogonal to the constants. 
\begin{coro}\label{C:D2_weakorth}
Given an irreducible differential algebraic equation $P(T)=0$ of order $n\ge 1$ having  Property D$_2$ and defined over an algebraically closed differential subfield  $K$, we have that its generic type over $K$ is weakly orthogonal to the constants. 
\end{coro}
\begin{proof}
Fix a generic solution $a$ of the equation over $K$. Recall that  the type of $a$ over $K$ is stationary, as $K$ is algebraically closed. 

In order to show that the type of $a$ over $K$ is weakly orthogonal to the constants, assume for a contradiction that there exists a constant element $c$ in $\str{K}{a}\setminus \CC_K$. In particular, the element $c$ is transcendental over $K=K^{alg}$.  Lemma~\ref{L:D2_fibration} yields that 
$a$ belongs to $\str{K}{c}$, so there are differential rational functions \[ \Phi(T)=\frac{R_1(T)}{R_2(T)} \text{ and }  \Psi(T)=\frac{S_1(T)}{S_2(T)}\] with all $R_i$'s and $S_j$'s in $K\{T\}$ such that $R_2(c)S_2(a)\ne 0$ with $a=\Phi(c)$ and $c=\Psi(a)$. 

By Quantifier Elimination~\ref{T:DCF_QE}, the set \[ X=\{ d \in \CC_\UU \ | \  \ P(\Phi(d))=0 \  \& \ d=\Psi(\Phi(d)) \] is Kolchin constructible over $K$ and contains $c$. A differential polynomial evaluated on a constant element is just a classical polynomial, so the above set $X$ is a Zariski constructible subset of the algebraically closed field $\CC_\UU$ (a priori, the parameters necessary to define $X$ lie in $K$, yet Remark~\ref{R:wronsk} implies that $X$ can be defined using only parameters from $\CC_K$). Since the element $c$ of $X$ is transcendental over $K$, we deduce that $X$ must be cofinite. In particular, there exists some natural number $m \ne 0$ such that  $c+m$ belongs to $X$. The element $a_1=\Phi(c+m)$ is a solution of the differential algebraic equation $P(T)=0$. 

Now, if $a_1$ were equal to $a$, then $c=\Psi(a)=\Psi(a_1)=c+m$, contradicting that $m\ne 0$. Moreover, the solution $a_1$ does not belong to $K$, for otherwise so would  $c+m=\Psi(a_1)$  lie in $K$, contradicting that $c$ is transcendental over $K$. Hence,  Property D$_2$ implies that $\str{K}{a}$ and $\str{K}{a_1}$ are independent over $K$, which immediately implies that $c$ is algebraic over $K$ and yields the desired contradiction. 
\end{proof}
Another easy yet surprising consequence of Lemma~\ref{L:D2_fibration} is that non-orthogonality is always due to internality. 
\begin{coro}\label{C:D2_nonorth}
Consider an irreducible differential algebraic equation $P(T)=0$ of order $n\ge 1$ having Property D$_2$  defined over some algebraically closed differential subfield $K$.  As in Proposition~\ref{P:dichotomy}, suppose that there is some differential field extension $L=\str{K}{\bar c}$ of $K$ such that the generic type of the equation over $K$ is non-orthogonal to the minimal type of some tuple $\bar b$ over $L$. It follows that  the generic type of the equation is internal to the family of $K$-conjugates of $\bar b$ over $L$. 

In particular, if the generic type of the equation over $K$ is non-orthogonal to the type of the constants, then it is internal to the constants. 
\end{coro}
\begin{proof}
Assume that $a$ is a solution of the irreducible differential algebraic differential equation $P(T)=0$ with $a$ generic over $K=K^{alg}$ such that the (stationary) type of $a$ over $K$ is
non-orthogonal to the minimal type of some $\bar b$ over $L=\str{K}{\bar c}$ as in Proposition~\ref{P:dichotomy}.  Proposition~\ref{P:nonorth_internal}.(b) implies that there exists some tuple $\bar d$ 
algebraic over $\str{K}{a}$ yet not algebraic over $K$ such that the
(stationary) type of $\bar d$ over $K^{alg}=K$ is internal to the family of $K$-conjugates of $\bar b$ over $L$.  Lemma~\ref{L:D2_fibration} yields that $a$ belongs to $\str{K}{\bar d}$. It follows directly from the definition of internality (Definition~\ref{D:internal}) that  the type of $a$ over $K$ is also internal to the family of $K$-conjugates of $\bar b$ over $L$, as desired.  
\end{proof}
Before stating the next proposition, we will provide a succint account
of $2$-transitive group actions, which will be needed for the proof.
\begin{rema}\label{R:2trans}
Given a set $X$ with at least two elements, set \[ X^{(2)}=\{ (x_1,
  x_2) \in X^2 \ | \ x_1\ne x_2\}\ne \emptyset.\] 
 An action 
 \[ \begin{array}{ccl}
G\times X &\to& X \\[1mm]
(g, x)& \mapsto& g\star x
 \end{array} \]
 of a group $G$ on $X$ is \emph{$2$-transitive} if the diagonal action of $G$ on $X^{(2)}$ is transitive, that is, for all
 $(x_1, x_2)$ and $(y_1, y_2)$ in $X^{(2)}$ there exists some $g$ in $G$ with
 $g\star x_i=y_i$ for $1\le i\le 2$. If $G$ acts $2$-transitively on
 $X$, then the action of $G$ on $X$ is already transitive (since $X$
 contains at least two elements). In particular, any two stabilizers $\Stab_G(x)$ and $\Stab_G(x')$, with $x$ and $x'$ in $X$, must be conjugate to each other. Moreover, for every $x$ in $X$, the subgroup
 $\Stab_G(x)$ is maximal in $G$, that is, the only subgroup of $G$ which properly contains   $\Stab_G(x)$ is $G$ itself. Indeed,  given a subgroup $H$ of $G$ with
 $\Stab_G(x)\lneq H$, there is some $h$ in $H$ with $h\star x\ne
 x$. In order to show that every element $g$ in $G$ belongs to $H$, we may assume that $g$ is not already in $\Stab_G(x)\subset H$, so $x\ne g\star x$. Thus, both pairs $(x, g\star x)$ and $(x, h\star x)$ belong to $X^{(2)}$,  so $2$-transitivity yields some $g_1$ in $\Stab_G(x)$ (and thus in $H$) with $(g_1\cdot g)\star x=h\star x$, that is, the element 
$h\inv\cdot g_1\cdot g$ belongs to $\Stab_G(x)\subset H$. We deduce that  $g$ lies in $H$, as desired. 

If $G$ is infinite and the action of $G$ on $X$ is faithful and $2$-transitive, then is $G$ centerless, and in particular non-abelian. Indeed, 
assume for a contradiction that $\ZZ(G)\ne\{1_G\}$. All stabilizers are conjugate to each other, so we deduce that $\ZZ(G)\not\subset \Stab_G(x)$ for some (equivalently, any) $x$ in $X$, since the transitive action is faithful and $\ZZ(G)$ is normal.  Hence, the subgroup $\Stab_G(x)\cdot \ZZ(G)$ must equal $G$, by maximality of the stabilizer.  It is easy to see that the induced action of $\ZZ(G)$ on $X$ is transitive (since $X$ is the $G$-orbit of the base-point $x$) and regular (as the center $\ZZ(G)$ commutes with every element in $G$). The bijection between $X$ and $\ZZ(G)$ induces an isomorphism of actions between the induced action of $\Stab_G(x)$ on $X$ and the action of $\Stab_G(x)$ on $\ZZ(G)$ by conjugation.  The latter action is trivial, so the stabilizer $\Stab_G(x)$ acts trivially on $X$. Since the action of $G$ on $X$ is faithful, we conclude that the maximal subgroup $\Stab_G(x)$ must be trivial group, which contradicts that $G$ is infinite, as desired.
\end{rema} 
From now on, we will assume additionally that the algebraically closed base subfield over which the equation is defined is contained in the field of constants $\CC_\UU$. 
\begin{prop}\label{P:D2_orthogonal_constants}
Consider an irreducible differential algebraic  equation $P(T)=0$ of order $n\ge 1$ having Property D$_2$ and defined over an algebraically closed differential subfield $K$ contained in  $\CC_\UU$. We have that the generic type of the equation over $K$ is orthogonal to the type of the constants.
\end{prop}
\begin{proof}
Suppose that the equation is defined over $K=K^{alg}\subset \CC_\UU$. Assume for a contradiction that the generic type over $K$ is non-orthogonal to the constants.  Corollaries~\ref{C:D2_nonorth} and~\ref{C:D2_weakorth} yield that the generic type is  internal to the constants and weakly orthogonal to the constants. Proposition~\ref{P:weak_orth_transitive} yields some \emph{isolating} Kolchin constructible set $X$ defined over $K$ such that a solution of the equation $a$ is generic over $K$ if and only if $a$ belongs to $X$.  

The binding group $\Bind(a/K)$ acts on $X$ and also on $X^{(2)}$, by definition. Lemma~\ref{L:D2_fibration} yields that the Kolchin constructible set $X^{(2)}$ isolates the type of a Morley sequence of length $2$ of realizations of the generic type. In particular, the type of a Morley sequence of length $2$ is again isolated and the action of its binding group is (isomorphic to) the diagonal action of $\Bind(a/K)$ on $X^{(2)}$. We deduce from Proposition~\ref{P:weak_orth_transitive}  that the action of $\Bind(a/K)$ on $X$ is $2$-transitive. In particular, it follows from Corollary~\ref{C:nonalg_infinitebinding} (since $n\ge 1$ )and Remark~\ref{R:2trans} that the binding group $\Bind(a/K)$ is centerless. 

By Lemma~\ref{L:fundsys_prime}, there exists a a fundamental system of solutions $\bar d$ of the type of $a$ over $K$ such that $\bar d$ belongs to a differential closure $\widehat{K}$ of $K$ and the action of the centerless group $\Bind(a/K)$ on the set of realizations of the type of $\bar d$ over $K$ coincides with the action of the binding group $\Bind(\bar d/K)$. Now, Fact~\ref{F:Kolchin_Galois} implies that $\Bind(\bar d/K)$ is isomorphic to the set of $\CC_\UU$-rational points of a connected algebraic group $G$ defined over $\CC_\UU$. The latter must be in particular centerless (since $\Bind(a/K)$ is). As shown by \textcite[Theorem 13]{Rosenlicht_alggp}, the algebraic group $G$ must be linear.  Since the type of $\bar d$ over $K$ is fundamental, the differential field extension $K\subset \str{K}{\bar d}$ is strongly normal, as in Fact~\ref{F:Kolchin_Galois}. In particular,  the differential field extension $K\subset \str{K}{\bar d}$ must be a Picard--Vessiot extension by Fact~\ref{F:linear_PV}. Now, the subfield $K$ of $\CC_\UU$ is algebraically closed, so Remark~\ref{R:PV_commutative} yields that $\Bind(\bar d/K)$  must be commutative.  However, our assumption yields that  $\Bind(\bar d/K)$  is centerless, so $\Bind(\bar d/K)$, and thus $\Bind(a/K)$, is the trivial group, contradicting  Corollary~\ref{C:nonalg_infinitebinding}. 
\end{proof}

We have now all the ingredients to prove one of the main theorems of the work of \textcite[Theorem 3.9]{FJM22}:
\begin{theo}\label{T:Main}
Consider an irreducible  differential algebraic equation $P(T)=0$ of order $n\ge 1$ defined over a countable algebraically closed subfield $K$ of the constant field $\CC_\UU$ of the ambient differentially closed field $\UU$. If the  equation has Property D$_2$ over $K$, then its generic type over $K$ is minimal and trivial. In particular, it has Property D$_m$ for all $m\ge 2$, that is, given $m\ge 2$  pairwise distinct solutions $a_1,\ldots, a_m$ over $K$, each non-algebraic over $K$,  we have that the set \[ a_1,\delta(a_1), \ldots, \delta^{n-1}(a_1), \ldots, a_m,\delta(a_m), \ldots, \delta^{n-1}(a_m)\] forms an algebraically independent family over $K$. 
\end{theo}
\begin{proof}
Consider a generic solution $a$ of equation $P(T)=0$ over $K$. The generic type of $a$ over $K$ is stationary, since $K$ is algebraically closed, 
and of finite rank, by Corollary~\ref{C:diffrank_trdeg}. Now, it follows from Proposition~\ref{P:nonorth_internal}.(b) and (c) together with Proposition~\ref{P:D2_orthogonal_constants}  that there exists some finite tuple $\bar d$ algebraic over $\str{K}{a}$ such that the stationary type of $\bar d$ over $K^{alg}=K$ is minimal and $1$-based. 

Lemma~\ref{L:D2_fibration} implies that  $a$ belongs to $\str{K}{\bar d}$.  An easy application of Lascar's inequalities yields that \[ \U(a/K)\stackrel{\ref{R:fterank}.(b)}{=}\U(a/ \str{K}{\bar d}) + \U(\bar d/K)=0+ \U(\bar d/K)=0+1=1.\] Hence,  the type of $a$ over $K$ is minimal. We deduce from Corollary~\ref{C:descent} and Proposition~\ref{P:D2_orthogonal_constants} that the generic  type of $a$ over $K$ must be trivial, as desired. 

Let us now show that Property D$_m$ holds:  Consider $m$  pairwise distinct solutions $a_1,\ldots, a_m$, each non-algebraic over $K$. By Remark~\ref{R:trivial} together with Property D$_1$, the independences \[ \str{K}{a_i}\ind_K \str{K}{a_1,\ldots, a_{i-1}} \ \text{ for } 1\le i\le m\] yield  that
\[ \tr_K\left(a_1,\delta(a_1), \ldots, \delta^{n-1}(a_1), \ldots, a_m,\delta(a_m), \ldots, \delta^{n-1}(a_m)\right)=m\cdot n,\] so 
   the set \[ a_1,\delta(a_1), \ldots, \delta^{n-1}(a_1), \ldots, a_m,\delta(a_m), \ldots, \delta^{n-1}(a_m) \] forms an algebraically independent family over $K$, as desired.  
\end{proof}
\begin{rema}\label{R:D_2notD3}
If our irreducible differential algebraic equation  is not defined over the constants, Property D$_3$ need not follow from D$_2$. Indeed, \textcite[Subsection 4.2, $n=3$ \& Claim 4.5]{FreitagMoosa} have an example of a linear differential system $\delta(\bar Y)=B\cdot \bar Y$ with Property D$_2$ (since the action of the binding group is $2$-transitive) for which Property D$_3$ fails. The matrix $B$ has  independent entries consisting of differentially transcendental elements contained in some base field $K \not\subset \CC_\UU$. 

Using the differential version of the primitive element theorem, shown by \textcite[p.  728 with $m=1$]{Kolchin_primitive}, the above system translates to an irreducible differential algebraic equation, given by a rational function over $K$, such that every generic solution of the equation is interalgebraic with a generic solution of the above linear system. Hence, the resulting differential algebraic equation has Property D$_2$ and not  Property D$_3$.  
\end{rema}

Whilst the fact that the base subfield $K$ was contained in the field of constants $\CC_\UU$ was crucial for the above proof, relevant information propagates to other differential field extensions of $K$, even if these are not subfields of $\CC_\UU$, 
once we know that the generic type is minimal and trivial. 
\begin{coro}\label{C:sm}
If an irreducible differential algebraic  equation $P(T)=0$ of order $n\ge 1$ has Property D$_2$ and is defined over a countable algebraically closed subfield $K$ of $\CC_\UU$, then the equation is \emph{strongly minimal}: For every differential field extension $L$ of $K$ and every Kolchin constructible subset $X$ of $\UU$  defined over $L$, exactly one of the two Kolchin constructible sets  \[ \{a\in \UU \ | \ P(a)=0\ \&\ a\in X \} \text{ or } \{a\in \UU \ | \ P(a)=0\ \&\ a\notin X \}  \]
is finite. 

In particular, any two solutions of $P(T)=0$, none of which is algebraic over $L$, have the same type over $L$. We refer to this type as the \emph{generic type} of the equation over $L$. 
\end{coro}

\begin{proof}
Assume for a contradiction that there exists some Kolchin constructible subset $X$ of $\CC_\UU$ defined over some differential field extension $L$ of $K$ such that both $X\cap (P(T)=0)$ and $(P(T)=0) \setminus X$ are infinite. 
\begin{claimstar}
There exists two solutions $a$ and $a_1$ of $P(T)=0$, none of which is algebraic over $L$, such that $a$ belongs to $X$ yet $a_1$ does not. 
\end{claimstar}
\begin{claimstarproof}
By $\aleph_1$-saturation of $\UU$, it suffices to show that the following collection of $L$-instances \begin{multline*}  \Sigma(x, x_1)= \{P(x)=0=P(x_1)\}\cup\{x\in X \land x_1\notin X\} \cup \\ \{Q(x)\cdot Q(x_1)\ne 0 \ | \ 0\ne Q(T) \in L[T]\}\end{multline*} is a partial $2$-type over $L$, where we use the notation $``x \in X"$ as an abbreviation for the corresponding $L$-instance defining $X$. Indeed, a finite collection of instances in $\Sigma$ involves only finitely many non-trivial polynomials $Q_1,\ldots, Q_m$ over $L$. Now, both the sets $(P(T)=0)\cap X$ and $(P(T)=0)\setminus X$ are infinite, so none of them is contained in the finite set of roots of the $Q_i$'s. Hence, choosing $b$ in $(P(T)=0)\cap X$ and $b_1$ in $(P(T)=0)\setminus X$, both different from all possible roots of the $Q_i$'s,  we obtain a realization $(b, b_1)$ of the finite collection of $L$-instances of $\Sigma$, as desired. 
\end{claimstarproof}

Now, neither $a$ nor $a_1$ as in the Claim belong to the algebraically closed subfield $K$, so they both realize the (stationary) generic type of the equation over $K$.  This type is minimal by Theorem~\ref{T:Main}, so we we deduce both independences \[  \str{K}{a}\ind_K L \text{ and }  \str{K}{a_1}\ind_K L,\]  since neither $a$ nor $a_1$ are algebraic over $L$.

It follows directly from {\bf Stationarity} that $a$ and $a_1$ must have the same type over $L$. However, the solution $a$ belongs to the Kolchin constructible subset $X$ defined over $L$, yet $a_1$ does not, contradicting the equivalence in Definition~\ref{D:types}, as desired. 
\end{proof}
We can now show below that Property D$_m$  holds  over any possible field extension $L$ in $\UU$ of the base subfield $K$ of $\CC_\UU$.
\begin{coro}\label{C:D2_otherparam}
Consider a countable algebraically closed subfield $K\subset \CC_\UU$ of the ambient differentially closed field $\UU$ and an irreducible differential algebraic  equation $P(T)=0$ of order $n\ge 1$ having Property D$_2$ and defined over $K$. 

Given a differential field extension $L$ of $K$, the equation $P(T)=0$   has Property D$_m$ over $L$,  that is, whenever the pairwise distinct  solutions $a_1,\ldots, a_m$ of the equation are each non-algebraic over $L$, then the set \[ a_1,\delta(a_1), \ldots, \delta^{n-1}(a_1), \ldots, a_m,\delta(a_m), \ldots, \delta^{n-1}(a_m) \] forms an algebraically independent family over $L$.  
\end{coro}

\begin{proof}
	
	Let $L$ and $a_1,\ldots, a_m$  be given as in the statement. Without loss of generality, we may assume that $L=L^{alg}$. By Corollary~\ref{C:sm}, we know that each $a_i$ realizes the unique independent extension of the generic type to  the field $L$,  since  none of them is algebraic over $L$. By Theorem~\ref{T:Main}, the generic type of the equation over $K$ is trivial. Thus,  the statement follows from Remark~\ref{R:trivial} as in the last part of the proof of~\ref{T:Main}, once we show that any two distinct  solutions, none of which is algebraic over $L$, are independent over $L$. 

Relabeling, consider two solutions $a_1\ne a_2$, none of them algebraic over $L$. In order to show that \[ \tag{$\star$} \str{L}{a_1}\ind_L  \str{L}{a_2},\] it suffices to show that \[ \tag{$\lozenge$} \str{K}{a_1,a_2}\ind_K L.\]  Indeed, Property D$_2$ yields that 
\[\tag{$\blacklozenge$}  \str{K}{a_1} \ind_{K} \str{K}{a_2},\] so the  independence ($\star$) follows from ($\lozenge$) and ($\blacklozenge$), by  {\bf Monotonicity \& Transitivity}.

The independence ($\lozenge$) is equivalent to  showing that the canonical base $\cb(a_1,a_2/L)$ is contained in $K^{alg}\subset L^{alg}=L$, by Proposition~\ref{P:CB}.(b). Choose   some  finite tuple $\bar b$ with $\str{\Q}{\bar b}=\cb(a_1,a_2/L^{alg})$, by Proposition~\ref{P:CB}.(d). By construction of the canonical base, the differential field  $\str{\Q}{\bar b, a_1, a_2}$ is linearly disjoint from $L$ over $\str{\Q}{\bar b}$. In particular, setting $k_1=\str{K}{\bar b}$, we have that $\str{k_1}{a_1, a_2}$ and $L^{alg}$ are linearly disjoint (and thus independent) over $k_1$.

By Lemma~\ref{L:trivial_1based}, the stationary type of $(a_1, a_2)$ over $k_1$ is $1$-based, since the type of $a_1$ over $K$ is trivial 
(and equals the type of $a_2$ over $K$). 
Hence, the differential subfield $k_1$ is contained in $\str{K}{a_1, a_2}^{alg}$.  By {\bf Extension}, there is some  realization $(a_1', a_2')$ of the stationary type of $(a_1, a_2)$ over $k_1$ with \[ \str{k_1}{a'_1, a'_2} \ind_{k_1} \str{k_1}{a_1, a_2},\]  so  $k_1$ is contained in $\str{K}{a'_1, a'_2}^{alg}$,  by construction.

If we show that \[ \str{K}{a'_1, a'_2} \ind_{K} \str{K}{a_1, a_2},\] then {\bf Algebraic Closure} yields immediately that $k_1=K$, and thus $\cb(a_1,a_2/L^{alg})$ belongs to $K$, as desired.  

Now, neither $a_i$ nor $ a'_j$ are algebraic over $k_1\subset L^{alg}$, so  the four solutions $a_1, a_2, a'_1$ and $a'_2$ are all distinct and generic over $K$.  Theorem~\ref{T:Main} and Property D$_4$ imply in particular that \[ \str{K}{a'_1, a'_2} \ind_{K} \str{K}{a_1, a_2},\] which yields the  desired result. 
\end{proof}

We will conclude by showing that (the theory, with the induced structure, of the set of generic solutions of) an equation with Property D$_2$ defined over the field of constants $\CC_\UU$ is \emph{$\aleph_0$-categorical}, as defined below, using one of the equivalent versions of $\aleph_0$-categoricity in Ryll-Nardzewski's theorem \textcite[Section \S 4.3]{TentZiegler}

\begin{coro}\label{C:Ryll}
Consider a countable algebraically closed subfield $K\subset \CC_\UU$ of the ambient differentially closed field $\UU$ and an irreducible differential algebraic equation $P(T)=0$ of order $n\ge 1$ having Property D$_2$ and defined over $K$. The set \[ X=\{ a\in \UU \ | \ P(a)=0 \text{ with $a$ generic over $K$}\} \] is $\aleph_0$\emph{-categorical}: for every natural number $n\ge 1$, the diagonal action of $\Aut_\delta(\UU/K)$ on $X^n$ is \emph{oligomorphic}, that is, there are only finitely many orbits.  
\end{coro}
\begin{proof}
We prove the statement by induction on $n\ge 1$. For $n=1$, there is a single orbit by Corollary \ref{C:univ} and Remark \ref{R:D2_generictype}. Assume now that the diagonal action of $\Aut_\delta(\UU/K)$ on $X^n$. The orbit of an $n+1$-tuple $(a_1,\ldots, a_{n+1})$ projects surjectively to the orbit of the $n$-tuple  $(a_1,\ldots, a_n)$. Moreover, the fiber of the projection is in bijection with all possible images of $a_{n+1}$ under an automorphism $\sigma$ in $\Aut_\delta(\UU/\str{K}{a_1,\ldots, a_n})$. Thus, it suffices to show that there are only $n+1$ many possible orbits of a singleton $x$ in $X$ under automorphism in $\Aut_\delta(\UU/\str{K}{a_1,\ldots, a_n})$. Indeed, if $x$ equals one of the $a_i$'s, its orbit is uniquely determined and consists of the singleton $\{a_i\}$. 

Hence, we may assume that none of the elements $x$ and $x_1$ in $X$ equal one of the $a_i$'s. By Theorem \ref{T:Main}, we have that $\str{K}{x}$, resp. $\str{K}{x_1}$, is algebraically independent from $\str{K}{a_1,\ldots, a_n}$ over $K$. By Corollary \ref{C:sm}, we deduce that $x$ and $x_1$ lie in the same $\Aut_\delta(\UU/\str{K}{a_1,\ldots, a_n})$-orbit, as desired. 
\end{proof}


\printshorthands 

\printbibliography

\end{document}
